\documentclass[a4paper, pdftex]{article}

\usepackage[utf8]{inputenc}
\usepackage[T1]{fontenc}
\usepackage[pdftex]{graphicx}
\usepackage{amsmath}
\usepackage{amsfonts}
\usepackage{amsthm}
\usepackage{amssymb}
\usepackage{setspace}
\usepackage[format=hang]{caption}
\usepackage{thmtools}
\usepackage{thm-restate}
\usepackage{mathtools}
\usepackage[colorlinks,citecolor=blue]{hyperref}
\usepackage[capitalize]{cleveref}
\usepackage{verbatim}
\usepackage[all, cmtip]{xy}

\begin{document}
\newtheorem*{thm*}{Theorem}
\newtheorem{theorem}{Theorem}[section]
\newtheorem{corollary}[theorem]{Corollary}
\newtheorem{lemma}[theorem]{Lemma}
\newtheorem{fact}[theorem]{Fact}
\newtheorem*{fact*}{Fact}
\newtheorem{proposition}[theorem]{Proposition}
\newcounter{theoremalph}
\renewcommand{\thetheoremalph}{\Alph{theoremalph}}
\newtheorem{thmAlph}[theoremalph]{Theorem}
\theoremstyle{definition}
\newtheorem*{question}{Question}
\newtheorem{definition}[theorem]{Definition}
\theoremstyle{remark}
\newtheorem{remark}[theorem]{Remark}
\newtheorem{example}[theorem]{Example}

\newcommand{\overbar}[1]{\mkern 2mu\overline{\mkern-2mu#1\mkern-2mu}\mkern 2mu}

\newcommand{\set}[1]{\ensuremath{ \left\lbrace #1 \right\rbrace}}
\newcommand{\bsl}{\ensuremath{\setminus}}
\newcommand{\grep}[2]{\ensuremath{\left\langle #1 \, \middle| \, #2\right\rangle}}
\renewcommand{\ll}{\left\langle}
\newcommand{\rr}{\right\rangle}

\newcommand{\rank}{n}
\newcommand{\Aut}[1]{\ensuremath{\operatorname{Aut}(#1)}}
\newcommand{\AutO}{\ensuremath{\operatorname{Aut}^0}}
\newcommand{\GL}[2]{\ensuremath{\operatorname{GL}_{#1}(#2)}}
\newcommand{\Out}[1]{\ensuremath{\operatorname{Out}(#1)}}
\newcommand{\Outo}[1][A_\Gamma]{\ensuremath{\operatorname{Out}^0(#1)}}
\newcommand{\Roar}[2]{\ensuremath{\operatorname{Out}^0(#1;\, #2)}}
\newcommand{\Stab}{\operatorname{Stab}}
\newcommand{\op}{\operatorname{op}}
\newcommand{\mcA}{\ensuremath{\mathcal{A}}}
\newcommand{\mcC}{\ensuremath{\mathcal{C}}}
\newcommand{\mcD}{\ensuremath{\mathcal{D}}}
\newcommand{\mcE}{\ensuremath{\mathcal{E}}}
\newcommand{\mcF}{\ensuremath{\mathcal{F}}}
\newcommand{\mcG}{\ensuremath{\mathcal{G}}}
\newcommand{\mcH}{\ensuremath{\mathcal{H}}}
\newcommand{\mcI}{\ensuremath{\mathcal{I}}}
\newcommand{\mcJ}{\ensuremath{\mathcal{J}}}
\newcommand{\mcN}{\ensuremath{\mathcal{N}}}
\newcommand{\mcP}{\ensuremath{\mathcal{P}}}
\newcommand{\mcQ}{\ensuremath{\mathcal{Q}}}
\newcommand{\mcR}{\ensuremath{\mathcal{R}}}
\newcommand{\mcS}{\ensuremath{\mathcal{S}}}
\newcommand{\mcU}{\ensuremath{\mathcal{U}}}
\newcommand{\mcV}{\ensuremath{\mathcal{V}}}
\newcommand{\mcW}{\ensuremath{\mathcal{W}}}
\newcommand{\mcZ}{\ensuremath{\mathcal{Z}}}
\newcommand{\FC}{\ensuremath{\operatorname{FC}_{\rank}}}
\newcommand{\CC}[2]{\ensuremath{\operatorname{CC}( #1 ,\, #2)}}
\newcommand{\on}[1]{\operatorname{#1}}
\newcommand{\CosPos}[2]{\ensuremath{\operatorname{CP}( #1 ,\, #2)}}
\newcommand{\st}{\operatorname{st}}
\newcommand{\lk}{\operatorname{lk}}
\newcommand{\im}{\operatorname{im}}
\newcommand{\rk}{\operatorname{rk}}
\newcommand{\corank}{\operatorname{crk}}
\newcommand{\Sym}{\operatorname{Sym}}
\newcommand{\AG}{A_\Gamma}

\newcommand{\FreeS}[1][n]{\mathcal{FS}_{#1}}
\newcommand{\FreeOne}{\mathcal{FS}^1}
\newcommand{\crX}{\operatorname{X}}
\newcommand{\crCore}{\operatorname{C}}

\newcommand{\real}[1]{\ensuremath{\left\lVert #1\right\rVert}}
\newcommand{\FRfactor}{\mathcal{F}}
\newcommand{\RAAGcomplex}{\mathcal{CC}}

\title{Between buildings and free factor complexes: A Cohen--Macaulay complex for Out(RAAGs)}
\author{Benjamin Br\"{u}ck}
\maketitle

\begin{abstract}
For every finite graph $\Gamma$, we define a simplicial complex associated to the outer automorphism group of the RAAG $A_\Gamma$. These complexes are defined as coset complexes of parabolic subgroups of $\Outo$ and interpolate between Tits buildings and free factor complexes. We show that each of these complexes is homotopy Cohen--Macaulay and in particular homotopy equivalent to a wedge of $d$-spheres. The dimension $d$ can be read off from the defining graph $\Gamma$ and is determined by the rank of a certain Coxeter subgroup of $\Outo$. In order to show this, we refine the decomposition sequence for $\Outo$ established by Day--Wade, generalise a result of Brown concerning the behaviour of coset posets under short exact sequences and determine the homotopy type of free factor complexes associated to relative automorphism groups of free products. 
\makeatletter{\renewcommand*{\@makefnmark}{}\footnotetext{MSC classes: 20F65, 20F28, 20E42, 20F36, 57M07.}\makeatother}
\end{abstract}

\section{Introduction}
Given a simplicial graph $\Gamma$, the associated right-angled Artin group (RAAG) $A_\Gamma$ is the group generated by the vertex set of $\Gamma$ subject to the relations $[v,w]=1$ whenever $v$ and $w$ are adjacent. If $\Gamma$ is a discrete graph (no edges),  $A_\Gamma$ is a free group whereas if $\Gamma$ is complete, the corresponding RAAG is a free abelian group; often, the family of RAAGs is seen as an interpolation between these two extremal cases.
Accordingly, the outer automorphism group $\Out{A_\Gamma}$ is often seen as interpolating between the arithmetic group $\GL{n}{\mathbb{Z}}=\Out{\mathbb{Z}^n}$ and $\Out{F_n}$. Over the last years, this point of view has served both as a motivation for studying automorphism groups of RAAGs and as a source of techniques improving our understanding of these groups.

This article contributes to this programme by providing a new geometric structure generalising well-studied complexes associated to arithmetic groups and automorphism groups of free groups. On the arithmetic side, we have the Tits building associated to $\GL{n}{\mathbb{Q}}$. It can be defined as the order complex of the poset of proper subspaces of $\mathbb{Q}^n$, ordered by inclusion, and is homotopy equivalent to a wedge of $(n-2)$-spheres (this is the Solomon--Tits Theorem). On the $\Out{F_n}$ side, there is the free factor complex, which is defined as the order complex of the poset of conjugacy classes of proper free factors of $F_n$, ordered by inclusion of representatives. This complex is homotopy equivalent to a wedge of $(n-2)$-spheres as well (see the work of Hatcher--Vogtmann \cite{HV:complexfreefactors} and of Gupta and the author \cite{BG:Homotopytypecomplex}). In this article, we construct a (simplicial) complex interpolating between these two structures.

It should be mentioned that the free factor complex was introduced in \cite{HV:complexfreefactors} in order to obtain an analogue of the more classical Tits building for the setting of $\Out{F_n}$. The same is true for the curve complex $\mcC(S)$ associated to a surface $S$, which was defined by Harvey \cite{Har:Boundarystructuremodular} and shown to be homotopy equivalent to a wedge of spheres by Harer \cite{Har:virtualcohomologicaldimension} and Ivanov \cite{Iva:ComplexescurvesTeichmullerb} --- it is an analogue of a Tits building for the setting of mapping class groups. In this sense, the complex we construct can also be seen as an $\Out{\AG}$-analogue of $\mcC(S)$.

Instead of looking at $\Out{\AG}$ itself, we will throughout work with its finite index subgroup $\Outo$, called the \emph{pure outer automorphism group}. The group $\Outo$ was first defined by Charney, Crisp and Vogtmann in \cite{CCV:Automorphisms2dimensional} and has since become popular as it avoids certain technical difficulties coming from automorphisms of the graph $\Gamma$; if $\AG$ is free or free abelian, we have $\Outo = \nolinebreak\Out{A_\Gamma}$. From now on, let $O\coloneqq \Outo$.

Above, we described the building associated to $\GL{n}{\mathbb{Q}}$ in terms of flags of subspaces of $\mathbb{Q}^n$. However, one can also describe the building associated to a group $G$ with BN-pair in a more intrinsic way using the parabolic subgroups of $G$. We use this definition as an inspiration for our construction: Given $O$, we define a family of \emph{maximal standard parabolic subgroups} $\mcP(O)$. Every element of $\mcP(O)$ is a proper, non-trivial subgroup of $O$ and defined as the stabiliser $\Stab_O(A_\Delta)$ of some special subgroup $A_\Delta\leq \AG$ (for the precise statement, see \cref{definition parabolics for RAAGs}). 
The complex we consider now is the \emph{coset complex of $O$ with respect to the family $\mcP(O)$}, i.e.~the simplicial complex whose vertices are given by cosets $gP$, with $g\in O$ and $P\in \mcP(O)$, and where a collection of such cosets forms a simplex if and only if its intersection is non-empty. This complex is denoted by $\CC{O}{\mcP(O)}$ and $O$ acts on it by left multiplication. 
The main result of this paper is the following:

\begin{thmAlph}
\label{introduction homotopy type}
The complex $\RAAGcomplex \coloneqq \CC{O}{\mcP(O)}$ is homotopy equivalent to a wedge of spheres of dimension $|\mcP(O)|-1$, where we call $\rk(O)\coloneqq |\mcP(O)|$ the \emph{rank} of $O$.
\end{thmAlph}
The rank $\rk(O)$ seems to be an interesting invariant of the group $O$, which has, to the best of the author's knowledge, not been studied in the literature so far.

In order to prove Theorem \ref{introduction homotopy type}, we are lead to study relative versions of it, namely we have to consider the case where $O$ is not given by all of $\Outo$, but rather by a \emph{relative outer automorphism group} $O=\Roar{\AG}{\mcG, \mcH^t}$ as defined by Day and Wade \cite{DW:Relativeautomorphismgroups} (for the definitions, see \cref{sec RAAGs and their automorphism groups} and \cref{sec generators of ROARs}). This is why we prove all of the results mentioned in this introduction in that more general setting. In particular, we define a set of maximal standard parabolic subgroups $\mcP(O)$ and the rank $\rk(O)$ for all such $O$.

In addition to Theorem \ref{introduction homotopy type}, we show that $\RAAGcomplex$ has the following properties which indicate that it is a reasonable analogue of Tits buildings and free factor complexes:
\paragraph{Properties of $\RAAGcomplex$} 
\begin{itemize}
\item \emph{Building.} If $O=\GL{n}{\mathbb{Z}}$, the complex  $\RAAGcomplex$ is isomorphic to the building associated to $\GL{n}{\mathbb{Q}}$. (\cref{building and coset complex})
\item \emph{Free factor complex.} If $O=\Out{F_n}$, the complex $\RAAGcomplex$ is isomorphic to the free factor complex associated to $F_n$. (\cref{isomorphism relative free factor complex})
\item \emph{Cohen--Macaulayness.} $\RAAGcomplex$ is homotopy Cohen--Macaulay and in particular a chamber complex. (\cref{CM of coset complex parabolic subgroups})
\item \emph{Facet-transitivity.} Any maximal simplex of $\RAAGcomplex$ forms a fundamental domain for the action of $O$. (\cref{section group actions detecting cc})
\item \emph{Stabilisers.} The vertex stabilisers of this action are exactly the conjugates of the elements of $\mcP(O)$. Stabilisers of higher-dimensional simplices are given by the intersections of such conjugates and can be seen as parabolic subgroups of lower rank. (\cref{section parabolics of lower rank})
\item \emph{Parabolics as relative automorphism groups.} Every maximal standard parabolic $P\in \mcP(O)$ is itself a relative automorphism group of the form $\Roar{\AG}{\mcG, \mcH^t}$ and $\rk(P)=|\mcP(P)|=\rk(O)-1$. (\cref{rank of parabolic subgroups})
\item \emph{Rank via Weyl group.} Similar to a group with BN-pair, the rank $\rk(O)$ is equal to the rank of a naturally defined Coxeter subgroup $\AutO(\Gamma)\leq O$. (\cref{rank via Coxeter system})
\item \emph{Direct and free products.} The construction is well-behaved under taking direct and free products of the underlying RAAGs, i.e. under passing from $\operatorname{Out}^0(A_{\Gamma})$ to $\operatorname{Out}^0(A_{\Gamma}\times A_{\Gamma'})$ or to $\operatorname{Out}^0(A_{\Gamma}\ast A_{\Gamma'})$. (\cref{section constructions}) 
\end{itemize}

\begin{remark}
The present article focuses on topological properties of $\RAAGcomplex$. This is the perspective that makes the complexes it generalises look quite similar, reflecting the fact that $\Out{F_n}$, $\GL{n}{\mathbb{Z}}$ and mapping class groups of surfaces share many homological properties (see e.g.~\cite{CFP:stabilityconjectureunstable} and references therein). It should be noted that from a geometric perspective, there are significant differences between the associated complexes: While the free factor complex \cite{BF:Hyperbolicitycomplexfree} and the curve complex \cite{MM:Geometrycomplexcurves.a} are hyperbolic, spherical buildings such as the one associated to $\GL{n}{\mathbb{Z}}$ have finite diameter. In general, $\Out{\AG}$ has many aspects of $\GL{n}{\mathbb{Z}}$, so one should probably not expect associated complexes to have a ``purely hyperbolic'' flavour (cf.~the work of Haettel \cite{Hae:Hyperbolicrigidityhigher}).
\end{remark}

As an application of Theorem \ref{introduction homotopy type} and the results of \cite{Bru:Highergeneratingsubgroups}, we obtain higher generating families of subgroups of $O$ in the sense of Abels--Holz (for the definition, see \cref{sec higher generation}):
\begin{thmAlph}
\label{introduction higher generation by parabolic subgroups}
The family $\mcP_m(O)$ of rank-$m$ parabolic subgroups of $O$ is $m$-generating.
\end{thmAlph}
Here, an element of $\mcP_m(O)$ is given by the intersection of $\rk (O) - m$ distinct elements from $\mcP(O)$ (see \cref{section parabolics of lower rank}).
Higher generation can be interpreted as an answer to the question ``How much information about $O$ is contained in the family of subgroups $\mcP_m(O)$?'' --- as an immediate consequence of this theorem, we are able to give for every $2\leq m \leq \rk(O)-1$ a presentation of $O$ in terms of the rank-$m$ parabolic subgroups (\cref{presentation by parabolic subgroups}).

The main ingredient in our proof of Theorem \ref{introduction homotopy type} is an inductive procedure developed by Day--Wade. In \cite{DW:Relativeautomorphismgroups}, they show how to decompose $O$ using short exact sequences into basic building blocks which consist of free abelian groups, $\GL{n}{\mathbb{Z}}$ and so-called Fouxe-Rabinovitch groups, which are groups of certain automorphisms of free products. We refine their induction in order to get a better control on the induction steps that are needed and to get a more explicit description of the resulting base cases. An overview of this can be found in \cref{section summary inductive procedure}.
In order to make use of this inductive procedure, we establish a theorem regarding the behaviour of coset posets and complexes under short exact sequences. This generalises results of Brown \cite{Bro:cosetposetprobabilistic}, Holz \cite{Hol:EndlicheIdentifizierbarkeitvon} and Welsch \cite{Wel:cosetposetsnerve} and can be phrased as follows:

\begin{thmAlph}
\label{introduction coset complexes and SES}
Let $G$ be a group, $\mcH$ a family of subgroups of $G$ and $N\triangleleft G$ a normal subgroup. If $\mcH$ is strongly divided by $N$, there is a homotopy equivalence
\begin{equation*}
\CC{G}{\mcH}\simeq \CC{G/N}{\overbar{\mcH}}\ast \CC{N}{\mcH\cap N}.
\end{equation*}
\end{thmAlph}
Here, $\ast$ denotes the join on geometric realisations, $\overbar{\mcH}$ and $\mcH\cap N$ are certain families of subgroups of $G/N$ and $N$, respectively, and being \emph{strongly divided by $N$} is a compatibility condition on the family $\mcH$. (For the definitions, see \cref{sec SES and cosets}; for an explicitly stated special case of Theorem \ref{introduction coset complexes and SES} that we will use in this article, see \cref{short exact sequences}.) 

Note that if two spaces $X$ and $Y$ are homotopy equivalent to wedges of spheres, then so is their join $X\ast Y$. Thus, combining Theorem \ref{introduction coset complexes and SES} with the decomposition of Day--Wade, we are able to reduce the question of sphericity of $\RAAGcomplex$ to the cases where $O$ is either isomorphic to $\GL{n}{\mathbb{Z}}$ or a Fouxe-Rabinovitch group. In the former case, sphericity follows from the Solomon--Tits Theorem (\cref{building and coset complex}). In the latter case, we are lead to study relative versions of free factor complexes (see \cref{definition relative free factor}) and show:

\begin{thmAlph}
\label{introduction free factor complex}
Let $A = F_n \ast A_1 \ast \cdots \ast A_k$ be a finitely generated group. Then the complex of free factors of $A$ relative to $\set{A_1, \ldots , A_k}$ is homotopy equivalent to a wedge of spheres of dimension $n-2$.
\end{thmAlph}

In the case where the group $A$ is a RAAG, Theorem \ref{introduction free factor complex} is a special case of Theorem \ref{introduction homotopy type}. However, we prove it without making this assumption by using the techniques of \cite{BG:Homotopytypecomplex}.

\paragraph{Structure of the article}
Many sections of this article can be read independently from  the others.
We start in \cref{sec topology preliminaries} by recalling some well-known results from topology that we will use throughout the paper. \cref{section coset posets and complexes} contains definitions and basic properties of coset complexes and higher generating subgroups as well as the proof of Theorem \ref{introduction coset complexes and SES}; it can be read completely independently from the rest of this text. The reader not so much interested in details about coset complexes might just want to skim \cref{section coset posets and complexes: I} and then have a look at \cref{short exact sequences}, which summarises the results of this section in the way they will be used later on.
In \cref{section building and relative free factor complex}, we give a definition of the building associated to $\GL{n}{\mathbb{Z}}$ and determine the homotopy type of relative free factor complexes (Theorem \ref{introduction free factor complex}). A reader willing to take this on faith may just have a look at the main results of this section, namely \cref{relative free factor complexes} and \cref{CM free factor complex}; the general theory of automorphisms of RAAGs is still not needed for this.
\cref{section ROARs} contains background about (relative) automorphism groups of RAAGs.
\cref{section spherical complex for Out(RAAG)} is in some sense the core of this article: We define (maximal) parabolic subgroups and the rank of $O$ and combine the results of the previous sections in order to prove Theorem \ref{introduction homotopy type}.
In \cref{section summary induction and examples}, we summarise to which extent we can refine the inductive procedure of Day--Wade and give examples of our construction for specific graphs $\Gamma$.
In \cref{section CM and higher generation}, we show Cohen--Macaulayness of $\RAAGcomplex$, define parabolic subgroups of lower rank and prove Theorem~{\ref{introduction higher generation by parabolic subgroups}}. We then show how the dimension of our complex is related to the rank of a Coxeter subgroup of $O$ (see \cref{rank via Coxeter system}). We close with comments about the limitations of our construction and open questions in \cref{section closing comments}.

\paragraph{Acknowledgements} 
I would like to thank my supervisor Kai-Uwe Bux for his support and guidance throughout this project. Furthermore, many thanks are due to Dawid Kielak and Ric Wade for numerous helpful remarks and suggestions at different stages of this project.
I would also like to thank Herbert Abels, Stephan Holz and Yuri Santos Rego for several interesting conversations about coset complexes and Kai-Uwe Bux, Radhika Gupta, Dawid Kielak and Yuri Santos Rego for helpful comments on a first draft of this text. 
Thanks are also due to the anonymous referee for the careful reading and many good suggestions that helped to improve the quality of this text.

The author was supported by the grant BU 1224/2-1 within the Priority Programme 2026 ``Geometry at infinity'' of the German Science Foundation (DFG).

\section{Preliminaries on (poset) topology}
\label{sec topology preliminaries}

\subsection{Posets and their realisations}
Let $P=(P,\leq)$ be a poset (partially ordered set). If $x\in P$, the sets $P_{\leq x}$ and $P_{\geq x}$ are defined by
\begin{align*}
P_{\leq x}\coloneqq\set{y\in P\mid y\leq x},&& P_{\geq x} \coloneqq\set{y\in P\mid y\geq x}.
\end{align*}
Similarly, one defines $P_{< x}$ and $P_{> x}$.
For $x,y\in P$, the \emph{open interval} between $x$ and $y$ is defined as
\begin{equation*}
(x,y)\coloneqq \set{z\in P \mid x<z<y}.
\end{equation*}
A \emph{chain of length $l$} in $P$ is a totally ordered subset $x_0<x_1<\ldots <x_l$.
For each poset $P=(P,\leq)$, one has an associated simplicial complex $\Delta(P)$ called the \emph{order complex} of $P$. Its vertices are the elements of $P$ and higher dimensional simplices are given by the chains of $P$. When we speak about the \emph{realisation of the poset $P$}, we mean the geometric realisations of its order complex and denote this space by $\real{ P } \coloneqq \real{ \Delta(P) }$.
By an abuse of notation, we will attribute topological properties (e.g. homotopy groups and connectivity properties) to a poset when we mean that its realisation has these properties.

The \emph{join} of two posets $P$ and $Q$, denoted $P\ast Q$, is the poset whose elements are given by the disjoint union of $P$ and $Q$ equipped with the ordering extending the orders on $P$ and $Q$ and such that $p< q$ for all $p\in P, \, q\in Q$.
The geometric realisation of the join of $P$ and $Q$ is homeomorphic to the topological join of their geometric realisations:
\begin{equation*}
\real{P\ast Q} \cong \real{P}\ast \real{Q}
\end{equation*}

The \emph{direct product} $P \times Q$ of two posets $P$ and $Q$ is the poset whose underlying set is the Cartesian product $\{(p,q) \mid p \in P,\, q \in Q\}$ and whose order relation is given by
\begin{equation*}
(p,q) \leq_{P \times Q} (p',q') \text{ if } p \leq_P p' \text{ and } q \leq_Q q'. 
\end{equation*}

A map $f\colon P\to Q$ between two posets is called a \emph{poset map} if $x\leq y$ implies $f(x)\leq f(y)$. Such a poset map induces a simplicial map from $\Delta(P)$ to $\Delta(Q)$ and hence a continuous map on the realisations of the posets. It will be denoted by $\real{ f }$ or just by $f$ if what is meant is clear from the context.

\subsection{Fibre theorems}

An important tool to study the topology of posets is given by so called fibre lemmas comparing the connectivity properties of posets $P$ and $Q$ by analysing the fibres of a poset map between them. The first such fibre theorem appeared in {\cite[Theorem A]{Qui:HigheralgebraicKa}} and is known as Quillen's fibre lemma:

\begin{lemma}[{\cite[Proposition 1.6]{Qui:Homotopypropertiesposet}}]
\label{Quillen fibre lemma contractibility}
Let $f\colon P\to Q$ be a poset map such that the fibre $f^{-1}(Q_{\leq x})$ is contractible for all $x\in Q$. Then $f$ induces a homotopy equivalence on geometric realisations.
\end{lemma}

The following result shows that if one is given a poset map $f$ such that the fibres have only vanishing homotopy groups up to a certain degree, one can also transfer connectivity results between the domain and the image of $f$. Recall that for $n\in \mathbb{N}$, a space $X$ is \emph{$n$-connected} if $\pi_i(X)=\set{1}$ for all $i\leq n$ and $X$ is $(-1)$-connected if it is non-empty.

\begin{lemma}[{\cite[Proposition 7.6]{Qui:Homotopypropertiesposet}}]
\label{Quillen fibre lemma k-connected}
Let $f\colon P\to Q$ be a poset map such that the fibre $f^{-1}(Q_{\leq x})$ is $n$-connected for all $x\in Q$. Then $P$ is $n$-connected if and only if $Q$ is $n$-connected.
\end{lemma}

For a poset $P=(P,\leq)$, let $P^{op}=(P,\leq_{op})$ be the poset that is defined by $x\leq_{op} y \colon \Leftrightarrow y\leq x$.
Using the fact that one has a natural identification $\Delta(P)\cong \Delta(P^{op})$, one can draw the same conclusion as in the previous lemmas if one shows that $f^{-1}(Q_{\geq x})$ is contractible or $n$-connected, respectively, for all $x\in Q$.

Another standard tool which is helpful for studying the topology of posets is:
\begin{lemma}[{\cite[1.3]{Qui:Homotopypropertiesposet}}]
\label{Quillen homotopic poset maps}
If $f,g\colon P\to Q$ are poset maps that satisfy $f(x)\leq \nolinebreak g(x)$ for all $x\in P$, then they induce homotopic maps on geometric realisations.
\end{lemma}

Later on, we will mostly use the following consequence of this lemma.
\begin{corollary}
\label{homotopic poset maps}
Let $P'$ be a subposet of $P$ and $f\colon P\to P'$ a  poset map such that $f|_{P'}= \operatorname{id}_{P'}$. If $f$ is \emph{monotone}, i.e. $f(x)\leq x$ for all $x\in P$ or $f(x)\geq x$ for all $x\in P$, then it defines a deformation retraction $\real{P}\to \real{P'}$.
\end{corollary}
\begin{proof}
Without loss of generality, assume that $f(x)\leq x$ for all $x\in P$.
Let $i\colon P'\hookrightarrow P$ denote the inclusion map. Then for all $x\in P$, we have $i \circ f (x) \leq x$, so by \cref{Quillen homotopic poset maps}, this composition is homotopic to the identity. As $f \circ i = \operatorname{id}_{P'}$, the inclusion $i$ is a homotopy equivalence and the claim follows from \cite[Proposition 0.19]{Hat:Algebraictopology}.
\end{proof}

\subsection{Spherical complexes and their joins}
\label{section spherical complexes}
Recall that a topological space is \emph{$n$-spherical} if it is homotopy equivalent to a wedge of $n$-spheres; as a convention, we consider a contractible space to be homotopy equivalent to a (trivial) wedge of $n$-spheres for all $n$ and the empty set to be $(-1)$-spherical. 
By the Whitehead theorem, an $n$-dimensional CW-complex is $n$-spherical if and only if it is $(n-1)$-connected. Furthermore, sphericity is preserved under taking joins:

\begin{lemma}
\label{homotopy type of wedges and joins}
Let $X$ and $Y$ be CW-complexes such that $X$ is $n$-spherical and $Y$ is $m$-spherical. Then the join $X \ast Y$ is $(n+m+1)$-spherical.
\end{lemma}

\subsection{The Cohen--Macaulay property}
\begin{definition}
Let $X$ be a simplicial complex of dimension $d<\infty$. Then $X$ is \emph{homotopy Cohen-Macaulay} if it is $(d-1)$-connected and the link of every $s$-simplex is $(d-s-2)$-connected.
\end{definition}
The word ``homotopy'' here refers to the original ``homological'' notion of being  ``Cohen--Macaulay over a field $k$''. This  homological condition is weaker than the homotopical one and came up in the study of finite simplicial complexes via their Stanley--Reisner rings (see \cite{Sta:CombinatoricsCommutativeAlgebra:}). For more details on Cohen--Macaulayness and its connections to other combinatorial properties of simplicial complexes, see \cite{Bjo:Topologicalmethods}.

\section{Coset posets and coset complexes}
\label{section coset posets and complexes}

\subsection{Definitions and basic properties}
\label{section coset posets and complexes: I}
\paragraph{Standing assumptions}
Throughout this section, let $G$ be a group and let $\mcH$ be a family of proper subgroups of $G$.

\subsubsection{Background and relation between poset and complex}
\begin{definition}
Let $X$ be a set and $\mathcal{U}$ be a collection of subsets of $X$ such that \mcU\, covers $X$. Then the \emph{nerve} $N(\mathcal{U})$ of the covering $\mcU$ is the simplicial complex that has vertex set \mcU\, and where the vertices $U_0,\ldots ,U_k\in \mathcal{U}$ form a simplex if and only if $U_0\cap\ldots \cap U_k\not = \emptyset$.
\end{definition}

\begin{definition}
Define $\mcU\coloneqq \coprod_{H\in \mcH}G/H$ to be the collection of cosets of the subgroups from $\mcH$.
\begin{enumerate}
\item The \emph{coset poset} $\CosPos{G}{\mcH} \coloneqq (\mcU, \subseteq)$ is the partially ordered set consisting of the elements of $\mcU$ where $g_1 H_1 \leq g_2 H_2$ if and only if $g_1 H_1 \subseteq g_2 H_2$.
\item The \emph{coset complex} $\CC{G}{\mcH}$ is the nerve $N(\mcU)$ of the covering of $G$ given by $\mcU$.
\end{enumerate}
\end{definition}

In this form, coset complexes were introduced by Abels--Holz in \cite{AH:Highergenerationsubgroups} but they appear with different names in several branches of group theory: The main motivation of Abels and Holz was to study finiteness properties of groups. Recent work in this direction can be found in the work of Bux--Fluch--Marschler--Witzel--Zaremsky \cite{BFM+:braidedThompson'sgroups} and Santos-Rego \cite{SReg:finitenesslengthsome}. In \cite{MMV:Highergenerationsubgroup}, Meier--Meinert--VanWyk used these complexes to study the BNS invariants of right-angled Artin groups.
Well-known examples of coset posets are given by Coxeter and Deligne complexes \cite{CD:K1problemhyperplane}. Brown \cite{Bro:cosetposetprobabilistic} studied the coset poset of all subgroups of a finite group and its connection to zeta functions. Generalisations of his work can be found in the articles of Ramras \cite{Ram:Connectivitycosetposet} and  Shareshian--Woodroofe \cite{SW:Ordercomplexescoset}. 
However, the examples that are most important to the present work are given by Tits buildings and free factor complexes (see \cref{section building and relative free factor complex}).

\begin{figure}
\begin{center}
\includegraphics[scale=1]{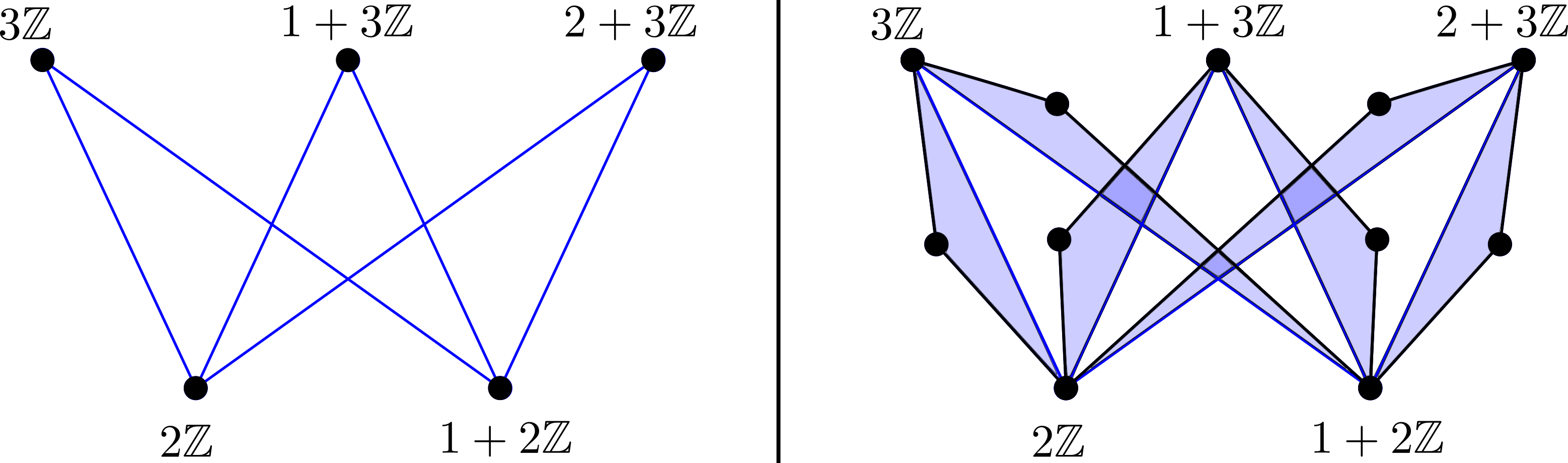}
\end{center}
\caption{The left hand side shows $\CC{\mathbb{Z}}{\mcH}$ and $\CosPos{\mathbb{Z}}{\mcH}$, the right hand side $\CC{\mathbb{Z}}{\widetilde{\mcH}}$ and $\CosPos{\mathbb{Z}}{\widetilde{\mcH}}$, where $\mcH=\set{2\mathbb{Z}, 3\mathbb{Z}}$ and $\widetilde{\mcH}=\set{2\mathbb{Z}, 3\mathbb{Z}, 6\mathbb{Z}}$.
In both pictures, the coset poset is drawn in black and the coset complex is obtained from it by adding the blue parts.}
\label{figure coset complex and poset}
\end{figure}

The order complex of the coset poset $\CosPos{G}{\mcH}$ has the same vertices as the coset complex $\CC{G}{\mcH}$ but the higher-dimensional simplices do not have to agree (see \cref{figure coset complex and poset}).
However, if we assume that $\mcH$ be closed under finite intersections, the topology of these complexes is the same:
\begin{lemma}
\label{homotopy equivalence cc and cospos}
Suppose that $H_1,H_2\in \mcH$ implies $H_1\cap H_2 \in \mcH$. Then $\CC{G}{\mcH}$ deformation retracts to $\CosPos{G}{\mcH}$. In particular, we have
\begin{equation*}
\CosPos{G}{\mcH} \simeq \CC{G}{\mcH}.
\end{equation*}
\end{lemma}
\begin{proof}
As $\mcH$ is closed under intersections,  the intersection of two cosets from $\mcU$ is either empty or also an element of $\mcU$. 
Hence, we can define a map from the poset of simplices of $\CC{G}{\mcH}$ to $\CosPos{G}{\mcH}$ by sending $(g_0 H_0, \ldots, g_k H_k)$ to $\bigcap_i g_i H_i$. 
On the corresponding order complexes, this defines a deformation retraction from the barycentric subdivision of $\CC{G}{\mcH}$ to $\Delta\CosPos{G}{\mcH}$ (see \cite[Theorem 1.4 (b)]{AH:Highergenerationsubgroups}). See \cref{figure coset complex and poset} for an easy example.
\end{proof}

Let $\widetilde{\mcH}$ denote the \emph{family consisting of all finite intersections of elements from $\mcH$}. The following was proved by Holz in \cite{Hol:EndlicheIdentifizierbarkeitvon}.

\begin{lemma}
\label{homotopy equivalence cc closed under intersections}
\leavevmode
\begin{enumerate}
\item Let $\mcH'$ be a family of subgroups of $G$ with $\mcH\subseteq \mcH'$ and such that for all $H'\in \mcH'$, there is $H\in \mcH$ with $H'\subseteq H$. Then there is is a homotopy equivalence $\CC{G}{\mcH}\simeq \CC{G}{\mcH'}$.
\item There is a homotopy equivalence $\CC{G}{\mcH}\simeq \CC{G}{\widetilde{\mcH}}$.
\end{enumerate}
\end{lemma}
\begin{proof}
The nerve $N(\mcU)$ of a collection $\mcU$ of subsets of $X$ is homotopy equivalent to the simplicial complex whose simplices are the non-empty finite subsets of $X$ contained in some $U \in \mcU$ (see \cite[Theorem 1.4 (a)]{AH:Highergenerationsubgroups}). This implies the first claim. The second statement is an immediate consequence of the first one.
\end{proof}

\begin{remark}
The preceding lemmas imply that for any family $\mcH$ of subgroups of $G$, we have
\begin{equation*}
\CC{G}{\mcH}\simeq \CC{G}{\widetilde{\mcH}}\simeq \CosPos{G}{\widetilde{\mcH}}.
\end{equation*}
It follows that we can always replace a coset complex by a coset poset. The advantage of this is that it allows us to apply the tools of poset topology, e.g. the Quillen fibre lemma, to study the topology of these complexes. The trade-off however is that we have to increase the size of our family of subgroups.
\end{remark}

\subsubsection{Higher Generation}
\label{sec higher generation}
We now turn our attention to coset complexes.

\begin{definition}
The \emph{free product of $\mcH$ amalgamated along its intersections} is the group given by the presentation $\ll X \mid R \rr$ where $X= \set{x_g \mid g\in \bigcup \mcH}$ and $R = \set{x_g x_h x_{gh}^{-1} \,\middle| \, \exists H\in \mcH: g,h\in H}$.
\end{definition}

\begin{definition}
 We say that \mcH\, is \emph{$n$-generating} for  $G$ if $\CC{G}{\mcH}$ is $(n-1)$-connected, i.e. $\pi_i(\CC{G}{\mcH})=\set{1}$ for all $i<n$.
\end{definition}

The term ``higher generating subgroups'' was coined by Holz in \cite{Hol:EndlicheIdentifizierbarkeitvon} and is motivated by the following:
\begin{theorem}[{\cite[Theorem 2.4]{AH:Highergenerationsubgroups}}] 
\label{higher generating subgroups by generation}
\leavevmode
\begin{enumerate}
\item $\mcH$ is $1$-generating if and only if $\bigcup \mcH$ generates $G$.
\item $\mcH$ is $2$-generating if and only if $G$ is the free product of $\mcH$ amalgamated along its intersections.
\end{enumerate}
\end{theorem}
Roughly speaking, the latter means that the union of the subgroups in $\mcH$ generates $G$ and that all relations that hold in $G$ follow from relations in these subgroups. The concept of $3$-generation has a similar interpretation using identities among relations (see \cite[2.8]{AH:Highergenerationsubgroups}).

\subsubsection{Group actions and detecting coset complexes}
\label{section group actions detecting cc}
Coset complexes are endowed with a natural action of $G$ given by left multiplication. These complexes are highly symmetric in the sense that this action is \emph{facet transitive}:
Assume that $\mcH$ is finite. Then $\CC{G}{\mcH}$ has dimension $|\mcH|-1$ and $\mcH$ itself is the vertex set of a \emph{facet}, i.e. a maximal simplex, of the coset complex. This (and hence any other) facet is a strict fundamental domain for the action of $G$. The following converse of this observation is due to Zaremsky.

\begin{proposition}[see {\cite[Proposition A.5]{BFM+:braidedThompson'sgroups}}]
\label{detecting CC}
Let $G$ be a group acting by simplicial automorphisms on a simplicial complex $X$, with a single facet $C$ as a strict fundamental domain.
Let
\begin{equation*}
\mathcal{P}\coloneqq \set{\Stab_G(v) \mid v \text{ is a vertex of } C}.
\end{equation*}
Then the map
\begin{align*}
\psi\colon \CC{G}{\mathcal{\mcP}}&\to X\\
 g\Stab_G(v)&\mapsto g.v
\end{align*}
is an isomorphism of simplicial $G$-complexes.
\end{proposition}

\subsection{Short exact sequences}
\label{sec SES and cosets}
We will later on study coset complexes in the setting where $G= \Out{A_\Gamma}$, the outer automorphism group of a right-angled Artin group. For this, we want to use the decomposition sequences of $ \Out{A_\Gamma}$ developed in \cite{DW:Relativeautomorphismgroups}. In order to do so, we need to study the following question: If $G$ fits into a short exact sequence, can the coset complex $\CC{G}{\mcH}$ be decomposed into ``simpler'' complexes related to the image and kernel of the sequence? There is a special case where this question can easily be answered:

\paragraph{Coset complexes and direct products} Assume that we have a group factoring as a direct product $G= G_1 \times G_2$ and let $\mcH$ be a family of subgroups such that each $H\in \mcH$ contains either $\set{1}\times G_2$ or $G_1 \times \set{1}$; denote the set of those elements of $\mcH$ satisfying the former by $\mcH_1$ and the set of those satisfying the latter by $\mcH_2$. Now given $H_1, H_1'\in \mcH_1$, we have
\begin{align*}
& (g_1,g_2) \cdot H_1 \cap (g_1',g_2')\cdot H_1'  \not= \emptyset \\
\Leftrightarrow \; & (g_1,1)\cdot H_1 \cap (g_1',1) \cdot H_1'  \not= \emptyset \\
\Leftrightarrow \; & g_1 \cdot p_1(H_1) \cap g_1' \cdot p_1(H_1')  \not= \emptyset ,
\end{align*}
where $p_1$ is the projection map $G\to G_1$. The analogous statement holds for $H_2, H_2'\in \mcH_2$. On the other hand, if we take $H_1\in \mcH_1$ and $H_2\in \mcH_2$, all of their cosets intersect non-trivially because 
\begin{align*}
(g_1,g_2) \cdot H_1 = (g_1,g_2') \cdot H_1 &&\text{and}&& (g_1',g_2') \cdot H_2 = (g_1,g_2') \cdot H_2.
\end{align*}
It follows that the coset complex $\CC{G}{\mcH}$ decomposes as a join
\begin{equation*}
\CC{G}{\mcH} \cong \CC{G_1}{p_1(\mcH_1)}\ast \CC{G_2}{p_2(\mcH_2)}.
\end{equation*}

However, the situation becomes more complicated if we consider semi-direct products or general short exact sequences
\begin{equation*}
1\to N\to G \to Q\to 1.
\end{equation*}
\cite[Proposition 5.17]{Hol:EndlicheIdentifizierbarkeitvon}, \cite[Theorem 7.3]{Wel:cosetposetsnerve} and \cite[Proposition 10]{Bro:cosetposetprobabilistic} contain results in this direction for the cases where every $H\in\mcH$ is a complement of $N$, every $H\in\mcH$ contains $N$ and where $G$ is a finite group and $\mcH$ is the set of \emph{all} subgroups of $G$, respectively.
 Our work in this section provides a common generalisation of all three of these results (see \cref{coset complexes and SES}).

\paragraph{Notation and standing assumptions}
\label{paragraph notation H^N H_N}
From now on, we will fix a normal subgroup $N\triangleleft G$ and assume that $\mcH$ is a set of proper subgroups of $G$.
In this situation, we can write $\mcH$ as a disjoint union $\mcH=\mcH_N \sqcup \mcH^N$, where
\begin{align*}
\mcH_N\coloneqq \set{H\in \mcH \mid HN \not=G} \text{ and } \mcH^N\coloneqq \set{K\in \mcH \mid KN =G}. 
\end{align*}
For elements $g\in G$ and subgroups $H\leq G$ of $G$, let $\bar{g}$ and $\overbar{H}$ denote the image of $g$ and $H$ in the quotient $G/N$, respectively. 

The family $\mcH_N$ gives rise to a family of proper subgroups of $G/N$, denoted by 
\begin{equation*}
\overbar{\mcH}\coloneqq \set{\overbar{H}\mid H\in\mcH_N}.
\end{equation*}
Similarly, $\mcH^N$ gives rise to a family of proper subgroups of $N$, denoted by 
\begin{equation*}
\mcH\cap N\coloneqq \set{K\cap N\mid K\in\mcH^N}.
\end{equation*}

\subsubsection{Coset posets and short exact sequences}
We start by considering the behaviour of coset posets under short exact sequences.

\begin{definition}
The family $\mcH$ of proper subgroups of $G$ is \emph{divided by $N$} if the following holds true:
\begin{enumerate}
\item  For all $H\in \mcH_N$, one has $HN\in \mcH$.
\item For all $H\in \mcH_N$ and $K\in \mcH^N$, one has $HN\cap K \in\mcH$.
\end{enumerate}
\end{definition}

In what follows, we will use the following elementary observations:
\begin{lemma}
\label{subgroups multiplied by N}
Let $H,K\leq G$ be two subgroups of $G$ and assume that $KN=G$. Then one has $(HN\cap K) \cdot N = HN$.
\end{lemma}
\begin{proof}
Obviously, $(HN\cap K)\cdot N$ is contained in $HN$. We claim that in fact, these sets are equal. Indeed, as $KN=G$, each $hn\in HN$ can be written as $hn=kn'$ with $k\in K$ and $n'\in N$. As $k=hnn'^{-1}$, it is contained in $HN \cap K$. Hence, $hn=kn'\in  (HN \cap K)\cdot N$. 
\end{proof}

\begin{lemma}
\label{intersection fibres}
Let $H\in \mcH_N$, $K \in \mcH^N$ and $g\in G$. If $\mcH$ is divided by $N$, then 
\begin{equation*}
(g\cdot HN)\cap K = k \cdot (HN\cap K)
\end{equation*}
for some $k \in K$. Furthermore, $(HN\cap K)\in \mcH_N$.
\end{lemma}
\begin{proof}
As $G=KN$, we can write $g= k n$ with $n\in N$ and $k\in K$. The intersection
\begin{equation*}
\label{equation representative from K}
(g\cdot HN)\cap K= (k n\cdot HN)\cap K = (k\cdot HN)\cap K
\end{equation*}
contains $k$, so it is 
equal to $k\cdot (HN\cap K)$.
That $HN\cap K$ is contained in $\mcH$ is clear because $\mcH$ is divided by $N$; that it is contained in $\mcH_N$ is a consequence of \cref{subgroups multiplied by N}.
\end{proof}

The next proposition is a generalisation of \cite[Proposition 10]{Bro:cosetposetprobabilistic}. Our proof closely follows the ideas of Brown.
\begin{proposition}
\label{coset posets and SES}
If $\mcH$ is divided by $N$, then there is a homotopy equivalence
\begin{equation*}
\CosPos{G}{\mcH}\simeq \CosPos{G/N}{\overbar{\mcH}}\ast\CosPos{G}{\mcH^N}.
\end{equation*}
\end{proposition}
\begin{proof}
Set $C\coloneqq \CosPos{G}{\mcH}$, $C_N\coloneqq \CosPos{G}{\mcH_N}$ and $C^N\coloneqq \CosPos{G}{\mcH^N}$. We define a map 
\begin{equation*}
f\colon C\to \CosPos{G/N}{\overbar{\mcH}}\ast C^N
\end{equation*}
such that $f$ restricts to the identity on $C^N$ and $f(gH)=\bar{g}\overbar{H}$ for all $gH\in \nolinebreak C_N$. As no coset from $C^N$ can be contained in a coset from $C_N$, this map is order-preserving, i.e.~a poset map. For $x\in \CosPos{G/N}{\overbar{\mcH}}\ast C^N$, define
\begin{equation*}
F\coloneqq f^{-1}(\,(\CosPos{G/N}{\overbar{\mcH}}\ast C^N \,)_{\leq x})
\end{equation*}
to be the fibre of $x$ with respect to $f$.  We want to use \cref{Quillen fibre lemma contractibility} to show that $f$ is a homotopy equivalence. For this, we need to show that $F$ is contractible.

If $x\in \CosPos{G/N}{\overbar{\mcH}}$, this is clear: Write $x=\bar{g}\overbar{H}$ such that $g \in G, \, H\in \mcH_N$. As $N$ divides $\mcH$, the subgroup $HN$ is contained in $\mcH$ and $g\cdot HN$ is the unique maximal element of $F$. This immediately implies contractibility of $F$.

Now assume $x\in C^N$. Using the natural action of $G$ on these posets, we can assume that $x = K \in \mcH^N$.
By definition of the join, the poset $F$ can as be written as $F=C_N \cup C_{\leq K}$. On the level of geometric realisations, it decomposes as 
\begin{equation*}
\real{F}=\real{C_N}\cup_{\real{C'}} \real{C_{\leq K}},
\end{equation*} 
where $C'\coloneqq C_N\cap C_{\leq K}$ is equal to $(C_N)_{\leq K}$. (To see this, note that no coset from $C^N$ can be contained in a coset from $C_N$ and that if $gH\in C_N$ is contained in some $g'H'\in C_{\leq K}$, we have $gH\in C'$.)
Next, we show that $\real{C'}$ is a strong deformation retract of $\real{C_N}$. This implies that $F$ is homotopy equivalent to $C_{\leq K}$, which is contractible as it has $K$ as unique maximal element.

The poset $C'$ is given by all cosets $gH\subseteq K$ such that $H\in \mcH_N$. Hence, \cref{intersection fibres} implies that for $gH\in C_N$, the intersection $(g\cdot HN)\cap \nolinebreak K$ is an element of $C'$. This allows us to define poset maps
\begin{align*}
\phi\colon  C_N 	&\to C' &\text{ and }& & \psi\colon  C' 	&\to C_N \\
gH 			&\mapsto 	(g\cdot HN)\cap K &&& gH 	&\mapsto g\cdot HN.
\end{align*}
For $gH\in C'$, we have $gH\subseteq K$, hence 
\begin{equation*}
\phi\circ\psi (gH)= (g\cdot HN)\cap K \supseteq gH\cap K =gH .
\end{equation*}
If on the other hand $gH\in C_N$, one has by \cref{intersection fibres}
\begin{align*}
\psi\circ\phi (gH)&= ((g\cdot HN)\cap K )\cdot N \\
& = k\cdot (HN\cap K)\cdot N 
\end{align*}
for some $k\in g\cdot H N\cap K$. By \cref{subgroups multiplied by N}, we have $(HN\cap K) \cdot N =HN$, so it follows that $\psi\circ\phi (gH) =k\cdot HN \supseteq g H$.
\cref{Quillen homotopic poset maps} now implies that $\phi$ and $\psi$ are homotopy equivalences which are inverse to each other. Furthermore, we have $gH\subseteq \psi(gH)$ for all $gH\in C'$, so again by \cref{Quillen homotopic poset maps}, the map $\psi$ is homotopic to the inclusion $C'\hookrightarrow C_N$ which must hence be a homotopy equivalence as well. It follows that $\real{C'}$ is a strong deformation retract of $\real{C_N}$.
\end{proof}

\subsubsection{Coset complexes and short exact sequences}
We will now translate the results obtained in the last section to coset complexes. 
The following observation follows from elementary group theory.

\begin{lemma}
\label{subgroups distinct in quotient}
Let $K_1 \not= K_2$ be subgroups of $G$ such that $G=(K_1\cap K_2) N$. Then one has $K_1\cap N \not= K_2\cap N$.
\end{lemma}

We obtain the following relation between $\CC{G}{\mcH^N}$ and $\CC{N}{\mcH\cap N}$:

\begin{lemma}
\label{intersection with kernel}
Assume that for every finite collection $K_1,\ldots, K_m \in \mcH^N$, one has $(K_1\cap \ldots \cap K_m)N=G$. Then there is an isomorphism
\begin{equation*}
\CC{G}{\mcH^N}\cong \CC{N}{\mcH\cap N}.
\end{equation*}
\end{lemma}
\begin{proof}
As $G=KN=NK$ for all $K \in \mcH^N$, each vertex of $\CC{G}{\mcH^N}$ can be written as $n K$ with $n\in N$. Use this to define the map
\begin{align*}
\psi\colon  \CC{G}{\mcH^N} & \to \CC{N}{\mcH\cap N}\\
n K & \mapsto n \cdot K\cap N
\end{align*}
which we claim is an isomorphism of simplicial complexes.

As $n\in N$, this map is well-defined on vertices. It also clearly is surjective on vertices.
Now assume that for $n_1,n_2\in N$ and $K_1,K_2\in \mcH^N$, one has $n_1 \cdot K_1\cap N = n_2 \cdot K_2 \cap N$. As the two cosets coincide, so do the subgroups $K_1\cap N = K_2 \cap N $. By \cref{subgroups distinct in quotient}, this implies that $K_1 = K_2$. It follows in particular that $n_1 K_1 = n_2 K_2$ which shows that $\psi$ defines a bijection between the vertex sets of the two coset complexes.

To see that $\psi$ is a simplicial map which defines a bijection between the set of simplices of the two complexes, take $n_1, \ldots , n_m \in N$ and $K_1,\ldots ,K_m\in \mcH^N$ and consider the following chain of equivalences:
\begin{align*}
&\bigcap_i n_i K_i \not =\emptyset \\
\Leftrightarrow \,& \exists\, g\in G: \bigcap_i n_i K_i = \bigcap_i g K_i = g\bigcap_i K_i\\
\stackrel{*}{\Leftrightarrow} \,& \exists\, n\in N: \bigcap_i n_i K_i = n \bigcap_i K_i \\
\Leftrightarrow \,& \emptyset \not =\left(\bigcap_i n_i K_i \right)\cap N = \bigcap_i n_i (K_i\cap N) ,
\end{align*} 
where $*$ follows because $G=N(K_1\cap \ldots \cap K_m)$.
\end{proof}

This motivates the following definition:
\begin{definition}
\label{definition strongly divided}
The family $\mcH$ of proper subgroups of $G$ is \emph{strongly divided by $N$} if the following holds true:
\begin{enumerate}
\item For all $H\in \mcH_N$, one has $N\subseteq H$.
\item For all $K_1,\ldots ,K_m \in \mcH^N$, one has $(K_1\cap \ldots \cap K_m)N=G$.
\end{enumerate}
\end{definition}

Using \cref{subgroups multiplied by N}, it is easy to see that every family of subgroups which is strongly divided by $N$ is also divided by $N$. On top of that, given a family which is strongly divided, we can even produce a family which is closed under intersections and still divided by $N$ as the following lemma shows. Recall that $\widetilde{\mcH}$ denotes the family of all finite intersections of elements from $\mcH$.

\begin{lemma}
\label{finite intersections divided by N}
If $\mcH$ is strongly divided by $N$, the family $\widetilde{\mcH}$ is divided by $N$. Furthermore,  we have
\begin{enumerate}
\item \label{item intersections and H^N} $\widetilde{\mcH}^N$ is equal to the family of all finite intersections of elements from $\mcH^N$, i.e. 
\begin{equation*}
\widetilde{\mcH}^N = \widetilde{\mcH^N}.
\end{equation*}
\item \label{item intersections and H_N} The image of $\widetilde{\mcH}_N$ in $G/N$ is equal to the family of finite intersections of elements from $\overbar{\mcH}$, i.e. 
\begin{equation*}
\overbar{\widetilde{\mcH}} = \widetilde{\overbar{\mcH}}.
\end{equation*}
\end{enumerate}
\end{lemma}
\begin{proof}
Every $\widetilde{H}\in \widetilde{\mcH}$ can be written as 
\begin{equation*}
\widetilde{H}= H_1\cap \ldots \cap H_n \cap K_1\cap\ldots \cap K_m
\end{equation*}
where for all $i$ and $j$, one has $N\subseteq H_i$ and $K_j \in \mcH^N$.

If $\widetilde{H}\in \widetilde{\mcH}^N$, we must have $n=0$, i.e. $\widetilde{H}=K_1\cap\ldots \cap K_m$ is a finite in\-ter\-section of elements from $\mcH^N$. On the other hand, every such finite intersection forms an element of $\widetilde{\mcH}^N$ because one has $(K_1\cap\ldots \cap K_m)N=G$, which proves \cref{item intersections and H^N}.

This also implies that if $\widetilde{H}\in \widetilde{\mcH}_N$, we have $n\geq 1$. It follows from \cref{subgroups multiplied by N} that $\widetilde{H}N$  is equal to $H_1\cap \ldots \cap H_n $. This is  a finite intersection of elements from $\mcH_N$ and hence contained in $\widetilde{\mcH}$. Furthermore, this implies that the image $\overbar{H}$ of $\widetilde{H}$ in $G/N$ is equal to $\overbar{H}= \overbar{H}_1\cap \ldots \cap \overbar{H}_n $, showing \cref{item intersections and H_N}.

The last thing that remains to be checked is that $\widetilde{H}$ is divided by $N$, i.e. that for all $\widetilde{H}\in \widetilde{\mcH}_N$ and $\widetilde{K}\in \widetilde{\mcH}^N$, one has $\widetilde{H} N\cap \widetilde{K} \in\widetilde{\mcH}$. However, we already know that $\widetilde{H}N = H_1\cap \ldots \cap H_n $, so $\widetilde{H} N\cap \widetilde{K}$ is itself a finite intersection of elements from $\mcH$.
\end{proof}

We are now ready to prove Theorem \ref{introduction coset complexes and SES} which we restate as:
\begin{theorem}
\label{coset complexes and SES}
If $\mcH$ is strongly divided by $N$, there is a homotopy equivalence
\begin{equation*}
\CC{G}{\mcH}\simeq \CC{G/N}{\overbar{\mcH}}\ast \CC{N}{\mcH\cap N}.
\end{equation*}
\end{theorem}
\begin{proof}
It follows from \cref{homotopy equivalence cc and cospos} and \cref{homotopy equivalence cc closed under intersections} that $\CC{G}{\mcH}$ is homotopy equivalent to $\CosPos{G}{\widetilde{\mcH}}$. Furthermore, \cref{finite intersections divided by N} tells us that $\widetilde{\mcH}$ is divided by $N$. Hence, we can apply \cref{coset posets and SES} to see that there is a homotopy equivalence
\begin{equation*}
\CosPos{G}{\widetilde{\mcH}} \simeq \CosPos{G/N}{\overbar{\widetilde{\mcH}}}\ast\CosPos{G}{\widetilde{\mcH}^N}.
\end{equation*}

By \cref{finite intersections divided by N}, we have $\overbar{\widetilde{\mcH}} = \widetilde{\overbar{\mcH}}$. Hence, using \cref{homotopy equivalence cc and cospos} and \cref{homotopy equivalence cc closed under intersections} again,
\begin{equation*}
\CosPos{G/N}{\overbar{\widetilde{\mcH}}}\simeq \CC{G/N}{\widetilde{\overbar{\mcH}}}\simeq \CC{G/N}{\overbar{\mcH}}.
\end{equation*}
On the other hand, \cref{finite intersections divided by N} also tells us that $\widetilde{\mcH}^N$ consists of all finite intersections of elements from $\mcH^N$. It follows that
\begin{equation*}
\CosPos{G}{\widetilde{\mcH}^N}\simeq \CC{G}{\widetilde{\mcH}^N}\simeq \CC{G}{\mcH^N}.
\end{equation*}
As $\mcH$ is strongly divided by $N$, we can finally apply \cref{intersection with kernel} and get that $\CC{G}{\mcH^N}\cong \CC{N}{\mcH\cap N}$.
\end{proof}

\subsubsection{Summary}
We summarise the results of this section in the form that we will use later on:

\begin{corollary}
\label{short exact sequences}
Let $G$ be a group and assume we have a short exact sequence
\begin{equation*}
1\to N\to G \stackrel{q}{\to} Q\to 1.
\end{equation*}
Let $S$ be a set of generators for $G=\ll S \rr$ and let $\mcP$ be a family of proper subgroups. Furthermore, assume that for all $P\in \mcP$, one of the following holds:
\begin{enumerate}
\item Either $P$ contains the kernel $N=\ker q$, or
\item $P$ contains $S\bsl N$.
\end{enumerate}
Then there is a homotopy equivalence
\begin{equation*}
\CC{G}{\mcP}\simeq \CC{Q}{\overbar{\mcP}}\ast\CC{N}{\mcP\cap N},
\end{equation*}
where $\overbar{\mcP}=\set{q(P) \mid P\in \mcP, \, N \subseteq P}$ and $\mcP\cap N = \set{P \cap N \mid P \in \mcP, \, S\bsl N \subseteq P}$.
\end{corollary}
\begin{proof}
We stick with the notation defined on page \pageref{paragraph notation H^N H_N}. If $P\in \mcP^N$, it cannot contain $N$. Hence, all such $P$ must contain the set $S\bsl N$ of elements from $S$ that are not contained in the kernel. It follows that for any $P_1,\ldots, P_m\in \mcP^N$, one has $(P_1\cap \ldots \cap P_m)N=G$. On the other hand, for every $P\in\mcP_N$, our assumption implies that $N\subseteq P$. Hence, $\mcP$ is strongly divided by $N$ and the claim follows from \cref{coset complexes and SES}.
\end{proof}

\section{The base cases: Buildings and relative free factor complexes}
\label{section building and relative free factor complex}
In this section, we study complexes of parabolic subgroups associated to two particular families of (relative) automorphism groups: The first one is $\GL{n}{\mathbb{Z}}$ (\cref{section buildings}), the second one are so-called Fouxe-Rabinovitch groups (\cref{section factor complexes}).
 On the one hand, these are special cases of the complexes we will consider in \cref{section spherical complex for Out(RAAG)}, on the other hand, they play a distinguished role because they appear as base cases of the inductive argument that we will use there.
We show that in both situations, the complexes one obtains are spherical, but the methods for the two cases are quite different. In the first one, the result follows without much effort from the Solomon--Tits Theorem while in the second one, we have to generalise the work of \cite{BG:Homotopytypecomplex} to the ``relative'' setting considered here.

\subsection{The building associated to $\boldsymbol{\GL{n}{\mathbb Z}}$ and the Solomon--Tits Theorem}
\label{section buildings}
The \emph{building associated to $\GL{n}{\mathbb Q}$} is the order complex of the poset $\mcQ$ of proper (i.e.~non-trivial and not equal to $\mathbb{Q}^n$) subspaces of $\mathbb{Q}^n$, ordered by inclusion. 

This is a special case of a Tits building and a lot can be said about the structure 
of these simplicial complexes --- we refer the reader to \cite{AB:Buildings} for further details.
However, the only non-trivial result about them that we need for this article is the following special case of the Solomon--Tits Theorem:
\begin{theorem}[{\cite{Sol:Steinbergcharacterfinite}}]
\label{Solomon-Tits}
The building associated to $\GL{n}{\mathbb Q}$ is homotopy equivalent to a wedge of $(n-2)$-spheres.
\end{theorem}

It is well-known that this building can equivalently be described as the coset complex of $\GL{n}{\mathbb Q}$ with respect to the family of maximal standard parabolic subgroups. We will now show that it can also be described as a coset complex of $\GL{n}{\mathbb Z} = \Out{\mathbb{Z}^n}$, an outer automorphism group of a RAAG.

A subgroup $A\leq \mathbb{Z}^n$ is called a \emph{direct summand} if there is $B\leq \mathbb{Z}^n$ such that $\mathbb{Z}^n = A \oplus B$. We say that a direct summand $A$ is \emph{proper} if it is neither trivial nor equal to $\mathbb{Z}^n$.
Let $\mcZ$ be the poset of all proper direct summands of $\mathbb{Z}^n$, ordered by inclusion. The group $\GL{n}{\mathbb Z}$ acts naturally on $\mcZ$.

Fix a basis $\set{e_1,\ldots , e_n}$ of $\mathbb{Z}^{n}$ and for all $1\leq i \leq n-1$, set $S_i \coloneqq \nolinebreak\ll e_1,\ldots, e_i \rr$.
Note that $S_i\in \mcZ$ for all $i$ and define 
\begin{equation*}
P_i\coloneqq\Stab_{\GL{n}{\mathbb Z}}(S_i)
\end{equation*}
to be the stabiliser of $S_i$ under the action of $\GL{n}{\mathbb Z}$ on $\mcZ$. We define the set of \emph{maximal standard parabolic subgroups of $\GL{n}{\mathbb Z}$} as 
\begin{align*}
\mcP = \mcP(\GL{n}{\mathbb Z}) \coloneqq\set{P_i \mid 1\leq i\leq n-1} . 
\end{align*}

\begin{remark}
\label{remark standard parabolics}
We called the elements of $\mcP$ the maximal \emph{standard} parabolic subgroups of $\GL{n}{\mathbb Z}$ to match the usual convention where an arbitrary parabolic subgroup is defined as the conjugate of a standard one. As we will however not work with non-standard parabolic subgroups in this article, we leave out this adjective from now on.  
\end{remark}
In terms of matrices, the maximal parabolic subgroups can be written in the form
\begin{equation*}
P_i = \begin{pmatrix}
\GL{i}{\mathbb{Z}} 	& M_{i,n-i}(\mathbb{Z})	\\
0	&\GL{n-i}{\mathbb{Z}}	
\end{pmatrix}\leq \GL{n}{\mathbb{Z}}.
\end{equation*}

\begin{proposition}
\label{building and coset complex}
The building associated to $\GL{n}{\mathbb{Q}}$ is $\GL{n}{\mathbb{Z}}$-equivariantly isomorphic to the coset complex $\CC{\GL{n}{\mathbb{Z}}}{\mcP}$.
\end{proposition}
\begin{proof}
Each $A\in \nolinebreak \mcZ$ is isomorphic to $\mathbb{Z}^i$ for an integer $i \coloneqq \rk(A)\in \set{1,\ldots, n-1}$, the rank of $A$. Furthermore, if $A\leq B$ in $\mcZ$, we have $\rk(A)\leq \rk(B)$ with equality if and only if $A$ and $B$ are equal. It follows that the maximal simplices of $\Delta(\mcZ)$ are given by chains $A_1 \leq \ldots \leq A_{n-1}$,
where $\rk(A_i)=i$. The group $\GL{n}{\mathbb{Z}}$ acts transitively on the set of all such chains and preserves the rank of each summand. Hence, the facet $S_1 \leq \ldots \leq S_{n-1}$ is a fundamental domain for this action and \cref{detecting CC} implies that the order complex of $\mcZ$ is $\GL{n}{\mathbb{Z}}$-equivariantly isomorphic to $\CC{\GL{n}{\mathbb{Z}}}{\mcP}$.

On the other hand, there is a poset map $f\colon  \mcQ \to \mcZ$ defined by sending $V$ to $V\cap \mathbb{Z}^n$. This is a $\GL{n}{\mathbb{Z}}$-equivariant isomorphism whose inverse is given by sending $A\leq \mathbb{Z}^n$ to its $\mathbb{Q}$-span $\ll A \rr_\mathbb{Q}$ (see e.g.~\cite[Corollary 2.5]{CP:codimensiononecohomology}).
\end{proof}

\subsection{Relative free factor complexes}
\label{section factor complexes}
The aim of this section is to generalise \cite[Theorem A]{BG:Homotopytypecomplex} which states that the complex of free factors of the free group $F_n$ is homotopy equivalent to a wedge of $(n-2)$-spheres. We want to extend this result to certain complexes of free factors of a free product $A= F_n \ast A_1 \ast \ldots \ast A_k$.
After adapting the definitions  to this setting, the proofs of \cite{BG:Homotopytypecomplex} largely go through without major changes.
We still include most of them here in order to make this section as self-contained as possible.

\subsubsection{Relative automorphism groups and relative Outer space}
\label{section relative automorphism groups}
\paragraph{Relative automorphism groups}
Let $A$ be a countable group. We will often use capital letters for elements from the outer automorphism group of $A$ and lower-case letters for the corresponding representatives from the automorphism group of $A$. I.e.~for $\Phi\in \Out{A}$, we write $\Phi=[\phi]$ where $\phi\in \Aut{A}$.
Let $\Phi$ be an outer automorphism of a group $A$ and $H\leq A$ a subgroup. 
Then $\Phi$ \emph{stabilises} $H$ or \emph{$H$ is invariant under $\Phi$} if there exists a representative $\phi\in\Phi$ such that $\phi(H)=H$. We say that $\Phi$ \emph{acts trivially} on $H$ if there is $\phi\in \Phi$ restricting to the identity on $H$.

If $\mcG$ and $\mcH$ are families of subgroups of $A$, the \emph{relative outer automorphism group} $\Out{A; \mcG, \mcH^t} $ is the subgroup of $\Out{A}$ consisting of all elements stabilising each $H\in \mcG$ and acting trivially on each $H\in \mcH$. If $\mcG$ or $\mcH$ are given by the empty set, we also write $\Out{A;\mcH^t} $ or $\Out{A; \mcG}$ for this group.

If $O\leq \Out{A}$ is a subgroup of the outer automorphism group of $A$ and $G\leq A$, we also write 
\begin{equation*}
\Stab_O(G)
\end{equation*}
for the subgroup of $O$ consisting of all elements that stabilise $G$. In the case where $O$ is equal to $\Out{A; \mcG, \mcH^t} $, we have $\Stab_O(G)=\Out{A; \mcG\cup \set{G}, \mcH^t}$.

\paragraph{Free splittings}
A \emph{free splitting} $S$ of $A$ is a non-trivial, minimal, simplicial $A$-tree with finitely many edge orbits and trivial edge stabilisers.
The \emph{vertex group system} of a free splitting $S$ is the (finite) set of conjugacy classes of its vertex stabilisers. 
 Two free splittings $S$ and $S'$ are equivalent if they are equivariantly isomorphic. We say that $S'$ collapses to $S$ if there is a \emph{collapse map} $S'\to S$ which collapses an $A$-invariant set of edges. The \emph{poset of free splittings} $\FreeS$ is given by the set of all equivalence classes of free splittings of $A$ where $S\leq S'$ if $S'$ collapses to $S$. The \emph{free splitting complex} is the order complex $\Delta(\FreeS)$ of the poset of free splittings.

\paragraph{Fouxe-Rabinovitch groups and relative Outer space}

Let $A$ be a finitely generated group that splits as a free product 
\begin{equation*}
A = F_n \ast A_1 \ast \cdots \ast A_k
\end{equation*}
where $F_n$ denotes the free group on $n$ generators and $n+k \geq 2$.
Define $\mcA \coloneqq \set{A_1, \ldots, A_k}$ and $O\coloneqq \Out{A; \mcA^t} $.
The group $O$ is also called a \emph{Fouxe-Rabinovitch group} because of the work of Fouxe-Rabinovitch on automorphism groups of free products \cite{Fou:UberdieAutomorphismengruppena}.

In \cite{GL:outerspacefree}, Guirardel and Levitt define a topological space called \emph{relative Outer space} for such groups.
This space contains a \emph{spine}, which is denoted by $L = L(A, \mcA)$. This spine is (the order complex of) the subposet of $\FreeS$ consisting of all free splittings whose vertex group system is equal to the set of conjugacy classes of elements from $\mcA$.
The poset $L$ is contractible and $O$ acts cocompactly on it.

\subsubsection{Parabolic subgroups and relative free factor complexes}
\label{sec:rel_free_factor}
\paragraph{Standing assumptions and notation}
From now on and until the end of \cref{section factor complexes}, fix a finitely generated group $A = F_n \ast A_1 \ast \cdots \ast A_k$ with $n\geq 2$ and a basis $\set{x_1,\ldots , x_n}$ of $F_n$. As above, let $\mcA \coloneqq \set{A_1, \ldots, A_k}$ and $O\coloneqq \Out{A; \mcA^t}$.\\

A \emph{free factor} of $A$ is a subgroup $B\leq A$ such that $A$ splits as a free product $A=B \ast C$. There is a natural partial order on the set of conjugacy classes of free factors of $A$ given by $[B_1] \leq [B_2]$ if, up to conjugacy, $B_1$ is contained in $B_2$.

\begin{definition}
\label{definition relative free factor}
Let $\FRfactor = \mcF(A;\mcA)$ denote the poset of all conjugacy classes of proper free factors $B\subset A$ such that there is a free factor $B'$ of $A$ with $[A_i] \leq [B']$ for all $i$ and $B'$ is a proper subgroup of $B$. (In particular, $[A_1 \ast \cdots \ast A_k]\not \in \FRfactor$.)
We call the order complex of $\FRfactor$ the \emph{free factor complex of $A$ relative to $\mcA$}. It carries a natural, simplicial action of $O$.
\end{definition}
\begin{remark}
If $k=0$, the poset $\FRfactor$ consists of \emph{all} conjugacy classes of proper free factors of $F_n$, so we recover the free factor complex of $F_n$. More generally, the free factor complex of $A$ relative to $\mcA$ is a subcomplex of the \emph{complex of free factor systems of $A$ relative to $\mcA$} as defined by Handel--Mosher \cite{HM:Relativefreesplitting}. The ordering $\sqsubseteq$ of free factor systems defined there restricts to the ordering on $\FRfactor$ for free factor systems having only one component.

Our definition however differs from the one used by Guirardel--Horbez, e.g. in \cite{GH:Algebraiclaminationsfree}; in their definition, a proper free factor $B\leq A$ is relative to $\mcA$ if $A=B \ast C$ where for all $i$, either $[A_i]\leq [B]$ or $[A_i]\leq [C]$ and $[A_i]\not= [B]$.

For studying geometric questions, the definition of the free factor complex and similar complexes is often adapted such that it becomes connected for low $n$ as well. This is not the case for the definition used in the present article, where the free factor complex associated to $\Out{F_2}$ is a disjoint union of points.
\end{remark}

\paragraph{Corank}
\cite[Lemma 2.11]{HM:Relativefreesplitting} implies that the elements of $\FRfactor$ are conjugacy classes of groups of the form
\begin{equation*}
B =  F \ast A_1^{a_1} \ast \ldots \ast A_k^{a_k},
\end{equation*} 
where $a_j\in A$ and $F$ is a free group with $1\leq \rk(F)\leq n-1$. Furthermore, we can write $A$ as a free product $A = B \ast C$, where $C$ is a free group of rank $n-\rk(F)$. The rank of $C$ is an invariant of the conjugacy class $[B]$ (see \cite[Section 2.3]{HM:Relativefreesplitting}). It is called the \emph{corank} of $[B]$ and will be denoted by $\corank [B]$.
\\

We study these relative free factor complexes because they can also be described as coset complexes of parabolic subgroups: Let
\begin{equation*}
S_i\coloneqq \ll x_1,\ldots, x_i \rr \ast A_1\ast \ldots \ast A_k.
\end{equation*}
Every $S_i$ is a free factor of $A$ because for all $i$, we have $A= S_i \ast \ll x_{i+1},\ldots, x_n \rr $.
We set $P_i\coloneqq\Stab_O(S_i)$ and define the set of \emph{maximal standard parabolic subgroups of $O$} as 
\begin{align*}
\mcP = \nolinebreak\mcP(O) \coloneqq\nolinebreak \set{P_i \mid 1\leq i\leq n-1} .
\end{align*}
As in the case of $\GL{n}{\mathbb{Z}}$, we will usually leave out the adjective ``standard'' (see \cref{remark standard parabolics}). 
\begin{proposition}
\label{isomorphism relative free factor complex}
The free factor complex of $A$ relative to $\mcA$ is $O$-equivariantly isomorphic to the coset complex $\CC{O}{\mcP}$.
\end{proposition}
\begin{proof}

If $[B_1] \leq [B_2]$, we know from \cite[Proposition 2.10]{HM:Relativefreesplitting} that the corank of $[B_2]$ is smaller than or equal to the corank of $[B_1]$ and that equality holds if and only if $[B_1]=[B_2]$.
Consequently, the simplices of $\Delta(\FRfactor)$ are given by chains of the form
\begin{equation*}
[B_1]  \leq [B_2] \leq \ldots \leq [B_m] 
\end{equation*}
with $\corank[B_1]< \corank[B_2]< \ldots < \corank[B_m]$. Let $i_j\coloneqq \corank[B_j]$.

We claim that for each such chain, there exists $\Phi\in O$ with $[\phi(S_{i_j})] = \nolinebreak {[B_j]}$ for all $j$.
To see this, first observe that sending each $A_i$ to a conjugate of itself and fixing all the other generators defines an automorphism of $A$ that represents an element in $O$.
 Hence, we can assume that $A_1\ast \ldots \ast A_k \leq B_1$. Now choose representatives such that  $B_j\leq B_{j+1}$ for all $j$. In order to use induction, assume that there is $\Phi'\in O$ such that for some $l$, we have $\phi'(S_{i_j}) = B_{j}$ for all $0\leq j\leq l$ --- this is true for $l=0$ where we define $i_0=0$ and $B_0=S_0=A_1\ast \ldots \ast A_k$.
By assumption, $\phi' (S_{i_{l}}) = B_l \leq B_{l+1}$, so \cite[Lemma 2.11]{HM:Relativefreesplitting} implies that
\begin{align*}
A = \phi'(S_{i_{l}}) \ast C \ast D , && \text{ where } && B_{l+1} = \phi'(S_{i_{l}}) \ast C
\end{align*}
and $C$ and $D$ are free groups of rank $(i_{l+1} - i_{l})$ and $(n-i_{l+1})$, respectively.
On the other hand, the group $A$ also decomposes as a free product
\begin{align*}
 A = S_{i_{l}} \ast \ll x_{i_l + 1} ,\ldots , x_{i_{l+1}} \rr \ast \ll x_{i_{l+1}+1} ,\ldots , x_n \rr.
\end{align*}
This allows us to define an automorphism $\phi$ of $A$ which agrees with $\phi'$ on $S_{i_{l}}$, maps $\ll x_{i_l + 1} ,\ldots , x_{i_{l+1}} \rr$ isomorphically  to $C$ and  $\ll x_{i_{l+1}+1} ,\ldots , x_n \rr$ to $D$.
As $\phi$ agrees with $\phi'$ on $S_{i_{l}}$, we know  that $[\phi(S_{i_j})] = [B_j]$ for all $j\leq l$  and that $\phi$ acts by conjugation on each $A_i$, i.e. $[\phi]\in O$. Furthermore, we have
\begin{align*}
\phi (S_{i_{l+1}}) &= \phi(S_{i_{l}})\ast \phi(\ll x_{i_l + 1} ,\ldots , x_{i_{l+1}} \rr) \\
 &= \phi'(S_{i_{l}}) \ast C = B_{l+1}.
\end{align*}
By induction, this proves the claim.

On the other hand, for each $[\phi]\in O$, the chain 
\begin{equation*}
[\phi (S_1)]  \leq [\phi (S_2)] \leq \ldots \leq [\phi (S_{n-1})] 
\end{equation*}
forms a facet in $\Delta(\FRfactor)$. Hence, every facet of $\Delta(\FRfactor)$ can be written in this form.
It follows that the natural action of $O$ on $\Delta(\FRfactor)$ has a fundamental domain given by the simplex
\begin{equation*}
[S_1]  \leq [S_2] \leq \ldots \leq [S_{n-1}] 
\end{equation*}
The result now follows from \cref{detecting CC}.
\end{proof}

Note that the corank played in this proof the same role as the dimension and rank did in the proof of \cref{building and coset complex}.

\subsubsection{The associated complex of free splittings}
In order to study the connectivity properties of relative free factor complexes, we will use yet another description of them; namely, we will show in this subsection that they are homotopy equivalent to certain posets of free splittings.

Let $L\coloneqq L(A,\mcA)$ be the spine of Outer space of $A$ relative to $\mcA$. 
Taking the quotient by the action of $A$, each free splitting $S\in L$ can equivalently be seen as a marked graph of groups $\mathbb{G}$. The edge groups of $\mathbb{G}$ are trivial and for all $1\leq i \leq k$, there is exactly one vertex group which is conjugate to $A_i$. All the other vertex groups are trivial. 
 The \emph{marking} is an isomorphism $\pi_1(\mathbb{G}) \to A$ that is well-defined up to composition with inner automorphisms. Using this description, the action of $O$ on $L$ is given by changing the marking. The underlying graph $G$ of $\mathbb{G}$ is finite, has fundamental group of rank $n$ and all of its vertices with valence one have non-trivial vertex group. 

\begin{definition}
\label{def:labelled&core}
\begin{enumerate}
\item A \emph{labelled graph} is a pair $(G, l)$ consisting of a graph $G$ and a map $l\colon \set{1,\ldots ,k}\to V(G)$ to its vertex set $V(G)$. We call the image of $l$ the \emph{labelled vertices} of $G$.
\item A connected labelled graph $(G,l)$ is called a \emph{core graph} if it has non-trivial fundamental group and every vertex of valence one lies in the image of $l$.
\end{enumerate}
\end{definition}

For the graph $G$ associated to $S\in L$ as above, there is a natural labelling $l\colon \set{1,\ldots ,k}\to V(G)$ of $G$ given by defining $l(i)$ as the vertex with vertex group conjugate to $A_i$. It follows that $(G,l)$ is a core graph. 
If $H$ is a connected subgraph of $G$ that contains all the vertices with non-trivial vertex group, then there is an induced structure of a marked graph of groups on $H$. We define the \emph{fundamental group $\pi_{\mathbb{G}}(H)$} as the fundamental group of this graph of groups. It is a subgroup of $A$ that is well-defined up to conjugacy and has the form
\begin{equation*}
\pi_{\mathbb{G}}(H) = F\ast A_1^{a_1} \ast \ldots \ast A_k^{a_k},
\end{equation*} 
where $a_i\in A$ and $F$ is a free group with rank equal to the rank of $\pi_1(H)$.

\begin{definition}
Let $S\in L$, let $\mathbb{G}$ be the associated graph of groups and $(G,l)$ the underlying labelled graph. Let $B\leq A$ be a subgroup of $A$. We say that \emph{$S$ has a subgraph with fundamental group $[B]$} if there is a subgraph $H$ of $G$ such that $[\pi_{\mathbb{G}}(H)]= [B]$.

If such a subgraph exists, there is also a unique core subgraph of $(G,l)$ with fundamental group $[B]$ which will be denoted by $B|S$.
We then also say that \emph{$B|S$ is a subgraph of $S$}.
\end{definition}

\paragraph{Notation}
To simplify notation, we will from now on not distinguish between a free splitting $S$ and the corresponding graph of groups. For example, we will talk about ``(core) subgraphs of $S$\,'' and mean (core) subgraphs of the corresponding labelled graph $(G,l)$. Instead we will use the letter $G$ for elements in $L=L(A,\mcA)$ and the letter $S$ for free splittings that have vertex group system different than $\mcA$. If $G\in L$ and $H$ is a subgraph, let $G/H$ denote the free splitting obtained by collapsing $H$.

\begin{definition}
Let $\FreeOne = \FreeOne(A; \mcA)$ be the poset of all free splittings $S$ of $A$ that have \emph{exactly one} conjugacy class $\mcV(S)$ of non-trivial vertex stabilisers and such that $\mcV(S)\in \FRfactor$.

For $[B]\in \FRfactor$, let $\FreeOne(B)$ be the poset consisting of all $S\in \nolinebreak \FreeOne$ such that $[B] \leq \mcV(S)$.
\end{definition}

\begin{proposition}
\label{contractibility of relative free splitting complex}
For all $[B]\in \FRfactor$, the poset $\FreeOne(B)$ is contractible.
\end{proposition}
This proposition can be shown as \cite[Theorem 5.8]{BG:Homotopytypecomplex}; there, only the case where $A$ is a free group is considered, but the proof generalises to the present situation without any major changes. In what follows, we provide an outline of the main steps.

\begin{proof}[Sketch of proof of \cref{contractibility of relative free splitting complex}]
For a chain of free factors of $A$ given by $ B \subset B_1 \subset \ldots \subset B_l \subset C_{0} \subset \ldots \subset C_m  $,
 let \emph{$X(B, B_1, \ldots, B_l : C_{0}, \ldots, C_m)$} be the poset of all free splittings $S$ such that $\mcV(S) \in\nolinebreak \{[B], [B_1], \ldots, [B_l]\}$ and $C_i|S$ is a subgraph of $S$ for every $0 \leq i \leq m$. We want to use induction on $l$ to show that this poset is contractible.

We start with the case $l=0$.
Let $\mcD$ be the Outer space of $A$ relative to $\{[B]\}$ as defined in \cite{GL:outerspacefree}. It can be seen as a subspace of the space of all non-trivial metric simplicial $A$-trees.
In \cite{GL:Deformationspacestrees}, Guirardel and Levitt show that its closure $\bar{\mcD}$ in this space is contractible. To do so, they use Skora folding paths to define a map $\rho:\bar{\mcD}\times [0,\infty ] \to \bar{\mcD}$.
The map $\rho$ depends on the choice of a ``base point'' $T_0\in \mcD$ and is defined such that for all $T$, one has $\rho(T,0) = T$ whereas $\rho(T,\infty)$ is contained in a contractible subspace (a closed simplex of $\bar{\mcD}$) containing $T_0$.
In \cite[Lemma 5.5]{BG:Homotopytypecomplex}, it is shown that for an appropriate choice of $T_0$ (namely for a tree in $X(B:C_0,...,C_m)$ with a minimal number of edge orbits), the map $\rho$ restricts to a continuous map on $\real{X(B:C_0,...,C_m)}$.
There, the argument is formulated for the case where $A$ a free group, but it applies verbatim in our setting as all the results in \cite{GL:Deformationspacestrees} are formulated in this more general situation anyway.

For $l>0$, assume that by induction, we know that the posets
\begin{gather*}
X_{l-1} \coloneqq X(B, B_1, \ldots, B_{l-1} : C_0, \ldots, C_m),\hspace{0.5 cm} X_l \coloneqq X(B_l:C_0, \ldots, C_m) \\
 \text{and} \hspace{0.5 cm} X_{l-1,l} \coloneqq X(B, B_1, \ldots, B_{l-1} : B_l, C_0, \ldots, C_m)
\end{gather*}
are contractible. The poset $X(B, B_1 \ldots , B_l : C_0, \ldots, C_m)$ is the union of $X_{l-1}$ and $X_{l}$. Furthermore, an element $S \in X_{l-1}$ collapses to some $S'\in X_l$ if and only if $S\in X_{l-1,l}$. It follows that 
\begin{equation*}
\real{X(B, B_1 \ldots , B_l : C_0, \ldots, C_m)}  = \real{X_{l-1}} \cup_{\real{X_{l,l-1}}} \real{X_l}
\end{equation*}
is contractible. In particular, $X(B, B_1 \ldots , B_l : A)$ is contractible.
\\

Now fix $[B] \in \FRfactor$. Each simplex $\sigma$ in the order complex $\Delta(\FreeOne(B))$ is given by a sequence $S_1, \ldots , S_l$ where every $S_i$ is a free splitting in $\FreeOne$ that collapses to $S_{i+1}$. It follows that the vertex groups of these splittings form a chain $\mcV(S_1) \leq \ldots \leq \mcV(S_l)$ such that $[B]\leq \mcV(S_i)$ for all $i$.
Hence, the simplex  $\sigma$ is contained in the order complex of $X(B, \mcV(S_1), \ldots  , \mcV(S_l):A)$. Consequently, the realisation $\real{\FreeOne(B)}$ can be written as a union
\begin{equation}
\label{eq:union_contractible}
\real{\FreeOne(B)}=\bigcup_{B \subset B_1 \subset \ldots  \subset B_l} \real{X(B, B_1, \ldots  , B_l :A)}.
\end{equation}
By the arguments above, all of these sets are contractible. Also, one has
\begin{equation*}
\real{X(B, B_1, \ldots  , B_l :A)} \cap \real{X(B, C_1, \ldots  , C_m :A)}= \real{X(B, D_1, \ldots  , D_k :A)},
\end{equation*}
where $[B] < [D_1] < \ldots  < [D_k]$ is the longest common subchain of $[B] < [B_1]<  \ldots  < [B_l]$ and $[A]< [C_1]<  \ldots  <[C_m]$. This implies that finite intersections of the sets appearing on the right hand side of \cref{eq:union_contractible} are contractible. By the nerve lemma (see \cite[Theorem 10.6]{Bjo:Topologicalmethods}), $\real{\FreeOne(B)}$ is homotopy equivalent to the nerve of this covering. This is contractible as all of these sets intersect non-trivially (they all contain $\real{X(B:A)}$).
\end{proof}

\begin{proposition}
There is a homotopy equivalence $\FreeOne \simeq \FRfactor$.
\end{proposition}
\begin{proof}
Assigning to each splitting $S \in \FreeOne$ the conjugacy class $\mcV(S)$ of its non-trivial vertex stabiliser defines a poset map $f\colon  \FreeOne \to \FRfactor^{op}$.
 As there is a natural isomorphism of the order complexes $\Delta(\FRfactor^{op})\cong \Delta(\FRfactor)$, we will interpret $f$ as an order-inverting map $f\colon  \FreeOne \to\nolinebreak \FRfactor$.
Now for any $B \in \FRfactor$, the fibre $f^{-1}(\FRfactor_{\geq B})$ is equal to the poset $\FreeOne(B)$ which is contractible by \cref{contractibility of relative free splitting complex}. The claim follows from \cref{Quillen fibre lemma contractibility}.
\end{proof}

\subsubsection{Homotopy type of relative free factor complexes}
In order to study the homotopy type of $\FreeOne$, we ``thicken it up'' by elements from $L=L(A,\mcA)$, the spine of Outer space of $A$ relative to $\mcA$.
\begin{definition}
Let $Y$ be the subposet of the product $L\times \FreeOne$ consisting of all pairs $(G,S)$ such that $S = G/H$ is obtained from $G$ by collapsing a proper core subgraph $H$ and let $p_1\colon Y\to L$ and $p_2\colon Y\to \FreeOne$ be the natural projection maps.
\end{definition}
By analysing the maps $p_1$ and $p_2$, we now want to show that $\FRfactor$ is spherical. This closely follows \cite[Section 7]{BG:Homotopytypecomplex}.
We first deformation retract the fibres of $p_2$ to a simpler subposet:

\begin{lemma}
\label{fibre of p_2 retracts to maximal element}
\leavevmode
For all $S\in \FreeOne$, the fibre $p_2^{-1}(\FreeOne_{\geq S})$ deformation retracts to $p_2^{-1}(S)$.
\end{lemma}
\begin{proof}
Let $F\coloneqq p_2^{-1}(\FreeOne_{\geq S})$ and define $f\colon F \to p_2^{-1}(S)$ as follows: If $(G',S')$ is an element of $F $, there are collapse maps $G'\to S'$ and $S'\to S$. Concatenating these maps, we see that $S$ is obtained from $G'$ by collapsing a subgraph $H'\subset G'$. The subgraph can be written as the union of a (possibly trivial) forest $T'$ and a unique maximal core graph $\mathring{H}'$.
We set $f(G',S')\coloneqq (G'/T',S)$. As $S=(G'/T')/\mathring{H}'$, this is indeed an element of $p_2^{-1}(S)$.
Also if $(G'',S'')\geq (G',S')$ in $F $, we have $c(T'')\supseteq T' $ which implies $G''/T'' \geq G'/T'$.
Consequently $f\colon F \to p_2^{-1}(S)$ is a well-defined, monotone poset map restricting to the identity on $p_2^{-1}(S)$. It follows from \cref{homotopic poset maps} that this defines a deformation retraction.
\end{proof}

Hence, instead of studying arbitrary fibres, it suffices to consider the preimages of single vertices.

\begin{lemma}
\label{fibre of p_2 contractible}
For all $S\in \FreeOne$, the preimage $p_2^{-1}(S)$ is contractible.
\end{lemma}
\begin{proof}
Let $[B]\coloneqq \mcV(S)$. Every element in $p_2^{-1}(S)$ is given by a pair $(G,S)$ such that $H\coloneqq B|G$ is a subgraph of $G$ and $S=G/H$. Forgetting the (constant) second coordinate, we can interpret these as elements of the Outer space of $A$ relative to $\mcA$. Let $X$ be the subspace of this Outer space that is given by all open simplices containing an element of $p_2^{-1}(S)$. Then $p_2^{-1}(S)$ is a deformation retract of $X$. In \cite[Proposition 7.3.1 and Proposition 7.2]{BG:Homotopytypecomplex}, Skora folds are used to show that $X$ is contractible. The proof in \cite{BG:Homotopytypecomplex} is formulated for the case where $A$ is free, but it applies here as well because it only uses the ideas of Guirdel--Levitt \cite{GL:Deformationspacestrees}, which hold true in the generality needed in our setting.
\end{proof}

In particular, these fibres are all contractible, so by \cref{Quillen fibre lemma contractibility}, we have:

\begin{corollary}
\label{fibres of p_2 contractible}
The map $p_2\colon Y\to \FreeOne$ is a homotopy equivalence.
\end{corollary}

We now turn to $p_1$ and show that its fibres are highly connected as well.

\begin{definition}
For a labelled graph  $(G,l)$, let $\crCore(G,l)$ denote the poset of all proper core subgraphs of $(G,l)$, where the partial order is given by inclusion of subgraphs.
\end{definition}

\begin{lemma}
\label{fibre of p_1 and poset of core subgraphs}
Let $G \in L$, and let $(G,l)$ denote the induced structure of a labelled graph. Then the fibre $p_1^{-1}(L_{\leq G})$ is homotopy equivalent to $\crCore(G,l)$.
\end{lemma}
\begin{proof}
Each element of $p_1^{-1}(L_{\leq G})$ consists of a pair $(G',S')$ where $G'\leq G$ in $L$ and $S'\in\FreeOne$ is obtained from $G'$ by collapsing a proper core subgraph $H'$.
As $G'$ is obtained from $G$ by collapsing a forest, $H\coloneqq \pi_{G'}(H')|G$ 
is a subgraph of $G$.

The collapse $G\to G'$ induces a collapse $G/H \to G'/H'=S'$. Hence, we get a monotone poset map
\begin{align*}
f\colon p_1^{-1}(L_{\leq G})&\to p_1^{-1}(G)\\
(G',S')&\mapsto (G,G/H)
\end{align*}
restricting to the identity on $p_1^{-1}(G)\subseteq p_1^{-1}(L_{\leq G})$.
Again \cref{homotopic poset maps} implies that $f$ defines a deformation retraction.

For every proper core subgraph $H$ of $G$, the pair $(G,G/H)$ forms an element of $p_1^{-1}(G)$.
Also, if $H$ and $H'$ are proper core subgraphs of $G$, one has $G/H\geq \nolinebreak G/H'$ in $\FreeOne$ if and only if $H\leq H'$ in $\crCore(G,l)$. Hence, the fibre $p_1^{-1}(G)$ can be identified with $\crCore(G,l)^{op}$. Because $\real{\crCore(G,l)^{op}}\cong \real{\crCore(G,l)}$, this finishes the proof.
\end{proof}

\begin{theorem}
\label{homotopy type poset of core subgraphs}
The poset $\crCore(G,l)$ is $(n-2)$-spherical.
\end{theorem}

We postpone the proof of this result until \cref{section posets of subgraphs} and first note the following corollary:

\begin{corollary}
\label{product poset connected}
The poset $Y$ is $(n-3)$-connected.
\end{corollary}
\begin{proof}
The projection $p_1\colon Y \to L$ is a map from $Y$ to the contractible poset $L$. By \cref{fibre of p_1 and poset of core subgraphs} and \cref{homotopy type poset of core subgraphs}, the fibres of this map are $(n-3)$-connected. Hence, the result follows from \cref{Quillen fibre lemma k-connected}.
\end{proof}

The main result of this section, which was stated as Theorem \ref{introduction free factor complex} in the introduction, is now an easy consequence of the last corollaries:
\begin{theorem}
\label{relative free factor complexes}
The free factor complex of $A= F_n \ast A_1 \ast \cdots \ast A_k$ relative to $\mcA=\set{A_1, \ldots, A_k}$ is homotopy equivalent to a wedge of $(n-2)$-spheres.
\end{theorem}
\begin{proof}
By \cref{fibres of p_2 contractible}, there is a homotopy equivalence $\FRfactor \simeq Y$. By \cref{product poset connected}, the poset $Y$ is $(n-3)$-connected. As $\FRfactor$ is $(n-2)$-dimensional, the claim follows.
\end{proof}

\subsubsection{Cohen--Macaulayness}

The relative formulations allow us to deduce that $\FRfactor$ or equivalently $\CC{O}{\mcP}$ is even Cohen--Macaulay:
\begin{theorem}
\label{CM free factor complex}
The coset complex $\CC{O}{\mcP}$ is homotopy Cohen--Macaulay.
\end{theorem}
\begin{proof}
By \cref{isomorphism relative free factor complex}, we have to show that the link of every $s$-simplex $\sigma =
[B_0]\leq \ldots \leq [B_s]$
in $\Delta(\mcF)$ is $(n-s-3)$-spherical. However, the link of this simplex is by definition given by the following join of posets
\begin{equation*}
\lk (\sigma) = \FRfactor_{< [B_0]} \ast ([B_0], [B_1]) \ast \ldots \ast ([B_{s-1}], [B_s]) \ast \FRfactor_{> [B_s]}.
\end{equation*}
As above, each $B_i$ can be written in the form
\begin{equation*}
B_i =  D_i \ast A_1^{a_1} \ast \ldots \ast A_k^{a_k},
\end{equation*}
where $D_i$ is a free group of rank $n-\corank [B_i]$. Using malnormality of free factors, it follows that two subgroups of a free factor $B$ of $A$ are conjugate in $A$ if and only they are conjugate in $B$ (see \cite[Lemma 2.1]{HM:Relativefreesplitting}).It follows that there are isomorphisms
\begin{align*}
\FRfactor_{< [B_0]} &\cong \mcF(B_0, \set{A_1^{a_1},\ldots , A_k^{a_k}}) , \\  
([B_i], [B_{i+1}]) &\cong \mcF(B_{i+1}, \set{B_i}) , \\
\FRfactor_{> [B_s]} &\cong \mcF(A, \set{B_s}). 
\end{align*}
The result now follows from \cref{homotopy type of wedges and joins} and \cref{relative free factor complexes}.
\end{proof}

\subsubsection{Posets of subgraphs}
\label{section posets of subgraphs}
In this section, we prove \cref{homotopy type poset of core subgraphs} by studying posets of subgraphs of labelled graphs. The results we obtain gene\-ralise \cite[Section 4.2]{BG:Homotopytypecomplex} and we closely follow the structure of the proofs there.

In what follows, all graphs are allowed to have loops and multiple edges. For a graph $G$, let $V(G)$ denote the set of its vertices and $E(G)$  the set of its edges. If $e\in E(G)$ is an edge, then $G-e$ is defined to be the graph obtained from $G$ by removing $e$ and $G/e$ is obtained by collapsing $e$ and identifying its two endpoints to a new vertex $v_e$. For a labelled graph $(G,l)$ (see \cref{def:labelled&core}), there are canonical labellings $\set{1,\ldots , k}\to G-e$ and $\set{1,\ldots , k}\to G/e$ that will be denoted by $l$ as well.
A graph is called a \emph{tree} if it is contractible.

\begin{definition}
For a labelled graph  $(G,l)$, let $\crX(G,l)$ be the poset of all connected subgraphs of $G$ which are not trees, contain all the labelled vertices and whose fundamental group is strictly contained in $\pi_1(G)$.
\end{definition}

The following lemma allows us to replace $\crCore(G,l)$ with $\crX(G,l)$. This bigger poset will be easier to handle for the inductive arguments that we want to use.

\begin{lemma}
\label{lem:cSubgraphstocrCores}
$\crX(G,l)$ deformation retracts to $\crCore(G,l)$.
\end{lemma}
\begin{proof}
By restricting the labelling, every $H\in \crX(G,l)$ can be seen as a labelled graph $(H,l)$. It contains contains a unique maximal core subgraph $(\mathring{H},l)$. Also, if $H_1\leq H_2$ in $\crX(G,l)$, one has $(\mathring{H_1},l)\subseteq (\mathring{H_2},l)$. Hence, sending $(H,l)$ to $(\mathring{H},l)$ defines a poset map $f\colon \crX(G,l)\to \crCore(G,l)$ that restricts to the identity on $\crCore(G,l)$. The claim now follows from \cref{homotopic poset maps}.
\end{proof}

An edge $e\in E(G)$ is called \emph{separating} if $G-e$ is disconnected; in particular, we consider edges adjacent to vertices of valence one to be separating. 

\begin{lemma}
\label{lemma valence-1 homotopies}
Let $(G,l)$ be a labelled graph where $G$ is finite and connected. Let $e\in E(G)$ be an edge that is not a loop and set
\begin{equation*}
Y_e\coloneqq \set{H \in X(G,l) \mid H\cup \{e\} \text{ connected and } \pi_1(H\cup \{e\}) = \pi_1(G)}.
\end{equation*}
Then $\crX(G,l)\bsl Y_e \simeq \crX(G/e,l)$.

Furthermore, if $e$ is separating, then $Y_e$ is empty, so $\crX(G,l) \simeq \crX(G/e,l)$.
\end{lemma}
\begin{proof}
Whenever $H\in \crX(G,l)$, the edges in $E(H)\bsl\set{e}$ form a connected subgraph of $G/e$ that will be denoted by $H/e$. It contains all labelled vertices of $(G/e,l)$ and has non-trivial fundamental group. 

If $H$ is not in $Y_e$, then either $e$ is not adjacent to $H$ and hence $\pi_1(H/e) \cong \pi_1(H) $, or $\pi_1(H/e)\leq \pi_1(H\cup \{e\}/e) \cong \pi_1(H\cup \{e\})$. In either case, $\pi_1(H/e)$ is a proper subgroup of $\pi_1(G/e) \cong \pi_1(G)$.
Consequently, we get a poset map
\begin{align*}
 f\colon  \crX(G,l)\bsl Y_e &\to \crX(G/e,l)\\
H&\mapsto H/e.
\end{align*}

On the other hand, if $K\in \crX(G/e,l)$ contains the vertex $v_e$ to which $e$ was collapsed, it is easy to see that $K\cup\set{e}$ is an element of $\crX(G,l)\bsl Y_e$.
This allows us to define a poset map
\begin{align*}
\label{map undoing a collapse}
 g\colon  \crX(G/e,l)&\to \crX(G,l)\bsl Y_e\\
K&\mapsto
\begin{cases}
K\cup\set{e}&,\,v_e\in V(K),\\
K&\text{, else.}
\end{cases}
\end{align*}
One has $ g\circ f(H)\supseteq H$ and $ f\circ g(K)=K$, so using \cref{Quillen homotopic poset maps}, these two posets are homotopy equivalent, which proves the first part of the statement.

For the second part, note that if $e$ is separating and $H\in \crX(G,l)$ such that $H\cup \{e\}$ is connected, then either $e$ is contained in $H$ or $e$ has valence one in $H\cup \{e\}$. In either case, we have $\pi_1(H\cup \{e\})=\pi_1(H)\not=\pi_1(G)$.
\end{proof}

To prove the following result, we apply an argument similar to the one used in \cite[Proposition 2.2]{Vog:Localstructuresome}. (Note that in \cite{Vog:Localstructuresome}, being $n$-spherical is only defined for $n$-dimensional posets.)

\begin{proposition}
\label{homotopy type poset of connected subgraphs}
Let $(G,l)$ be a labelled graph where $G$ is finite, connected  and has fundamental group of rank $\rank\geq 2$. Then $\crX(G,l)$ is $(n-2)$-spherical.
\end{proposition}
\begin{proof}
If $e\in E(G)$ is separating, then by \cref{lemma valence-1 homotopies}, we have $\crX(G,l) \simeq \crX(G/e,l)$. As $G/e$ has one edge less than $G$, we can apply induction to assume that $G$ does not have any separating edges. 

We do induction on $\rank$ and start with the case $\rank=2$. By \cref{lem:cSubgraphstocrCores}, it suffices to show that $\crCore(G,l)$ is homotopy equivalent to a wedge of $0$-spheres, i.e. a disjoint union of points. To see this, let $H\in \crCore(G,l)$. As $1< \pi_1(H) < \pi_1(G)$, the fundamental group of $H$ is infinite cyclic. Let $e\in H$ be an edge of $H$. We distinguish between the two cases where $e$ is non-separating or separating in $H$. If $e$ is non-separating, then $H - e$ has trivial fundamental group while if $e$ is separating, $H - e$ has two connected components both of which either have non-trivial fundamental group or contain at least one labelled vertex. In both cases, no $K\in \crCore(G,l)$ can be contained in $H - e$. Hence, the order complex of $\crCore(G,l)$ does not contain any simplex of dimension greater than zero which proves the claim.

Now let $\rank>2$. If every edge of $G$ is a loop, $G$ is a rose with $\rank$ petals and every proper non-empty subset of $E(G)$ forms an element of $\crX(G,l)$. In this case, the order complex of $\crX(G,l)$ is given by the set of all proper faces of a simplex of dimension $\rank-1$ whose vertices are in 1-to-1 correspondence with the edges of $G$ and hence is homotopy equivalent to an $(n-2)$-sphere.

So assume that $G$ has an edge $e$ that is not a loop. 
As we assumed that $e$ is non-separating, $G-e$ is a connected graph having the same number of vertices as $G$ and one edge less. This implies that $\rk(\pi_1(G-e))=\rank-1$. Collapsing separating edges and using Lemma \ref{lemma valence-1 homotopies}, we see that $\crX(G-e, l)\simeq \crX(G',l)$ where $G'$ has the same rank as $G-e$, at most as many edges and no separating edges. 
Hence, $\crX(G-e,l)\simeq \crX(G',l)$ is by induction homotopy equivalent to a wedge of $(\rank-3)$-spheres.

$\real{\crX(G,l)}$ is obtained from $\real{\crX(G,l)\bsl \set{G-e}}$ by attaching the star of $G-\nolinebreak e$ along its link. The link of $G-e$ in $\real{\crX(G,l)}$ is isomorphic to $\real{\crX(G-e,l)}$ and its star is contractible. Gluing a contractible set to an $(\rank-2)$-spherical complex along an $(\rank-3)$-spherical subcomplex results in an $(\rank-2)$-spherical complex, so the claim follows (see e.g.~\cite[Lemma 6.3]{BSV:bordificationouterspace}).
\end{proof}

\begin{proof}[Proof of \cref{homotopy type poset of core subgraphs}]
That $\crCore(G,l)$ is $(n-2)$-spherical is an immediate consequence of \cref{lem:cSubgraphstocrCores} and \cref{homotopy type poset of connected subgraphs}.
\end{proof}

\section{Relative automorphism groups of RAAGs}
In this section, we examine relative automorphism groups of right-angled Artin groups. These groups were studied in detail in  \cite{DW:Relativeautomorphismgroups} and many of the results here are either taken from the work of Day--Wade or build on their ideas.
For an overview about other literature on relative automorphism groups, see \cite[Section 6.1]{DW:Relativeautomorphismgroups}.
In this article, such relative automorphism groups occur in two ways: On the one hand, they arise as the images and kernels of restriction and projection homomorphisms, which in turn play an important role for the inductive procedure of Day--Wade; on the other hand, the parabolic subgroups we will define in \cref{section spherical complex for Out(RAAG)} are themselves relative automorphism groups of RAAGs.
For the purpose of this text, the present section mostly serves as a toolbox for the inductive proof of Theorem \ref{introduction homotopy type} in \cref{section spherical complex for Out(RAAG)}. Its main goals are to collect all the results from \cite{DW:Relativeautomorphismgroups} that we will need afterwards, to adapt them to our purposes and, maybe most importantly, to set up the language we will use later on.

\label{section ROARs}
\paragraph{Standing assumption}
From now on, all graphs that we consider will be finite and simplicial, i.e. without loops or multiple edges. To emphasise this difference to \cref{section building and relative free factor complex}, they will be denoted by Greek letters.

\subsection{RAAGs and their automorphism groups}
\label{sec RAAGs and their automorphism groups}
\paragraph{Subgraphs, links and stars}
In contrast to \cref{section building and relative free factor complex}, if we talk about a subgraph $\Delta$ of a graph $\Gamma$, we will from now on always mean a \emph{full} subgraph, i.e. if two vertices $v,w\in V(\Delta)$ are connected by an edge in $\Gamma$, they are connected in $\Delta$ as well. A full subgraph of $\Gamma$ can also be seen as a subset of the vertex set $V(\Gamma)$; we will often take this point of view, identify $\Delta$ with $V(\Delta)$ and write $\Delta\subseteq \Gamma$ or $\Delta\subset \Gamma$ if we want to emphasise that $\Delta$ is a proper subgraph of $\Gamma$.

Given a vertex $v\in V(\Gamma)$, the \emph{link} $\lk(v)$ of $v$ is the subgraph of $\Gamma$ consisting of all the vertices that are adjacent to $v$. The \emph{star} $\st(v)$ of $v$ is the subgraph of $\Gamma$ with vertex set $\set{v}\cup \lk(v)$. We also write $\lk_\Gamma(v)$ or $\st_\Gamma(v)$ if we want to distinguish between links and stars in different graphs.

\paragraph{RAAGs and special subgroups}
Given a graph $\Gamma$, the associated \emph{right-angled Artin group} --- abbreviated as RAAG --- $A_\Gamma$ is defined to be the group generated by the set $V(\Gamma)$ subject to the relations $[v,w]=1$ for all $v,w\in V(\Gamma)$ which are adjacent to each other.

Given any subgraph $\Delta\subseteq \Gamma$, the inclusion $V(\Delta)\to V(\Gamma)$ induces an injective homomorphism $A_\Delta\hookrightarrow A_\Gamma$. This allows us to interpret $A_\Delta$ as a subgroup of $A_\Gamma$. Subgroups of this type are called \emph{special subgroups} of $A_\Gamma$.

\paragraph{The standard ordering and its equivalence classes}
There is a so-called \emph{standard ordering} on the vertex set $V(\Gamma)$ that is the partial pre-order given by $v\leq w$ if and only if $\lk(v)\subseteq \st(w)$. The induced equivalence relation of this partial pre-order will be denoted by $\sim$, i.e. $v\sim w$ if and only if $v\leq w$ and $w\leq v$. The equivalence class of $v$ will be denoted by $[v]$. The standard ordering induces a partial order on the equivalence classes where we say $[v]\leq [w]$ if $v\leq w$ (this does not depend on the choice of representatives).
If two equivalent vertices $v\sim w$ are adjacent, it follows that the vertices from their equivalence class $[v]$ form a complete subgraph of $\Gamma$. In this case, the special subgroup $A_{[v]}$ is isomorphic to $\mathbb{Z}^{|[v]|}$ and we call $[v]$ an \emph{abelian} equivalence class. If on the other hand $[v]$ does not contain any pair of adjacent vertices, it can be seen as discrete subgraph of $\Gamma$. In this case, we call $[v]$ a \emph{free} equivalence class because $A_{[v]}$ is isomorphic to the free group $F_{|[v]|}$.
For more details about this ordering and the equivalence relation, see \cite{CV:Finitenesspropertiesautomorphism}.

\paragraph{Automorphisms of RAAGs}
Let $\Aut{A_\Gamma}$ and $\Out{A_\Gamma}$ denote the automorphism group and the group of outer autmorphisms of $A_\Gamma$, respectively.  By the work of Servatius \cite{Ser:Automorphismsgraphgroups} and Laurence \cite{Lau:generatingsetautomorphism}, the group $\Aut{A_\Gamma}$ is generated by the following automorphisms:
\begin{itemize}
\item \emph{Graph automorphisms.} Any automorphism of the graph $\Gamma$ gives rise to an automorphism of $A_\Gamma$ by permuting the generators of the RAAG.
\item \emph{Inversions.} Let $v\in V(\Gamma)$. The map sending $v$ to $v^{-1}$ and fixing all the other generators induces an automorphism of $\AG$. It is called an \emph{inversion} and denoted by $\iota_v$. 
\item \emph{Transvections.} Let $v,w\in V(\Gamma)$ with $v\leq w$. The \emph{transvection} $\rho_v^w$ is the automorphism of $\AG$ induced by sending $v$ to $vw$ and fixing all the other generators. We call $w$ the \emph{acting letter} of $\rho_v^w$.
\item \emph{Partial conjugations.} Let $v\in V(\Gamma)$ and $K$ a union of connected components of $\Gamma\bsl \st(v)$. The map sending every vertex $w$ of $K$ to $vwv^{-1}$ and fixing the remaining generators induces an automorphism $\pi_K^v$ of $\AG$ and is called a \emph{partial conjugation}. We call $v$ the \emph{acting letter} of $\pi_K^v$.
\end{itemize}

We will use the same notation to denote the images of these automorphisms in $\Out{\AG}$ and call these (outer) automorphisms the \emph{Laurence generators} of $\Aut{\AG}$ or $\Out{\AG}$, respectively.

The subgroup of $\Out{\AG}$ generated by all inversions, transvections and partial conjugations is denoted by $\Outo$. It is called the \emph{pure outer automorphism group} of $\AG$ and has finite index in $\Out{\AG}$. If $\AG$ is equal to $\mathbb{Z}^n$ or $F_n$, we have that $\Outo=\Out{\AG}$.

\subsection{Generators of relative automorphism groups}
\label{sec generators of ROARs}
Recall that for a group $G$ and families of subgroups $\mcG$ and $\mcH$, the group $\Out{G; \mcG, \mcH^t} $ is defined as the subgroup of $\Out{G}$ consisting of all elements stabilising each $H\in \mcG$ and acting trivially on each $H\in \mcH$ (see \cref{section relative automorphism groups}).

Given a pair $(\mcG,\mcH)$ of families of special subgroups of $A_\Gamma$, we define 
\begin{equation*}
\Roar{\AG}{\mcG, \mcH^t} \coloneqq \Out{\AG; \mcG, \mcH^t} \cap\Outo 
\end{equation*}
as the intersection of $\Out{\AG; \mcG, \mcH^t}$ with $\Outo$.  Building on the work of Laurence, Day--Wade show:

\begin{theorem}[{\cite[Theorem D]{DW:Relativeautomorphismgroups}}]
\label{generators ROARs}
If $\mcG$ and $\mcH$ are families of special subgroups of $\AG$, the group $\Roar{\AG}{\mcG, \mcH^t}$ is generated by the set of all inversions, transvections and partial conjugations of $\Out{\AG}$ it contains.
\end{theorem}

In order to prove this, Day and Wade give a description of the Laurence generators contained in such a relative automorphism group. To state it, we first need to set up the terminology developed in their article.

\paragraph{$\mcG$-components and $\mcG$-ordering}
Let $\mcG$ be a family of proper special subgroups of $\AG$. We say that $v,w\in V(\AG)$ are \emph{$\mcG$-adjacent} if $v$ is adjacent to $w$ or if there is some $A_\Delta\in \mcG$ such that $v,w \in \Delta$. A subgraph $\Delta\subseteq \Gamma$ is \emph{$\mcG$-connected} if for all $v,w\in \Delta$, there is a sequence of vertices in $\Delta$ which starts with $v$, ends with $w$ and such that each of its vertices is $\mcG$-adjacent to the next one. A maximal $\mcG$-connected subgraph of $\Gamma$ is called a \emph{$\mcG$-component}.

The \emph{$\mcG$-ordering} $\leq_\mcG$ on $V(\Gamma)$ is the partial pre-order defined by saying that $v\leq_\mcG w$ if and only if $v\leq w$ and for all $A_\Delta\in \mcG$, if $v\in\Delta$, one has $w\in \Delta$. The equivalence relation of this pre-order is denoted by $\sim_\mcG$, its equivalence classes by $[\cdot]_\mcG$.

Note that in the case where $\mcG=\emptyset$, a $\mcG$-component of $\Gamma$ is just a connected component and the $\mcG$-ordering is the standard ordering on $V(\Gamma)$.
\\

For $v\in V(\Gamma)$, let $\mcG^v\coloneqq\set{A_\Delta\in \mcG \mid v\not\in \Delta}$. It is easy to see that every $\mcG^v$-component of $\Gamma\bsl\st(v)$ is a union of connected components of $\Gamma\bsl\st(v)$.
Suppose that $\mcH$ is a family of special subgroups of $\AG$. The \emph{power set of $\mcH$}, denoted by $P(\mcH)$, is defined as the set of all special subgroups $A_\Delta\leq \AG$ which are contained in some element of $\mcH$. 

\begin{lemma}[{\cite[Proposition 3.9]{DW:Relativeautomorphismgroups}}]
\label{Laurence generator containment criterion}
Let $\mcG$ and $\mcH$ be families of special subgroups of $\AG$ such that $\mcG$ contains $P(\mcH)$. Let $v,w \in V(\Gamma)$ and let $K$ be a union of connected components of $\Gamma\bsl \st(v)$.
Then:
\begin{itemize}
\item The inversion $\iota_v$ is contained in $\Roar{\AG}{\mcG, \mcH^t}$ if and only if there is no subgroup $A_\Delta\in \mcH$ with $v\in \Delta$.
\item The transvection $\rho_v^w$ is contained in $\Roar{\AG}{\mcG, \mcH^t}$ if and only if $v\leq_\mcG w$.
\item The partial conjugation $\pi_K^v$ is contained in $\Roar{\AG}{\mcG, \mcH^t}$ if and only if $K$ is a union of $\mcG^v$-components of $\Gamma\bsl\st(v)$.
\end{itemize}
\end{lemma}
Note that it imposes no great restriction to assume that the power set of $\mcH$ be contained in $\mcG$ because for any families $\mcG$ and $ \mcH$ of special subgroups, one has
\begin{equation*}
\Roar{\AG}{\mcG, \mcH^t} = \Roar{\AG}{\mcG\cup P(\mcH), \mcH^t}
\end{equation*}
(see \cite[Lemma 3.8]{DW:Relativeautomorphismgroups}).

The next result is the key ingredient for the proof of \cref{Laurence generator containment criterion} in \cite{DW:Relativeautomorphismgroups}. We include it here because it will allow us a more convenient description of the parabolic subgroups that we will study later on.

\begin{lemma}[{\cite[Lemma 2.2]{DW:Relativeautomorphismgroups}}]
\label{Laurence generator stabilise criterion}
Let $A_\Delta$ be a special subgroup of $\AG$.
Let $v,w \in V(\Gamma)$ and let $K$ be a union of connected components of $\Gamma\bsl \st(v)$.
Then:
\begin{itemize}
\item The inversion $\iota_v$ acts trivially on $A_\Delta$ if and only if $v\not \in \Delta$; it always stabilises $A_\Delta$.
\item The transvection $\rho_v^w$ acts trivially on $A_\Delta$ if and only if $v\not \in \Delta$; it stabilises $A_\Delta$ if it acts trivially on it or $w\in \Delta$.
\item The partial conjugation $\pi_K^v$ acts trivially on $A_\Delta$ if and only if 
\begin{equation*}
K\cap \Delta=\emptyset \text{ or } \Delta\bsl \st(x)\subseteq K;
\end{equation*}
it stabilises $A_\Delta$ if it acts trivially on it or $w\in \Delta$.
\end{itemize}
\end{lemma}

\subsection{Restriction and projection homomorphisms}
Let $O$ be a subgroup of $\Out{\AG}$. If the special subgroup $A_\Delta\leq A_\Gamma$ is stabilised by $O$, there is a \emph{restriction homomorphism}
\begin{equation*}
R_\Delta\colon  O \to \Out{A_\Delta},
\end{equation*}
where $R_\Delta(\Phi)$ is the outer automorphism given by taking a representative $\phi\in \Phi$ that sends $A_\Delta$ to itself and restricting it to $A_\Delta$.
If the normal subgroup $\langle \langle A_\Delta \rangle\rangle$ generated by $A_\Delta$ is stabilised by $O$, there is a \emph{projection homomorphism}
\begin{equation*}
P_{\Gamma\bsl\Delta}\colon  O \to \Out{A_{\Gamma\bsl \Delta}},
\end{equation*}
which is induced by the quotient map
\begin{equation*}
\AG \to \AG/\langle \langle A_\Delta \rangle\rangle \cong A_{\Gamma\bsl \Delta}.
\end{equation*}

Restriction and projection maps were first defined in \cite{CCV:Automorphisms2dimensional} and have since become an important tool for studying automorphism groups of RAAGs via inductive arguments.

\subsubsection{Generators of image and kernel}
Day--Wade obtained a complete description of the image and kernel of restriction homomorphisms.
Again let $\mcG$ and $\mcH$ be families of special subgroups of $A_\Gamma$. We say that $\mcG$ is \emph{saturated with respect to $(\mcG,\mcH)$}, if it contains every proper special subgroup stabilised by $\Roar{\AG}{\mcG, \mcH^t}$. Given a special subgroup $A_\Delta\leq A_\Gamma$, set
\begin{equation*}
\mcG_\Delta\coloneqq \set{A_{\Delta\cap \Theta} \mid A_\Theta\in \mcG}.
\end{equation*}
We define $\mcH_\Delta$ analogously.

\begin{theorem}[{\cite[Theorem E]{DW:Relativeautomorphismgroups}}]
\label{image and kernel of restriction homomorphism}
Let $\mcG$ be saturated with respect to $(\mcG,\mcH)$ and let $A_\Delta\in \mcG$. The restriction homomorphism
\begin{equation*}
R_\Delta\colon \Roar{\AG}{\mcG, \mcH^t} \to \Out{A_\Delta}
\end{equation*}
has image equal to
\begin{equation*}
\im R_\Delta= \Roar{A_\Delta}{\mcG_\Delta, \mcH_\Delta^t}
\end{equation*}
and kernel equal to
\begin{equation*}
\ker R_\Delta= \Roar{\AG}{\mcG, (\mcH\cup \set{A_\Delta})^t}.
\end{equation*}
\end{theorem}
It is not hard to see that both restriction and projection maps send each Laurence generator either to the identity or to a Laurence generator of the same type. For the proof of \cref{image and kernel of restriction homomorphism}, Day--Wade show that for restriction maps, a converse of this is true as well: Every Laurence generator in $\im R_\Delta$ is given as the restriction of a Laurence generator of $\Roar{\AG}{\mcG, \mcH^t}$.

In general, image and kernel of projection homomorphisms are more difficult to describe. However, we will only need to consider them in a special case:
The center $Z(\AG)$ of $A_\Gamma$ is generated by all vertices $z\in V(\Gamma)$ such that $\st(z)=\Gamma$. If $Z(\AG)$ is non-trivial, these vertices form an abelian equivalence class $Z\coloneqq [z]$ and $\Gamma$ can be written as a join $\Gamma=Z\ast\Delta$ where $\Delta=\Gamma\bsl Z$.
If we have a graph of this form, the center $Z(\AG)=A_Z$ is a normal subgroup which is stabilised by all of $\Out{\AG}$. Hence, there is a projection map
\begin{equation*}
P_{\Delta}\colon  \Out{\AG} \to \Out{A_{\Delta}}.
\end{equation*}
The image of this projection map can be described very similar to the the one of a restriction map. In fact, the situation in this special case is even easier as we do not even need to assume any kind of saturation for our families of special subgroups:
\begin{lemma}
\label{image of projection map}
Assume that $\Gamma$ can be decomposed as a join $\Gamma=Z\ast\Delta$ where $Z$ is a complete graph.
Let $\mcG$ and $\mcH$ be any two families of special subgroups of $\AG$ and let $A_Z\in \mcG$. The projection homomorphism
\begin{equation*}
P_\Delta\colon \Roar{\AG}{\mcG, \mcH^t} \to \Out{A_\Delta}
\end{equation*}
has image equal to
\begin{equation*}
\im P_\Delta= \Roar{A_\Delta}{\mcG_\Delta, \mcH_\Delta^t}.
\end{equation*}
\end{lemma}
\begin{proof}
The inclusion ``$\subseteq$'' follows immediately from the definitions.

For the other inclusion, we start by defining $\tilde{\mcG}\coloneqq \mcG\cup P(\mcH)$ as the union of $\mcG$ and the power set of $\mcH$. As observed above, we have 
\begin{equation*}
O\coloneqq \Roar{\AG}{\mcG, \mcH^t}=\Roar{\AG}{\tilde{\mcG}, \mcH^t}.
\end{equation*}
Furthermore, $\tilde{\mcG}_\Delta=\mcG_\Delta \cup P(\mcH_\Delta)$, so we also have 
\begin{equation}
\label{equation added power set}
O_\Delta \coloneqq \Roar{A_\Delta}{\mcG_\Delta, \mcH_\Delta^t}=\Roar{A_\Delta}{\tilde{\mcG}_\Delta, \mcH_\Delta^t}.
\end{equation}

By \cref{generators ROARs}, we know that $O_\Delta$ is generated by the inversions, transvections and partial conjugations it contains. Hence, it suffices to find a preimage under $P_\Delta$ for each of those generators. Combining \cref{equation added power set} with \cref{Laurence generator containment criterion}, we have a complete description of the generators in $O_\Delta$. In what follows, we will use this description to construct the preimages one generator at a time.

The inversion $\iota_v$ is contained in $O_\Delta$ if and only if $v\in\Delta$ and there is no $A_{\Delta'}\in \mcH_\Delta$ such that $v\in\Delta'$. However, this implies that there is no $A_{\Delta'}\in \mcH$ with $v\in\Delta'$, so the inversion at $v$ is an element of $O$. It will be denoted by $\bar{\iota}_v$ and gets mapped to $\iota_v$ under $P_\Delta$.

If one has a transvection $\rho_v^w\in O_\Delta$, \cref{Laurence generator containment criterion} implies that $v\leq_{\tilde{\mcG}_{\Delta}}w$, i.e. $lk_\Delta(v)\subseteq st_\Delta(w)$ and for each $A_{\Delta'}\in \tilde{\mcG}_\Delta$, one has that $v\in \Delta'$ implies $w\in\Delta'$. We want to show that $v\leq_{\tilde{\mcG}}w$. As $\Gamma$ is a join $Z\ast\Delta$, the link and star of $v$ and $w$ in $\Gamma$ are of the form
\begin{align*}
\lk_\Gamma(v)=\lk_\Delta(v)\cup Z, && \st_\Gamma(w)=\st_\Delta(w)\cup Z.
\end{align*}
In particular, $\lk_\Gamma(v)\subseteq \st_\Gamma(w)$. The vertex $v$ cannot be contained in any $\Delta'$ with $A_{\Delta'}\in P(\mcH)$ as this would imply $A_{\set{v}} \in P(\mcH_\Delta)\subseteq \tilde{\mcG}_\Delta$, contradicting the assumption that $v\leq_{\tilde{\mcG}_{\Delta}}w$.
Now take $A_{\Delta'}\in \mcG$ such that $v\in\Delta'$. 
 If $\Delta\subseteq\Delta'$, both $v$ and $w$ are contained in $\Delta'$. If on the other hand $\Delta\cap \Delta'$ is a proper subset of $\Delta$, one has $A_{\Delta\cap \Delta'}\in \mcG_{\Delta}\subseteq \tilde{\mcG}_{\Delta}$, so $w\in \Delta'$. 
It follows that $v\leq_{\tilde{\mcG}} w$, so the transvection multiplying $v$ by $w$ defines an element of $O$. It will be denoted by $\bar{\rho}_v^w\in O$ and is a preimage of $\rho_v^w$.

Again using \cref{Laurence generator containment criterion}, the partial conjugation $\pi_K^v$ is contained in $O_\Delta$ if and only if $v\in \Delta$ and $K$ is a union of $\tilde{\mcG}_\Delta^v$-components of $\Delta\bsl \st(v)$. 
We claim that every $\tilde{\mcG}_\Delta^v$-component $C$ of $\Delta\bsl \st_\Delta(v)$ is also a $\tilde{\mcG}^v$-component of $\Gamma\bsl \st_\Gamma(v)$. To see this, first recall that each element of $Z$ is connected to every vertex of $\Gamma$, so $\Gamma\bsl \st_\Gamma(v) = \Delta\bsl \st_\Delta(v)$. Furthermore, it follows right from the definitions that two vertices $x,y\in \Delta\bsl \st_\Delta(v)$ are $\tilde{\mcG}_\Delta$-adjacent in $\Delta\bsl \st_\Delta(v)$ if and only if they are $\tilde{\mcG}$-adjacent in $\Gamma\bsl \st_\Gamma(v)$. The claim follows and implies that the partial conjugation of $K$ by $v$ defines an element of $O$.
As above, it will be denoted by $\bar{\pi}_{K}^v$ and we note that it is a preimage of $\pi_{K}^v$.
\end{proof}
\subsubsection{Relative orderings in image and kernel}
\paragraph{Standing assumptions and notation} From now on and until the end of \cref{section ROARs}, let $O\coloneqq \Outo[A_\Gamma; \mcG, \mcH^t]$ where $\mcG$ and $\mcH$ are families of special subgroups of $\AG$ such that $\mcG$ is saturated with respect to $(\mcG,\mcH)$; note that saturation implies that $P(\mcH)\subseteq \mcG$. Set $\preceq\coloneqq \leq_\mcG$ to be the $\mcG$-ordering on $V(\Gamma)$.
\\

\begin{remark}
\label{remark relative orderings and saturation}
Given an arbitrary relative automorphism group, there might be several ways of ``representing'' this group by families of subgroups that are stabilised or acted trivially upon. I.e., we might have
\begin{equation*}
\Roar{\AG}{\mcG_1, \mcH_1^t} = \Roar{\AG}{\mcG_2, \mcH_2^t}
\end{equation*}
with $(\mcG_1, \mcH_1)\not = (\mcG_2, \mcH_2)$.
However, if in this situation, we have both $P(\mcH_1) \subseteq \nolinebreak\mcG_1$ and $P(\mcH_2) \subseteq \mcG_2$, the orderings $\leq_{\mcG_1}$ and $\leq_{\mcG_2}$ agree: By \cref{Laurence generator containment criterion}, for every $v,w\in V(\Gamma)$, there is a chain of equivalences
\begin{align*}
v \leq_{\mcG_1} w & \Leftrightarrow  \rho_v^w \in \Roar{\AG}{\mcG_1, \mcH_1^t} = \Roar{\AG}{\mcG_2, \mcH_2^t}\\
& \Leftrightarrow v \leq_{\mcG_2} w .
\end{align*}
In particular, the ordering $\leq_\mcG$ of $V(\AG)$ where $\mcG$ is saturated with respect to $(\mcG,\mcH)$ is an invariant of the group $\Roar{\AG}{\mcG, \mcH^t}$; it depends on the transvections contained in this group but not on any other choices.
\end{remark}

As mentioned above, a restriction homomorphism maps every transvection that is not contained in its kernel to a transvection of the same type. The con\-se\-quen\-ces for the relative ordering in image and kernel are as follows:

\begin{lemma}
\label{partial order in image and kernel}
Let $A_\Delta\in \mcG$ be a special subgroup that is stabilised by $O$ and let $R_\Delta$ denote the corresponding restriction homomorphism. If we write
\begin{equation*}
\im R_\Delta = \Roar{A_\Delta}{\mcG_{\im},\mcH_{\im}^t} \ \text{ and } \ker R_\Delta = \Roar{A_\Gamma}{\mcG_{\ker},\mcH_{\ker}^t}
\end{equation*}
with $\mcG_{\im}$ and $\mcG_{\ker}$ saturated with respect to $(\mcG_{\im},\mcH_{\im})$ and $(\mcG_{\ker},\mcH_{\ker})$, respectively, the following holds true:
\begin{enumerate}
\item
For $v,w\in \Delta$, one has $v\leq_{\mcG_{\im}} w$ if and only if $v\preceq w$.
\item
For $v\not= w\in V(\Gamma)$, one has $v\leq_{\mcG_{\ker}} w$ if and only if $v\in V(\Gamma)\bsl \Delta$ and $v\preceq w$.
\end{enumerate}
\end{lemma}
\begin{proof}
As $\mcG$ is saturated, we know that $\im R_\Delta = \Roar{A_\Delta}{\mcG_\Delta, \mcH_\Delta^t}$. For $v,w\in \nolinebreak \Delta$, \cite[Proposition 4.1]{DW:Relativeautomorphismgroups} shows that $v\leq_{\mcG_{\Delta}} w$ if and only if $v\preceq w$. 
Again because of the saturation of $\mcG$, we have $P(\mcH)\subseteq \mcG$. Hence, $P(\mcH_{\Delta})\subseteq \mcG_\Delta$. 
As in \cref{remark relative orderings and saturation}, it follows that $v\leq_{\mcG_\Delta} w$ if and only if $v\leq_{\mcG_{\im}} w$ for $\mcG_{\im}$ saturated with respect to $(\mcG_{\im},\mcH_{\im})$.

For the second point, we have $v\leq_{\mcG_{\ker}} w$ if and only if $\rho_v^w\in \ker R_\Delta$. This is the case if and only if $\rho_v^w$ is contained in $O$ and acts trivially on $A_\Delta$. The claim now follows from \cref{Laurence generator containment criterion} and \cref{Laurence generator stabilise criterion}.
\end{proof}

\subsubsection{Stabilisers in image and kernel}
\cref{image and kernel of restriction homomorphism} gives us a complete description of the image and kernel of a restriction map $R_\Delta \colon  \Roar{\AG}{\mcG, \mcH^t} \to \Out{A_\Delta}$ in the case where $\mcG$ is saturated with respect to $(\mcG,\mcH)$. 
However, if we consider a subgroup of the form 
\begin{equation*}
\Stab_O(A_\Delta) = \Roar{\AG}{\mcG \cup \set{A_\Delta}, \mcH^t},
\end{equation*}  
the family $\mcG \cup \set{A_\Delta}$ is not necessarily saturated with respect to $(\mcG \cup \set{A_\Delta},\mcH)$ and its image under $R_\Delta$ is more difficult to describe. However, the parabolic subgroups we will consider in \cref{section spherical complex for Out(RAAG)} are exactly of this form. The next two lemmas show that in special cases, we can describe their images under $R_\Delta$ without passing to saturated pairs.

\begin{lemma}
\label{restriction of parabolics}
\label{intersection of parabolics with ROAR kernel}
Assume that $O$ stabilises a special subgroup $A_\Delta \leq A_\Gamma$ and let  $R_\Delta\colon O \to \Out{A_\Delta}$ denote the corresponding restriction homomorphism.
Take $\Theta\subset \Gamma$. Then:
\begin{enumerate}
\item  $\Stab_O(A_{\Theta})\cap \ker R_\Delta = \Stab_{\ker R_\Delta }(A_{\Theta})$.
\item If $\Theta\subseteq \Delta$, one has $R_\Delta (\Stab_O(A_{\Theta}))= \Stab_{\im R_\Delta}(A_{\Theta})$.
\end{enumerate}

\end{lemma}
\begin{proof}
The first point becomes tautological after spelling out the definitions.

For the second point, the inclusion ``$\subseteq$'' is clear.
On the other hand, each $\Phi\in \im R_\Delta$ can by definition be written as $\Phi=[\psi|_{A_\Delta}]$ where $[\psi]\in O$ and  $\psi(A_\Delta)=A_\Delta$. If $\Phi$ stabilises $A_{\Theta}$, we know that $\psi$ conjugates $A_{\Theta}$ to a subgroup of $A_\Delta$. Hence, $[\psi] \in \Stab_O(A_{\Theta})$ and the second claim follows.
\end{proof}

\begin{lemma}
\label{projection of parabolics}
Assume that $\Gamma$ can be decomposed as a join $\Gamma=Z\ast\Delta$ where $Z$ is a complete graph and $A_Z\in \mcG$.
Let $P_\Delta$ denote the projection map $O\to \Outo[A_\Delta]$. Then for every $\Theta\subset \Gamma$ one has
\begin{equation*}
P_\Delta (\Stab_O(A_{\Theta}))= \Stab_{\im P_\Delta}(A_{\Theta\cap \Delta}).
\end{equation*}
\end{lemma}
\begin{proof}
The stabiliser $\Stab_{O}(A_{\Theta})$ is the same as the relative automorphism group $\Roar{A_\Gamma}{\mcG\cup \set{A_\Theta}, \mcH^t}$. By \cref{image of projection map}, the image of this group is equal to
\begin{equation}
\label{equation stabilisers under projection maps}
P_\Delta (\Stab_O(A_{\Theta})) = \Roar{A_\Delta}{\mcG_\Delta\cup \set{A_{\Theta\cap \Delta}}, \mcH_\Delta^t}.
\end{equation}

On the other hand, we have $\im P_\Delta = \Roar{A_\Delta}{\mcG_\Delta, \mcH_\Delta^t}$, so the right hand side of \cref{equation stabilisers under projection maps} is also equal to $\Stab_{\im P_\Delta}(A_{\Theta\cap \Delta})$ and the claim follows.
\end{proof}

\subsection{Restrictions to conical subgroups}
In this section, we define a family of special subgroups that will play an important role in our inductive arguments later on and study some properties of these special subgroups.

For a vertex $v\in V(\Gamma)$, define the following subgraphs of $\Gamma$:
\begin{align*}
\Gamma_{\succeq v} \coloneqq \set{w\in V(\Gamma) \mid v \preceq w} \text{ and }  \Gamma_{\succ v} \coloneqq \set{w\in V(\Gamma) \mid v \prec w},
\end{align*}
where $v \prec w$ if $v \preceq w$ and $w\not\sim_\mcG v$. We define 
\begin{align*}
A_{\succeq v} \coloneqq A_{\Gamma_{\succeq v}} \text{ and } A_{\succ v} \coloneqq A_{\Gamma_{\succ v}}
\end{align*}
as the special subgroups of $\AG$ corresponding to these subgraphs.
Note that these special subgroups only depend on the $\sim_\mcG$-equivalence class of $v$, i.e. if $v\sim_\mcG w$, we have $A_{\succeq v}=A_{\succeq w}$.

In the ``absolute setting'' where $\mcG$ and $\mcH$ are trivial and $\preceq$ is equal to the standard ordering of $V(\Gamma)$, these special subgroups appear as \emph{admissible subgroups} in the work of Duncan--Remeslennikov \cite{DR:Automorphismspartiallycommutative}. 
We will also refer to them as \emph{conical subgroups} of $\AG$ as they are generated by elements corresponding to an upwards-closed cone in the Hasse diagram of the partial order that $\preceq$ induces on the equivalence classes of $\sim_\mcG$ (see \cref{figure conical subgroup}).

\begin{figure}
\begin{center}
\includegraphics[scale=1]{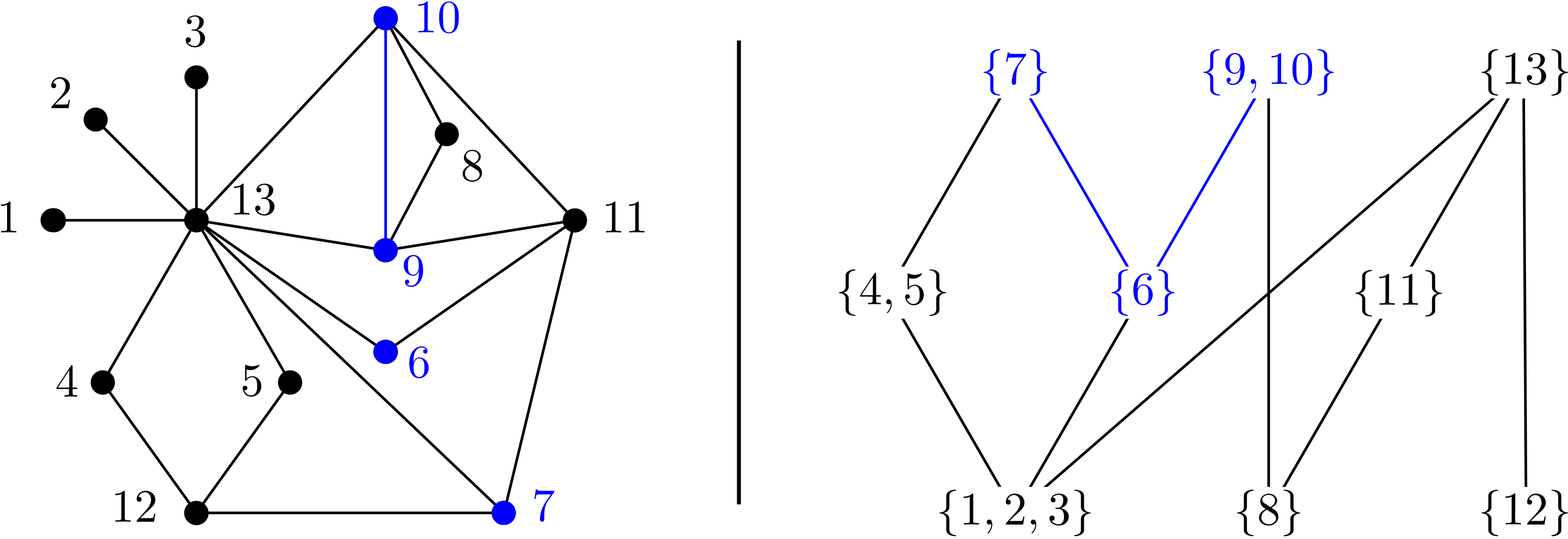}
\end{center}
\caption{A graph $\Gamma$ with vertex set $V(\Gamma)=\set{1, \ldots , 13}$ and the Hasse diagram of the associated partial order $\leq$ of the standard equivalence classes of its vertices. The conical subgroup $\Gamma_{\geq 6}$ is marked in blue.}
\label{figure conical subgroup}
\end{figure}

The elements of $\Outo$ are characterised among all elements of $\Out{\AG}$ by the property that they stabilise these special subgroups. Namely, the following holds true:
\begin{lemma}[{\cite[Proposition 3.3]{DW:Relativeautomorphismgroups}}]
Let $\mcG_\geq$ be the set of special subgroups of $A_\Gamma$ of the form $A_{\geq v}$. Then
\begin{equation*}
\Outo = \Out{\AG;\, \mcG_\geq}.
\end{equation*}
\end{lemma}

In particular, each of these special subgroups is stabilised by all of $\Outo$. We will need a relative version of this statement.

\begin{lemma}
\label{trivial action of partial conjugation}
Let $v,x \in V(\Gamma)$, let $K$ be a union of $\mcG^x$-components of $\Gamma\bsl \st(x)$ and let $\pi_K^x\in O$ denote the corresponding partial conjugation. If $v\npreceq x$, the partial conjugation $\pi_K^x$ acts trivially on $A_{\succeq v}$.
\end{lemma}
\begin{proof}
As $v$ is not smaller than $x$ with respect to $\preceq$, there either is an element in $\lk(v)$ which is not contained in $\st(x)$ or there is $A_\Delta\in\mcG$ such that $\Delta$ contains $v$ but does not contain $x$. We claim that in both cases, $\Gamma_{\succeq v}$ intersects at most one $\mcG^x$-component of $\Gamma\bsl \st(x)$.

Indeed, if there is $y\in \lk(v)\bsl \st(x)$, one has $y\in \st(w)\bsl \st(x)$ for all $w\in \Gamma_{\succeq v}$. Hence, all elements of $\Gamma_{\succeq v}$ are adjacent to $x$ and $\Gamma_{\succeq v}\bsl \st(x)$ is contained in a single $\mcG^x$-component of $\Gamma\bsl \st(x)$.
 If on the other hand for some $A_\Delta \in\mcG$, one has $v\in \Delta$, it follows that $w\in  \Delta$ for all $w\in \Gamma_{\succeq v}$. Now if $x\not\in \Delta$, this implies that all elements of $\Gamma_{\succeq v}$ are $\mcG^x$-adjacent, so they in particular lie in the same $\mcG^x$-component.
Either way, \cref{Laurence generator stabilise criterion} implies that $\pi_K^x$ acts trivially on $A_{\succeq v}$.
\end{proof}

\begin{proposition}
\label{stabilisation of cones}
For every vertex $v \in V(\Gamma)$, the special subgroup $A_{\succeq v}$ is stabilised by every element from $O$.
\end{proposition}
\begin{proof}
As $O$ is generated by the inversions, transvections and partial conjugation it contains, it suffices to prove the statement for each such element. As above, this can be done using \cref{Laurence generator containment criterion} and \cref{Laurence generator stabilise criterion}.

For inversions, there is nothing to show as they always stabilise every special subgroup. 
If we have a transvection $\rho_x^y\in O$, we must have $x\preceq y$. The set $\Gamma_{\succeq v}$ is upwards-closed with respect to $\preceq$, hence $x\in \Gamma_{\succeq v}$ implies $y\in \Gamma_{\succeq v}$. It follows that  $\rho_x^y$ stabilises $A_{\succeq v}$.
Given a partial conjugation $\pi_K^x\in O$, we either have $v\npreceq x$, in which case \cref{trivial action of partial conjugation} implies that $\pi_K^x$ even acts trivially on $A_{\succeq v}$, or we have $x\in \Gamma_{\succeq v}$ which implies that $\pi_K^x$ stabilises $A_{\succeq v}$.
\end{proof}

A consequence of this is that for every equivalence class $[v]_\mcG$ of vertices of $\Gamma$, we have a restriction map 
\begin{equation*}
R_{\succeq v}= R_{A_{\succeq v}}\colon  O \to \Outo[A_{\succeq v}].
\end{equation*} 
These maps are crucial for the line of argument in the following section. We will study some of their properties in \cref{strongly divided for ROARs}.

\section{A spherical complex for $\Out{A_\Gamma}$}
\label{section spherical complex for Out(RAAG)}
In this section, we define maximal parabolic subgroups of $\Outo$ in the general case. We then prove Theorem \ref{introduction homotopy type} which states that the coset complex associated to these parabolic subgroups is homotopy equivalent to a wedge of spheres.

\paragraph{Notation and standing assumptions} As before, let $\Gamma$ be a graph, $\mcG$ and $\mcH$ families of special subgroups of $\AG$ such that $\mcG$ is saturated with respect to $(\mcG, \mcH)$, define $O\coloneqq \Outo[A_\Gamma; \mcG, \mcH^t]$ and set $\preceq\coloneqq \leq_\mcG$ to be the $\mcG$-ordering on $V(\Gamma)$.
Let $T_\mcG$ denote the set of $\sim_\mcG$-equivalence classes of vertices of $\Gamma$. 

\subsection{Rank and maximal parabolic subgroups}

\begin{definition}
We define the \emph{rank of $O$} as
\begin{equation*}
\rk(O)\coloneqq \sum_{[v]_{\mcG}\in T_\mcG}(|[v]_{\mcG}|-1)= |V(\Gamma)|-|T_\mcG|.
\end{equation*} 
\end{definition}

Now fix an ordering $[v]_\mcG=\set{v_1,\ldots , v_n}$ on each equivalence class $[v]_\mcG \in\nolinebreak T_\mcG$. 
For all $[v]_\mcG \in T_\mcG$ and $1\leq j \leq n-1$, let $\Delta_v^j\subset \Gamma$ be the full subgraph of $\Gamma$ with vertex set $\set{v_1, \ldots , v_j} \cup \Gamma_{\succ v}$.
\begin{lemma}
\label{parabolics are proper subgroups}
For all $[v]_\mcG \in T_\mcG$ and $1\leq j \leq n-1$, the stabiliser $\Stab_O(A_{\Delta_v^j})$ is a proper subgroup of $O$.
\end{lemma}
\begin{proof}
Again, we use \cref{Laurence generator containment criterion} and \cref{Laurence generator stabilise criterion}:
As all vertices of $[v]_\mcG$ are equivalent with respect to $\leq_\mcG$, the transvection $\rho_{v_1}^{v_n}$ is an element of $O$. However, 
this transvection does not stabilise $A_{\Delta_v^j}$ because $v_1$ is contained in $\Delta_v^j$ while $v_n$ is not.
\end{proof}

For $[v]_\mcG \in T_\mcG$, let
\begin{equation*}
\mcP_{[v]_\mcG}\coloneqq \set{\Stab_O(A_{\Delta_v^j}) \, \middle| \, 1\leq j \leq n-1},
\end{equation*}
where if $|[v]_\mcG|=1$, this is to be understood as $\mcP_{[v]_\mcG}=\emptyset$. 

\begin{definition} We define the set of \emph{maximal standard parabolic subgroups of $O$} as the union
\label{definition parabolics for RAAGs}
\begin{equation*}
\mcP(O)\coloneqq \bigcup_{[v]_\mcG\in T_\mcG} \mcP_{[v]_\mcG}.
\end{equation*}
\end{definition}

The reader might at this point want to verify that for the graph $\Gamma$ depicted in \cref{figure conical subgroup} on page \pageref{figure conical subgroup}, one has $|\mcP(\Outo)|=4$.
The term ``maximal'' parabolic will become clear in \cref{section CM and higher generation} where we will define and study parabolic subgroups of lower rank. As before, we will usually leave out the adjective ``standard'' (see \cref{remark standard parabolics}). 

\begin{remark}
We note the following properties of $\mcP(O)$ and $\rk(O)$:
\begin{enumerate}
\item $\rk(O)=|\mcP(O)|$. We will also give an alternative interpretation of $\rk(O)$ in \cref{section rank via Coxeter groups}.
\item  By \cref{parabolics are proper subgroups}, every element of $\mcP(O)$ is a proper subgroup of $O$.
\item Following \cref{remark relative orderings and saturation}, the definition of parabolic subgroups depends on the ordering chosen for each equivalence class, but not on the pair $(\mcG, \mcH)$ we chose to represent $O$.
\item If $O$ is equal to  $\GL{n}{\mathbb{Z}}$ or a Fouxe-Rabinovitch group, we recover the definitions of parabolic subgroups in these groups as defined in \cref{section buildings} and \cref{section factor complexes}. Furthermore, $\rk(\GL{n}{\mathbb{Z}})=\rk(\Out{F_n})= n-1$.
\end{enumerate}
\end{remark}

Note that it is possible that there is no $\mcG$-equivalence class of size bigger than one. In this case, the rank of $O$ is zero and $\mcP(O)$ is empty. For further comments on this, see \cref{section closing comments}.

\subsection{The parabolic sieve}
\label{section induction}
In this subsection, we explain the idea of the inductive argument that we will use to show sphericity of the coset complexes $\CC{O}{\mcP(O)}$.

\paragraph{Outline of proof}
Whenever $\Delta\subset \Gamma$ is stabilised by $O$, the restriction map $R_\Delta$ gives rise to a short exact sequence
\begin{equation*}
1\to N\to O \stackrel{R_\Delta}{\to} Q\to 1
\end{equation*}
and by \cref{image and kernel of restriction homomorphism}, both $N$ and $Q$ are relative automorphism groups of RAAGs again. Using the considerations of \cref{section ROARs}, we will show that for the correct choice of $\Delta$, every $P\in \mcP(O)$ satisfies the following dichotomy: Either $R_\Delta(P)$ is contained in $\mcP(Q)$ or $P\cap N$ forms an element of $\mcP(N)$. Applying a restriction homomorphism hence has the effect of a sieve on $\mcP(O)$ --- some of the parabolic subgroups pass through and form parabolics of the quotient $Q$ while others remain in the sieve and form parabolics of the subgroup $N$.
Now using the results of \cref{section coset posets and complexes}, this allows us to describe the homotopy type of $\CC{O}{\mcP(O)}$ in terms of the topology of the lower-dimensional coset complexes $\CC{Q}{\mcP(Q)}$ and $\CC{N}{\mcP(N)}$. This is used for an inductive argument with two phases: We first apply restriction maps to conical subgroups (\cref{section induction conical restrictions}) and then analyse the homotopy type of coset complexes in the conical setting (\cref{section coset complexes of conical ROARs}).
Concrete examples of this induction will be given in \cref{section summary induction and examples}.

\subsubsection{Conical restrictions}
\label{section induction conical restrictions}
\begin{lemma}[Induction step]
\label{induction step}
Let $v\in V(\Gamma)$ and let $R\coloneqq R_{\succeq v}$ denote the corresponding restriction map to $A_{\succeq v}$.
Then there is a homotopy equivalence
\begin{equation*}
\CC{O}{\mcP(O)} \simeq \CC{\im R}{\mcP(\im R)} \ast \CC{\ker R}{\mcP(\ker R)}.
\end{equation*}
\end{lemma}

We want to apply \cref{short exact sequences} to prove this statement. To do so, we have to show that for each $P\in \mcP(O)$, either $\ker R\subseteq P$ or $P$ contains all inversions, transvections and partial conjugations of $O$ that are not contained in $\ker R$. This is the content of the following lemma.

\begin{lemma}
\label{strongly divided for ROARs}
The restriction map $R= R_{\succeq v}$ has the following properties:
\begin{enumerate}
\item For all $\Delta\subseteq \Gamma_{\succeq v}$, one has $\ker R\subseteq \Stab_O(A_\Delta)$.
\item For all $w\in V(\Gamma)$, the following holds: If $\Delta\subseteq \Gamma_{\succeq w}$ such that 
\begin{equation*}
\Gamma_{\succeq v}\cap \Gamma_{\succeq w}\subseteq \Delta,
\end{equation*}
the stabiliser $\Stab_O(A_\Delta)$ contains all inversions, transvections and partial conjugations of $O$ that are not contained in $\ker R$.
\end{enumerate}
\end{lemma}
\begin{proof}
By \cref{image and kernel of restriction homomorphism}, the kernel of $R$ consists of all elements from $O$ that act trivially on the special subgroup $A_{\succeq v}$. This immediately implies the first claim.

For the second one, we again use \cref{Laurence generator containment criterion} and \cref{Laurence generator stabilise criterion}. First note that $\Stab_O(A_\Delta)$ contains all inversions of $O$.

Next assume we have a transvection $\rho_x^y\in O$. If $x\not\in \Gamma_{\succeq v}$, the transvection is contained in $\ker R$.
If on the other hand $x \not\in \Delta$, the transvection $\rho_x^y$ acts trivially on $A_\Delta$ and hence is contained in $\Stab_O(A_\Delta)$.
Now observe that the assumption that $\Gamma_{\succeq v}\cap \Gamma_{\succeq w}\subseteq \Delta$ implies that $\Delta\cap \Gamma_{\succeq v}$ is equal to $\Gamma_{\succeq v}\cap \Gamma_{\succeq w}$, a set which is upwards-closed with respect to $\preceq$. So if $x \in \Delta\cap \Gamma_{\succeq v}$, we also have $y \in \Delta\cap \Gamma_{\succeq v}$. Again it follows that $\rho_x^y\in \Stab_O(A_\Delta)$. 

Lastly, consider a partial conjugation $\pi_K^x\in O$. If $v\npreceq x$, \cref{trivial action of partial conjugation} implies that $\pi_K^x$ is contained in $\ker R$.
This lemma also show that if $w\npreceq x$, the partial conjugation $\pi_K^x$  acts trivially on $A_{\succeq w}$, and hence is contained in $\Stab_O(A_\Delta)$.
The only case that remains is that $x$ is greater than both $v$ and $w$, i.e. $x\in \Gamma_{\succeq v}\cap \Gamma_{\succeq w}$. As we assumed that $\Gamma_{\succeq v}\cap \Gamma_{\succeq w}\subseteq \Delta$, this implies that $x\in \Delta$, so again $\pi_K^x\in \Stab_O(A_\Delta)$.
\end{proof}

\begin{proof}[Proof of \cref{induction step}]
Set $\mcP\coloneqq \mcP(O)$.

Take $[w]_\mcG\in T_\mcG$ and $P=\Stab_O(A_\Delta)\in\mcP_{[w]_\mcG}$ with $\Delta=\Delta_w^j$ as above.
 If $v\preceq w$, we have $\Delta\subset \Gamma_{\succeq v}$. Hence by the first point of \cref{strongly divided for ROARs}, we know that $\ker R\subseteq P$. 
If on the other hand $w\prec v$, one has $\Gamma_{\succeq v}\cap \Gamma_{\succeq w} = \Gamma_{\succeq v} \subset \Delta$. 
Similarly if $v$ and $w$ are incomparable, one has $\Gamma_{\succeq v} \cap \Gamma_{\succeq w}\subseteq  \Gamma_{\succ w} \subset \Delta$. In both cases, the second point of \cref{strongly divided for ROARs} tells us that $P$ contains all inversions, transvections and partial conjugations of $O$ which are not contained in $\ker R$. 
From this, it follows that
\begin{align}
\label{equation P^ker R}
\mcP_{\ker R}=\set{P\in \mcP_{[w]_{\mcG}} \mid v\preceq w}& &\text{ and }	&& \mcP^{\ker R}=\set{P\in \mcP_{[w]_{\mcG}} \mid v\npreceq w},
\end{align}
with notation as defined on page \pageref{paragraph notation H^N H_N}. 
\cref{short exact sequences} now shows that there is a homotopy equivalence
\begin{equation*}
\CC{O}{\mcP}\simeq \CC{\im R}{\overbar{\mcP}} \ast \CC{\ker R}{\mcP\cap \ker R},
\end{equation*}
where $\overbar{\mcP}= \set{R(P) \mid P\in \mcP_{\ker R}}$ and $\mcP\cap{\ker R}=\set{P\cap \ker R \mid P\in \mcP^{\ker R}}$.

If we have $P\in \mcP_{\ker R}$, there is $\Delta\subset \Gamma_{\succeq v}$ such that $P=\Stab_O(A_\Delta)$. Using \cref{restriction of parabolics}, it follows that $R(P)=\Stab_{\im R} (A_\Delta)$. 
\cref{partial order in image and kernel} implies that one has
\begin{equation*}
\overbar{\mcP}=\mcP(\im R).
\end{equation*}
For every $P = \Stab_O(A_\Delta) \in \mcP^{\ker R}$, we know by \cref{restriction of parabolics} that 
\begin{equation*}
P\cap \ker R = \Stab_{\ker R}(A_{\Delta}).
\end{equation*}
Write $\ker R=\Roar{A_\Gamma}{\mcG_{\ker},\mcH_{\ker}^t}$ where $\mcG_{\ker}$ is saturated with respect to $(\mcG_{\ker},\mcH_{\ker})$. Then by \cref{partial order in image and kernel}, for $x,y \in V(\Gamma)$, we have $x\leq_{\mcG_{\ker}} y$ if and only if $v\npreceq x$ and $x\preceq y$. Combining this with \cref{equation P^ker R}, it follows that 
\begin{equation*}
\mcP \cap \ker R = \mcP(\ker R).
\end{equation*}
This finishes the proof.
\end{proof}

For the first phase of our induction, we now use this iteratively in order to obtain:

\begin{proposition}
\label{join of base cases}
There is a homotopy equivalence  
\begin{equation*}
\CC{O}{\mcP(O)} \simeq \ast_{[v]_{\mcG} \in T_\mcG} \CC{O_v}{\mcP(O_v)}
\end{equation*}
where for all $[v]_{\mcG} \in T_\mcG$, one has $O_v = \Roar{A_{\succeq v}}{\mcG_v,\mcH_v^t}$ such that:
\begin{enumerate}
\item	$\mcG_v$ is saturated with respect to $(\mcG_v,\mcH_v^t)$,
\item	$\mcH_v=\mcH_{\Gamma_{\succeq v}} \cup \set{A_{\succeq w} \mid v \prec w}$,
\item	for $x\not= y \in \Gamma_{\succeq v}$, one has $x\leq_{\mcG_v} y$ if and only if $x\in [v]_{\mcG}$ and $x\preceq y$.
\end{enumerate}
\end{proposition}
\begin{proof}
We want to inductively use the restriction maps $R_{\succeq w}$. In order to do this, assume that we have shown that $\CC{\Outo}{\mcP(O)\,}$ is homotopy equivalent to a join of coset complexes of the form
\begin{equation*}
\CC{U}{\mcP(U)},
\end{equation*}
where $U=\Roar{A_\Theta}{\mcE,\mcF^t}$ with $\Theta=\Gamma_{\succeq v}$ and such that the following hypotheses hold:
\begin{enumerate}
\item	$\mcE$ is saturated with respect to $(\mcE,\mcF^t)$,
\item	$\mcF = \mcH_\Theta \cup \mcF'$ where $\mcF' \subseteq \set{A_{\succeq w} \mid v \prec w}$,
\item	for all $w\in \Theta$, either $U$ acts trivially on $A_{\succeq w}$ or $\Theta_{\geq_\mcE w} = \Theta_{\succeq w}$.
\end{enumerate}
Note that a priori, there is slight ambiguity in writing $A_{\succeq w}$ without specifying the ambient graph. Here we can however ignore this issue because for all $w\succ v$, we have $\Gamma_{\succeq w} = \Theta_{\succeq w}$.
For technical reasons we allow $v$ to be a formal element $0$ with $\Gamma_{\succeq 0} \coloneqq \Gamma$.
These three hypotheses hold in particular for the initial case where $v=0$, $\mcE=\mcG$ and $\mcF=\mcH$.

Now assume that there is $w \in \Theta$ such that $U$ does not act trivially on $A_{\succeq w}$. In this case, we have $\Theta_{\geq_\mcE w} = \Theta_{\succeq w}$, so by \cref{stabilisation of cones}, the special subgroup $A_{\succeq w}\leq A_\Theta$ is stabilised by $U$ and we can consider the restriction map $R\colon U \to \Outo[A_{\succeq w}]$.
By \cref{induction step}, this yields a homotopy equivalence
\begin{equation*}
\CC{U}{\mcP(U)} \simeq \CC{\im R}{\mcP(\im R)} \ast \CC{\ker R}{\mcP(\ker R)}.
\end{equation*}

By \cref{image and kernel of restriction homomorphism}, we can write
\begin{equation*}
\ker R = \Roar{A_\Theta}{\mcE_{\ker}, (\mcF\cup \set{A_{\succeq w}})^t},
\end{equation*}
where $\mcE_{\ker}$ is saturated with respect to the pair $(\mcE_{\ker}, \mcF\cup \set{A_{\succeq w}})$. Furthermore, \cref{partial order in image and kernel} together with the third hypothesis of our induction imply that for all $w\in \Theta$, either $\ker R$ acts trivially on $A_{\succeq w}$ or $\Theta_{\geq_{\mcE_{\ker}} w} = \Theta_{\succeq w}$.
It follows that $\CC{\ker R}{\mcP(\ker R)}$ satisfies the hypotheses of our induction.

Again using \cref{image and kernel of restriction homomorphism}, we can write
\begin{equation*}
\im R = \Roar{A_{\succeq w}}{\mcE_{\im}, \mcF_{\im}^t}
\end{equation*}
where $\mcE_{\im}$ is saturated with respect to $(\mcE_{\im},\mcF_{\im})$ and the elements of $\mcF_{\im}$ are the special subgroups generated by the vertices of $\Delta\cap \Gamma_{\succeq w}$ for some $\Delta \in \mcF$.
This implies that $\CC{\im R}{\mcP(\im R)}$ satisfies the first hypotheses of our induction. The third one is an immediate consequence of  \cref{partial order in image and kernel}.

Now apply induction to these coset complexes.
This process ends if we arrive at a case where for all $w \succ v$, the group $U$ acts trivially on $A_{\succeq w}$. But then we can set $\mcF=\mcH_{\Theta} \cup \set{A_{\succeq w} \mid v \prec w}$ and for $x\not= y$, $x \leq_\mcE y$ is only possible if $x\in[v]_\mcG$. If $v=0$, this means that the relative ordering on $\Gamma_{\preceq 0}= \Gamma$ is trivial, so $\mcP(U)=\emptyset$. If $v\not =0$, the group $O_v\coloneqq U$ satisfies all conditions of the claim.
\end{proof}

\subsubsection{Coset complexes of conical RAAGs}
\label{section coset complexes of conical ROARs}
We now want to deal with the coset complexes $\CC{O_v}{\mcP(O_v)}$ of conical RAAGs that we obtained in \cref{join of base cases}. This is why in this subsection, we impose the following assumptions.
\paragraph{Standing assumptions}
Until the end of \cref{section coset complexes of conical ROARs}, we assume that:
\begin{enumerate}
\item  \label{item conical subgroup} There is a vertex $v\in V(\Gamma)$ such that $\Gamma= \Gamma_{\succeq v}$, i.e. every vertex of $\Gamma$ is greater than or equal to $v$ with respect to $\preceq$.
\item \label{item trivial action} For all $w\succ v$, the group $O$ acts trivially on the special subgroup $A_{\set{w}}\leq \nolinebreak A_\Gamma$.
\end{enumerate}
Observe that \cref{item conical subgroup} implies that 
\begin{align}
\label{equation only bigger things stabilised}
\text{for all } A_\Delta\in \mcG \text{, we have  } \Delta\subseteq \Gamma_{\succ v}:
\end{align}
By \cref{Laurence generator containment criterion} and \cref{Laurence generator stabilise criterion},
every $\Delta\subseteq \Gamma$ such that $O$ stabilises $A_\Delta$ must be upwards-closed with respect to $\preceq$. Hence, if $\Delta$ intersects $[v]_\mcG$ non-trivially, it follows that $\Delta=\Gamma$.
Furthermore, \cref{item trivial action} implies that for all $w\succ v$, the equivalence class $[w]_\mcG$ is a singleton. Hence, we have 
\begin{align}
\label{equation only one set of parabolics}
\mcP(O)=\mcP_{[v]_\mcG}.
\end{align}
\newline

In this situation, let
\begin{equation*}
Z \coloneqq \set{w\in V(\Gamma) \mid v\prec w \text{ and $w$ is adjacent to $v$} }.
\end{equation*} 
We define the \emph{group of twists by elements in $\Gamma_{\succ v}$} as the subgroup $T \leq O$ generated by the transvections $\rho_x^z$ with $x\in [v]_\mcG$ and $z\in Z$.

\begin{lemma}
\label{SES for projection map}
$T$ is a free abelian group. Furthermore, $\Gamma$ can be decomposed as a join $\Gamma= \nolinebreak Z\ast\Delta$ and there is a short exact sequence
\begin{equation*}
1 \to T \to O \stackrel{P_\Delta}{\to} \Roar{A_\Delta}{\mcG_\Delta, \mcH_\Delta^t} \to 1.
\end{equation*}
\end{lemma}
\begin{proof}
If $Z = \emptyset$, the statement is trivial, so we can assume that $Z$ contains at least one element. 
By definition, we have $Z \subseteq \lk(v) \bsl [v]_\mcG$.
 As every vertex of $\Gamma$ is greater than or equal to $v$ with respect to $\preceq$, this implies that $Z$ is a complete graph and we can write $\Gamma= \nolinebreak Z\ast\Delta$.

Using the assumption that $O$ acts trivially on $A_{\set{w}}$ for all $w\succ v$, \cref{Laurence generator containment criterion} and \cref{Laurence generator stabilise criterion} imply that $O$ acts trivially on the normal subgroup $A_Z \triangleleft \nolinebreak \AG$.
Consequently, we have a well-defined projection map $P_\Delta\colon O \to \Out{A_\Delta}$. By \cref{image of projection map} the image of this map is equal to $\Roar{A_\Delta}{\mcG_\Delta, \mcH_\Delta^t}$.

The description of the kernel $\ker P_\Delta$ as the free abelian group $T$ generated by the transvection $\rho_x^z$ with $x\in [v]_\mcG$ and $z\in Z$ follows from \cite[Proposition 4.4]{CV:Finitenesspropertiesautomorphism} because $O$ acts trivially on $A_Z$ (see also \cite[5.1.4]{DW:Relativeautomorphismgroups}). 
\end{proof}

\begin{lemma}
\label{parabolics after projection map}
Let $\Delta \coloneqq \Gamma \bsl Z$ and let $P_\Delta$ denote the corresponding projection map. There is a a homotopy equivalence
\begin{equation*}
\CC{O}{\mcP(O)} \simeq \CC{\im P_\Delta}{\mcP(\im P_\Delta)}.
\end{equation*}
\end{lemma}
\begin{proof}
By \cref{SES for projection map}, we have a short exact sequence
\begin{equation*}
1 \to T \to O \stackrel{P_\Delta}{\to} \im P_\Delta \to 1,
\end{equation*}
where $\im P_\Delta = \Roar{A_\Delta}{\mcG_\Delta, \mcH_\Delta^t}$.
We first claim that every parabolic subgroup $P\in \mcP(O)$ contains $T$. Indeed, we observed above that $\mcP(O)=\mcP_{[v]_\mcG}$ (see \cref{equation only one set of parabolics}). By definition, every $P\in \mcP_{[v]_\mcG}$ is of the form $P=\Stab_O(A_\Theta)$ for some $\Theta$ containing $\Gamma_{\succ v}$. The claim now follows from \cref{Laurence generator stabilise criterion}.

Hence, \cref{short exact sequences} yields a homotopy equivalence
\begin{equation*}
\CC{O}{\mcP(O)} \simeq \CC{\im P_\Delta}{\overbar{\mcP(O)}} \ast \emptyset = \CC{\im P_\Delta}{\overbar{\mcP(O)}}
\end{equation*}
The ordering $\leq_{\mcG_\Delta}$ is just the restriction of $\preceq$ to $\Delta$, so \cref{projection of parabolics} implies that $\overbar{\mcP(O)} = \mcP(\im P_\Delta)$.
\end{proof}

We now distinguish between the case where $[v]_\mcG$ is an abelian and the case where it is a free equivalence class.
\begin{lemma}
\label{coset complex of abelian equivalence class}
Assume that $\Gamma= \Gamma_{\succeq v}$ where $[v]_{\mcG}$ is an abelian equivalence class of size $n\coloneqq |[v]_{\mcG}| \geq 2$.
Then the coset complex
\begin{equation*}
\CC{O}{\mcP(O)}
\end{equation*}
is homotopy equivalent to the Tits building associated to $\GL{n}{\mathbb{Q}}$.
\end{lemma}
\begin{proof}
By \cref{parabolics after projection map}, we have a homotopy equivalence
\begin{equation*}
\CC{O}{\mcP(O)} \simeq \CC{\im P_\Delta}{\mcP(\im P_\Delta)},
\end{equation*}
where $\Delta = \Gamma \bsl Z$ and $\im P_\Delta = \Roar{A_\Delta}{\mcG_\Delta, \mcH_\Delta^t}$.

By assumption, the abelian equivalence class $[v]_{\mcG}$ contains at least two elements which are adjacent to each other. As every vertex of $\Gamma$ is greater than or equal to $v$ with respect to $\preceq$, this implies that every vertex of $\Gamma_{\succ v}$ must be adjacent to $v$. Hence, $Z=\Gamma_{\succ v}$ and $\Delta= [v]_\mcG$.
As observed above (\cref{equation only bigger things stabilised}), every $\Theta\subseteq \Gamma$ with $A_\Theta \in \mcG$ is entirely contained in $\Gamma_{\succ v}$. Consequently, we have $\mcG_\Delta = \mcH_\Delta = \emptyset$ and
\begin{equation*}
\Roar{A_\Delta}{\mcG_\Delta, \mcH_\Delta^t} = \GL{n}{\mathbb{Z}}.
\end{equation*}
This means that $\CC{O}{\mcP(O)}\simeq \CC{\GL{n}{\mathbb{Z}}}{\mcP(\GL{n}{\mathbb{Z}})}$ and this coset complex is isomorphic to the Tits building associated to $\GL{n}{\mathbb{Q}}$ by \cref{building and coset complex}.
\end{proof}

In the setting of a free equivalence class, the situation is slightly more complicated: As before, we start by projecting away from $Z$, but we then might have to apply further restriction maps.
\begin{lemma}
\label{coset complex of free equivalence class}
Assume that $\Gamma= \Gamma_{\succeq v}$ where $[v]_{\mcG}$ is a free equivalence class of size $n\coloneqq |[v]_{\mcG}| \geq 2$.
Then there is a special subgroup $A\leq A_\Gamma$ such that 
\begin{equation*}
A = F \ast A_1 \ast \cdots \ast A_k
\end{equation*}
where $F$ is the free group of rank $n$ generated by $[v]_{\mcG}$ and the coset complex
\begin{equation*}
\CC{O}{\mcP(O)}
\end{equation*}
is homotopy equivalent to the free factor complex of $A$ relative to $\set{[A_1],\ldots, [A_k]}$.
\end{lemma}

\begin{proof}
Again by \cref{parabolics after projection map}, we have a homotopy equivalence
\begin{equation*}
\CC{O}{\mcP(O)} \simeq \CC{\im P_\Delta}{\mcP(\im P_\Delta)},
\end{equation*}
where $\Delta = \Gamma \bsl Z$ and $\im P_\Delta = \Roar{A_\Delta}{\mcG_\Delta, \mcH_\Delta^t}$. As noted above, 
the $\mcG_\Delta$-ordering on $\Delta$ is just the restriction of $\preceq$ to $\Delta$; in particular we have $[v]_{\mcG_\Delta}= [v]_\mcG$ and $\Delta=\Delta_{\succeq v}$.

As no two vertices from $[v]_{\mcG}$ are adjacent to each other, the link $\lk_\Gamma(v)$ is entirely contained in $Z$, so every element of $[v]_{\mcG}$ forms an isolated vertex of $\Delta$.
This implies that $\Delta$ decomposes as a disjoint union $\Delta=[v]_\mcG \sqcup \bigsqcup \Delta_i$ where each $\Delta_i$ is a $\mcG_\Delta$-component of $\Delta$. In particular, we have
\begin{equation*}
A_\Delta = A_{[v]_\mcG} \ast A_{\Delta_1} \ast \ldots \ast A_{\Delta_k}.
\end{equation*}
Moreover, for all $i$, the group $\im P_\Delta$ stabilises $A_{\Delta_i}$: If $\Delta_i$ contains at least two vertices, this is \cite[Lemma 3.13.1]{DW:Relativeautomorphismgroups} and if $\Delta_i$ is a singleton, the action on $A_{\Delta_i}$ is trivial by assumption.

If there is an $i$ such that $\im P_\Delta$ acts non-trivially on $\Delta_i$, there is a non-trivial restriction map $R\colon  \im P_\Delta \to \Out{A_{\Delta_i}}$. Its kernel can be written as 
\begin{equation*}
\ker R=\Roar{A_\Delta}{\mcG_{\ker}, (\mcH_\Delta\cup\set{A_{\Delta_i}})^t},
\end{equation*}
where $\mcG_{\ker}$ is saturated with respect to $(\mcG_{\ker}, \mcH_\Delta\cup\set{A_{\Delta_i}})$. 
One can easily check that each $P\in \mcP(\im P_\Delta)$ contains all the inversions, transvections and partial conjugations not contained in $\ker R$: The kernel contains all inversions and transvections from $\im P_\Delta$ as well as the partial conjugations that have acting letter contained in $[v]_\mcG$. The remaining partial conjugations are contained in all of the parabolic subgroups.

Hence, by \cref{short exact sequences}, we have a homotopy equivalence
\begin{equation*}
\CC{\im P_\Delta}{\mcP(\im P_\Delta)} \simeq \emptyset \ast \CC{\ker R}{\mcP\cap \ker R}= \CC{\ker R}{\mcP\cap \ker R}.
\end{equation*}
\cref{partial order in image and kernel} implies that the ordering $\leq_{\mcG_{\ker}}$ agrees with $\preceq$ on $\Delta$; hence using \cref{restriction of parabolics}, we obtain $\mcP(\ker R) = \mcP\cap \ker R$. 
All the $A_{\Delta_i}$ are stabilised by $\ker R$, so we can use induction and apply restriction maps until we reach the group $\Roar{A_\Delta}{\set{A_{\Delta_i}}_i^t}$. This group is equal to $\Out{A_\Delta, \set{A_{\Delta_i}}_i^t}$ and hence a Fouxe-Rabinovitch group. The claim now follows from \cref{isomorphism relative free factor complex}.
\end{proof}

Using the results of \cref{section building and relative free factor complex}, the last two lemmas can be summarised as:
\begin{corollary}
\label{homotopy type of base cases}
Assume that $\Gamma= \Gamma_{\succeq v}$ where $[v]_{\mcG}$ is an equivalence class of size $n\coloneqq |[v]_{\mcG}|$.
Then $\CC{O}{\mcP(O)}$ is $(n-2)$-spherical.
\end{corollary}
\begin{proof}
If $n=1$, the statement is trivial as in this case, the set $\mcP(O)=\mcP_{[v]_{\mcG}}$ is empty. Hence, the complex $\CC{O}{\mcP(O)}$ is the empty set which we consider to be $(-1)$-spherical (see \cref{section spherical complexes}). Now let $n\geq 2$.
If $[v]_{\mcG}$ is abelian, \cref{coset complex of abelian equivalence class} implies that the coset complex is homotopy equivalent to the Tits building associated to $\GL{n}{\mathbb{Q}}$ which is $(n-2)$-spherical by the Solomon--Tits Theorem.
If on the other hand $[v]_{\mcG}$ is free, it is by \cref{coset complex of free equivalence class} homotopy equivalent to a relative free factor complex which is by \cref{relative free factor complexes} $(n-2)$-spherical as well.
\end{proof}

\subsection{Proof of Theorem A}
We return to the general situation where $\Gamma$ is any graph and $\mcG$ and $\mcH$ are any families of special subgroups of $\AG$ such that $\mcG$ is saturated with respect to $(\mcG, \mcH)$. Recall that $\preceq$ denotes the $\mcG$-ordering of $V(\Gamma)$ and $T_\mcG$ denotes the set of associated $\sim_\mcG$-equivalence classes.

The only thing that is left to be done for the proof of Theorem \ref{introduction homotopy type}, which we restate below, is to collect the results obtained in \cref{section induction}.
\begin{theorem}
\label{homotopy type complex of parabolics in ROARs}
Let $O\coloneqq \Outo[A_\Gamma; \mcG, \mcH^t]$. The coset complex $\CC{O}{\mcP(O)}$ is homotopy equivalent to a wedge of spheres of dimension $\rk(O)-1$.
\end{theorem}
\begin{proof} 
By \cref{join of base cases}, we know that there is a homotopy equivalence  
\begin{equation*}
\CC{O}{\mcP(O)} \simeq \ast_{[v]_{\mcG} \in T_\mcG} \CC{O_v}{\mcP(O_v)}
\end{equation*}
where for all $[v]_{\mcG} \in T_\mcG$, one has $O_v = \Roar{A_{\succeq v}}{\mcG_v,\mcH_v^t}$ such that:
\begin{enumerate}
\item	$\mcG_v$ is saturated with respect to $(\mcG_v,\mcH_v^t)$,
\item	$\mcH_v=\mcH_{\Gamma_{\succeq v}} \cup \set{A_{\succeq w} \mid v \prec w}$,
\item	for $x \not= y \in \Gamma_{\succeq v}$, one has $x\leq_{\mcG_v} y$ if and only if $x\in [v]_{\mcG}$ and $x\preceq y$.
\end{enumerate}

Let $[v]_{\mcG} \in T_\mcG$. Condition 3 implies that the $\mcG_v$-equivalence class of $v$ in $\Gamma_{\succeq v}$ is equal to $[v]_{\mcG}$ and that all other $w\in \Gamma_{\succeq v}$ are greater than $v$ with respect to $\leq_{\mcG_v}$. Now Condition 2 implies that for all $w$ with $w >_{\mcG_v} v$, the group $O_v$ acts trivially on $A_{\set{w}} \leq A_{\succeq v}$. Hence, the assumptions of \cref{section coset complexes of conical ROARs} are fulfilled and \cref{homotopy type of base cases} implies that $\CC{O_v}{\mcP(O_v)}$ is spherical of dimension $|[v]_{\mcG}|-2$. It follows from \cref{homotopy type of wedges and joins} that the join of these complexes is spherical of dimension $\sum_{[v]_{\mcG}\in T_\mcG}(|[v]_{\mcG}|-1) -1=\rk(O)-1$.
\end{proof}

\section{Summary of the inductive procedure and examples}
\label{section summary induction and examples}
\subsection{Consequences for the induction of Day--Wade}
\label{section summary inductive procedure}

The proof of \cref{homotopy type complex of parabolics in ROARs} relies on the inductive procedure defined in \cite{DW:Relativeautomorphismgroups}: The authors there show that for every graph $\Gamma$, the group $\Outo$ has a subnormal series 
\begin{equation*}
1= N_0 \leq N_1\leq \ldots \leq N_k = \Outo
\end{equation*}
such that for all $i$, the quotient $N_{i+1}/N_{i}$ is isomorphic to either a free abelian group, to $\GL{n}{\mathbb{Z}}$ or to a Fouxe-Rabinovitch group (see \cite[Theorem A]{DW:Relativeautomorphismgroups}). The methods we use in \cref{section induction} provide more detailed information about this inductive procedure which decomposes $\Outo$ in terms of short exact sequences related to restriction and projection homomorphisms: We are able to give an explicit description of the restriction and projection maps that one has to use during the induction and of the base cases one obtains this way. In what follows, we will give a summary of these results. See also \cref{figure decomposition tree}.

\begin{figure}[t]
\begin{center}
\includegraphics{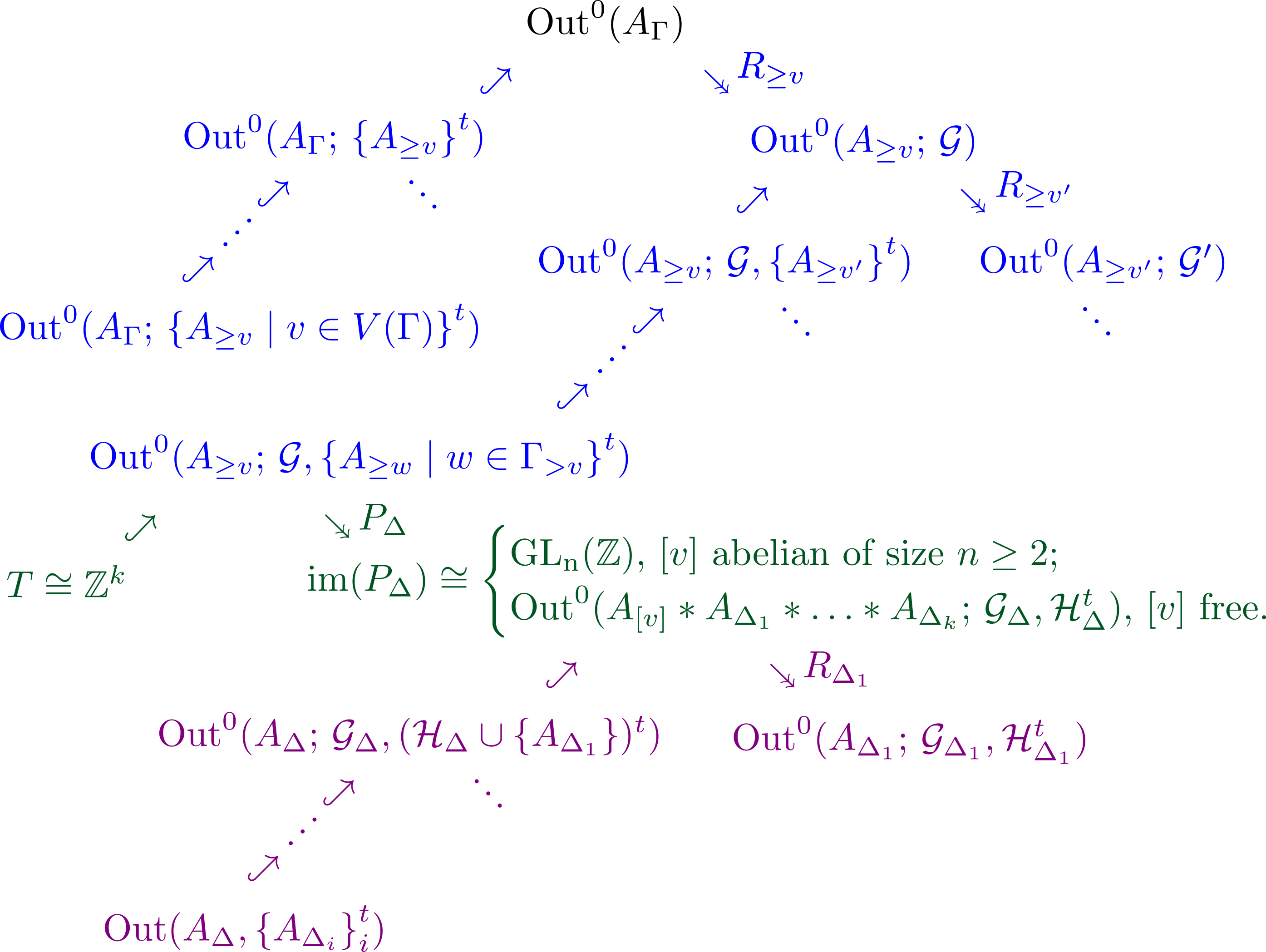}
\end{center}
\caption{Decomposition of $\Outo$. Step 1 is coloured in blue, Step 2 in green and Step 3 in magenta.}
\label{figure decomposition tree}
\end{figure}

To simplify notation, we will describe the decomposition of $O=\Outo$, however all of this can also be stated in the more general case where $O$ is any relative automorphism group of a RAAG.

\paragraph{Step 1}
First one iteratively restricts to conical subgroups $A_{\geq v}$ until one is left with relative automorphism groups that act trivially on all of their proper conical subgroups --- for this, one needs to apply exactly one restriction map for every (standard) equivalence class of $V(\Gamma)$ and the order in which one applies the corresponding restriction maps does not change the base cases of this first induction step. One of these base cases is given by the intersection of the kernels of all the conical restriction maps; it is the group $\Roar{\AG}{\set{A_{\geq v} \mid v \in V(\Gamma)}^t}$ which does not contain any inversions or transvections. The other base cases are all of the form $\Roar{A_{\geq v}}{\mcG, \set{A_{\geq w} \mid w \in \Gamma_{> v}}^t}$ for some $v\in V(\Gamma)$ and some family $\mcG$ of special subgroups of $A_{\geq v}$. There is exactly one such base case for every equivalence class $[v]$ of $V(\Gamma)$ and it is generated by all the restrictions to $A_{\geq v}$ of inversions, transvections and partial conjugations of $\Outo$ that act trivially on $A_{\geq w}$ for every $w>v$.

\paragraph{Step 2}
Now for each of these groups, one applies the (possibly trivial) projection map $P_\Delta$ where $\Delta\coloneqq \Gamma_{\geq v}\bsl Z$ and $Z$ is the full subgraph of $\Gamma_{\geq v}$ consisting of all those vertices of $\Gamma$ which are adjacent to $v$ and strictly greater than $v$ with respect to the standard ordering on $V(\Gamma)$. The kernel of this projection map is given by the free abelian group $T$ generated by all twist of elements in $[v]$ by elements in $\Gamma_{> v}$.
We now have to distinguish two cases: If $[v]$ is an abelian equivalence class of size $n\geq 2$, then the image of $P_\Delta$ is given by $\Out{A_{[v]}} \cong \GL{n}{\mathbb{Z}}$.
If this is not the case, we proceed with Step 3.

\paragraph{Step 3} If $[v]$ is a free equivalence class, the graph $\Delta$ decomposes as a disjoint union $\Delta=[v] \sqcup \bigsqcup \Delta_i$ where each $\Delta_i$ is a relative connected component of $\im(P_\Delta)$.
One can show that the $\Delta_i$ are precisely the non-empty intersections $\Delta_i \nolinebreak= (\Delta\bsl [v] ) \cap \Gamma_i$ where $\Gamma_i$ is a connected component of $\Gamma\bsl \lk(v)$.
We now iteratively apply the restriction maps $R_{\Delta_i}$. This yields two kinds of base cases: The first one is given by the intersection of the kernels of all the $R_{\Delta_i}$ and can be described as the Fouxe-Rabinovitch group $\Out{A_\Delta, \set{A_{\Delta_i}}_i^t}$. The second one is given by the images of the restriction maps. For each $i$, this is a relative automorphism group of $A_{\Delta_i}$; as $\Delta_i\subseteq \Gamma_{>v}$, this group contains no inversions or transvections and is generated by partial conjugations. 

\paragraph{The base cases} In summary, our induction yields the following base cases:
\begin{enumerate}
	\item \label{item base case left-most kernel} The ``left-most'' kernel $\Roar{\AG}{\set{A_{\geq v} \mid v \in V(\Gamma)}^t}$;
	\item for every abelian equivalence class $[v]$ of size $n\geq 2$:
		\begin{enumerate}
			\item the free abelian group $T$ generated by all twist of elements in $[v]$ by elements in $\Gamma_{> v}$, which has rank $n\cdot |\Gamma_{> v}\cap \lk(v)|$;
			\item \label{item base case GLn} $\Out{A_{[v]}} \cong \GL{n}{\mathbb{Z}}$;
		\end{enumerate}
	\item for every free equivalence class $[v]$:
		\begin{enumerate}
			\item the free abelian group $T$ generated by all twist of elements in $[v]$ by elements in $\Gamma_{> v}$, which has rank $n\cdot |\Gamma_{> v}\cap \lk(v)|$;
			\item \label{item base case relative connected components} for every connected component $\Gamma_i$ of $\Gamma\bsl \lk(v)$ such that 
				\begin{equation*}
				\Delta_i \coloneqq (\Gamma_{> v}\bsl \lk(v) ) \cap \Gamma_i\not = \emptyset:
				\end{equation*}			
			a subgroup of $\Outo[A_{\Delta_i}]$ generated by partial conjugations;
			\item \label{item FR with transvections} a Fouxe-Rabinovitch group $\Out{A_\Delta, \set{A_{\Delta_i}}_i^t}$ where $\Delta=\Gamma_{\geq v}\bsl \lk(v)$ and the $\Delta_i$ are as in \cref{item base case relative connected components}.
		\end{enumerate}
\end{enumerate}

The base cases having a non-empty set of parabolic subgroups are \cref{item base case GLn} and \cref{item FR with transvections} if $|[v]|\geq 2$. Note that we allow $|v|=1$ in \cref{item FR with transvections} which results in $\Out{\ll v\rr}\cong \mathbb{Z}/2\mathbb{Z}$ if $v$ is maximal.
One should mention that \cref{item base case left-most kernel} and \cref{item base case relative connected components} are not necessarily base cases of the induction of Day--Wade: There might be further non-trivial restriction and projection maps and after applying them one can decompose these groups into Fouxe-Rabinovitch groups and free abelian groups generated by partial conjugations.

\subsection{Examples}

\subsubsection{String of diamonds}
Let $\Gamma$ be the string of $d$ diamonds (see \cref{figure string of diamonds}), as considered in \cite[Section 5.3]{CSV:Outerspaceuntwisted} and \cite[Section 6.3.1]{DW:Relativeautomorphismgroups}.
 Assume $d\geq 2$. The standard equivalence classes of $\Gamma$ are given by 
\begin{align*}
[a_i]=[b_i]=\set{a_i, b_i}, \, 1\leq i \leq d &&\text{ and }&& [c_i]=\set{c_i}, \, 0\leq i \leq d.
\end{align*}
The conical subgroups here are
\begin{align*}
&\Gamma_{\geq a_i}=[a_i]=\set{a_i, b_i}, \, 1\leq i \leq d,\\
&\Gamma_{\geq c_i} =[c_i], \, 1\leq i \leq d-1 ,\\
&\Gamma_{\geq c_0}=[c_0]\cup [c_1] \cup [a_1] \text{ and } \Gamma_{\geq c_d}=[c_d]\cup [c_{d-1}] \cup [a_d].
\end{align*}
If we order $[a_i]$ as $(a_i,b_i)$, the family of maximal parabolic subgroups of the group $O\coloneqq \Outo$ is given as
\begin{align*}
\mcP(O)=\set{\Stab_O(\ll a_i \rr) \mid 1\leq i \leq d}
\end{align*}
and we have $\rk(O)=d$, i.e. $\CC{O}{\mcP(O)}$ is $(d-1)$-spherical.

\begin{figure}
\begin{center}
\includegraphics{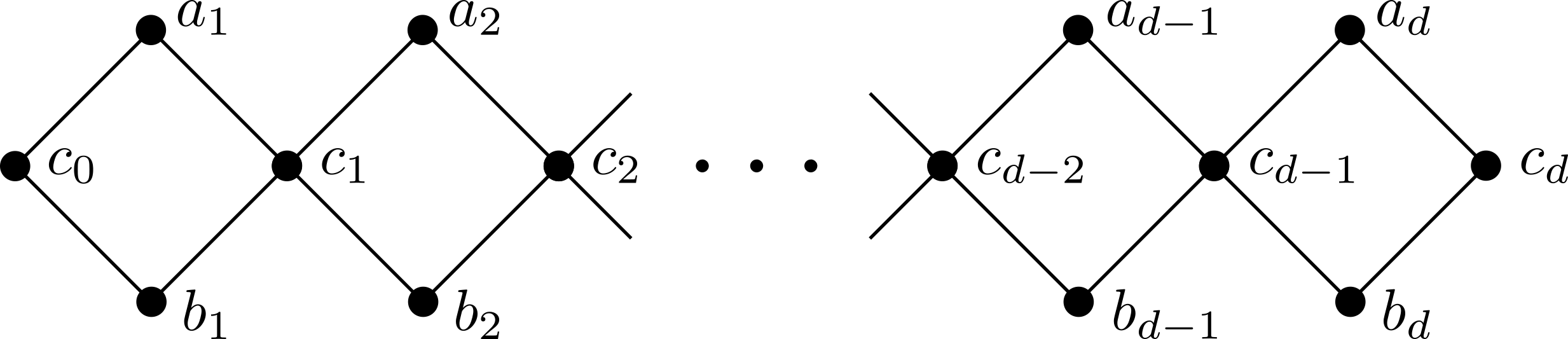}
\end{center}
\caption{String of $d$ diamonds.}
\label{figure string of diamonds}
\end{figure}

After restricting to these conical subgroups (Step 1 of our induction), we are left with the following base cases:
\begin{enumerate}
\item The left-most kernel $\Roar{\AG}{\set{A_{\geq v} \mid v \in V(\Gamma)}^t}$ which here is generated by all the partial conjugations of $O$;
\item for all $1\leq i \leq d-1$: the group $\Out{\ll c_i \rr}\cong \mathbb{Z}/2\mathbb{Z}$;
\item$ \begin{aligned}[t]
	\operatorname{Out}^0(\ll c_0,c_1,a_1,b_1\rr; \, \set{\ll c_1\rr , \ll a_1, b_1\rr}&^t) \\
	=&\Roar{\ll c_0,c_1,a_1,b_1\rr}{\set{\ll c_1, a_1, b_1\rr}^t};
	\end{aligned}$
\item$ \begin{aligned}[t]
	\operatorname{Out}^0(\ll c_{d},c_{d-1},a_d,b_d\rr; \, \{\ll c_{d-1} \rr, &\ll  a_d, b_d\rr\}^t) \\
	=&\Roar{\ll c_{d},c_{d-1},a_d,b_d\rr}{\set{\ll c_{d-1}, a_d, b_d\rr}^t};
	\end{aligned}$
\item for all $1\leq i \leq d$: the group $\Out{\ll a_i,b_i \rr}\cong \Out{F_2}$.
\end{enumerate}
Only the groups of the last item have a non-empty set of parabolic subgroups (each given by the singleton $\set{\Stab_{\Out{\ll a_i,b_i \rr}}(\ll a_i \rr)}$). All items but the first one describe Fouxe-Rabinovitch groups, so the induction already ends here and we do not have to apply Step 2 and Step 3.

The following direct argument gives a more explicit description of the coset complex: For all $i$, we have a surjective restriction map $O \to \Out{\ll a_i, b_i \rr}$. These can be amalgamated to a map $R: O \to \oplus_{i=1}^n \Out{\ll a_i, b_i \rr}$. We have $\ker(R)\subseteq P_i \coloneqq \Stab_O(\ll a_i \rr)$ for all $i$, so by \cref{short exact sequences}, 
\begin{equation*}
\label{eq:explicit_diamonds}
\CC{O}{\mcP(O)} \cong \CC{\oplus_{i} \Out{\ll a_i, b_i \rr}}{\overbar{\mcP(O)}} \cong \ast_{i} \CC{\Out{\ll a_i, b_i \rr}}{\Stab(\ll a_i \rr)\,}.
\end{equation*}
(For the second isomorphism, we used that $\overbar{P_i}$ contains $\Out{\ll a_j, b_j \rr}$ for $j\not=i$.) 
Each factor in this join is a copy of the free factor complex associated to $\Out{F_2}$ and $O$ acts on their join in the obvious way.
 
\subsubsection{Trees}
Let $\Gamma$ be a tree, define $O\coloneqq\Outo$ and, to simplify notation, assume that $|V(\Gamma)|\geq 3$. 
Let $L$ denote the set of leafs of $\Gamma$. For each leaf $l$, let $z_l$ denote the (unique) vertex adjacent to $l$ and let $Z=\set{z_l \mid l\in L}$ be the set of vertices of $\Gamma$ that are adjacent to some leaf. Then we have
\begin{align*}
[v]=\begin{cases}
	\set{v} &,\, v\in V(\Gamma)\bsl L;\\
	\lk(z_v)\cap L &,\, v\in L.
	\end{cases}
\end{align*}
The conical subgroups are given by
\begin{align*}
\Gamma_{\geq v}=\set{v}, \, v\in V(\Gamma)\bsl L && \text{ and } && \Gamma_{\geq l}= \st(z_l),\, l\in L .
\end{align*}
Now for each $z\in Z$, let $\set{v^z_1,\ldots, v^z_{k_z}}=\lk(z)\bsl L$ be the non-leaf vertices adjacent to $z$ and $\set{l^z_1,\ldots, l^z_{n_z}}=\lk(z)\cap L$ the leafs adjacent to $z$ (see \cref{figure star in tree}). 
Then 
\begin{align*}
\mcP(O)=\bigcup_{z \in Z} \set{\Stab_O\left(\ll l^z_1,\ldots, l^z_i, v^z_1,\ldots, v^z_{k_z}, z \rr \right) \, \middle| \, 1\leq i \leq n_z-1}
\end{align*}
and $\rk(O)= \sum_{z \in Z} (n_z-1) = |L|-|Z|$, which implies that $\CC{O}{\mcP(O)}$ is $(|L|-|Z|-1)$-spherical.

\begin{figure}
\begin{center}
\includegraphics{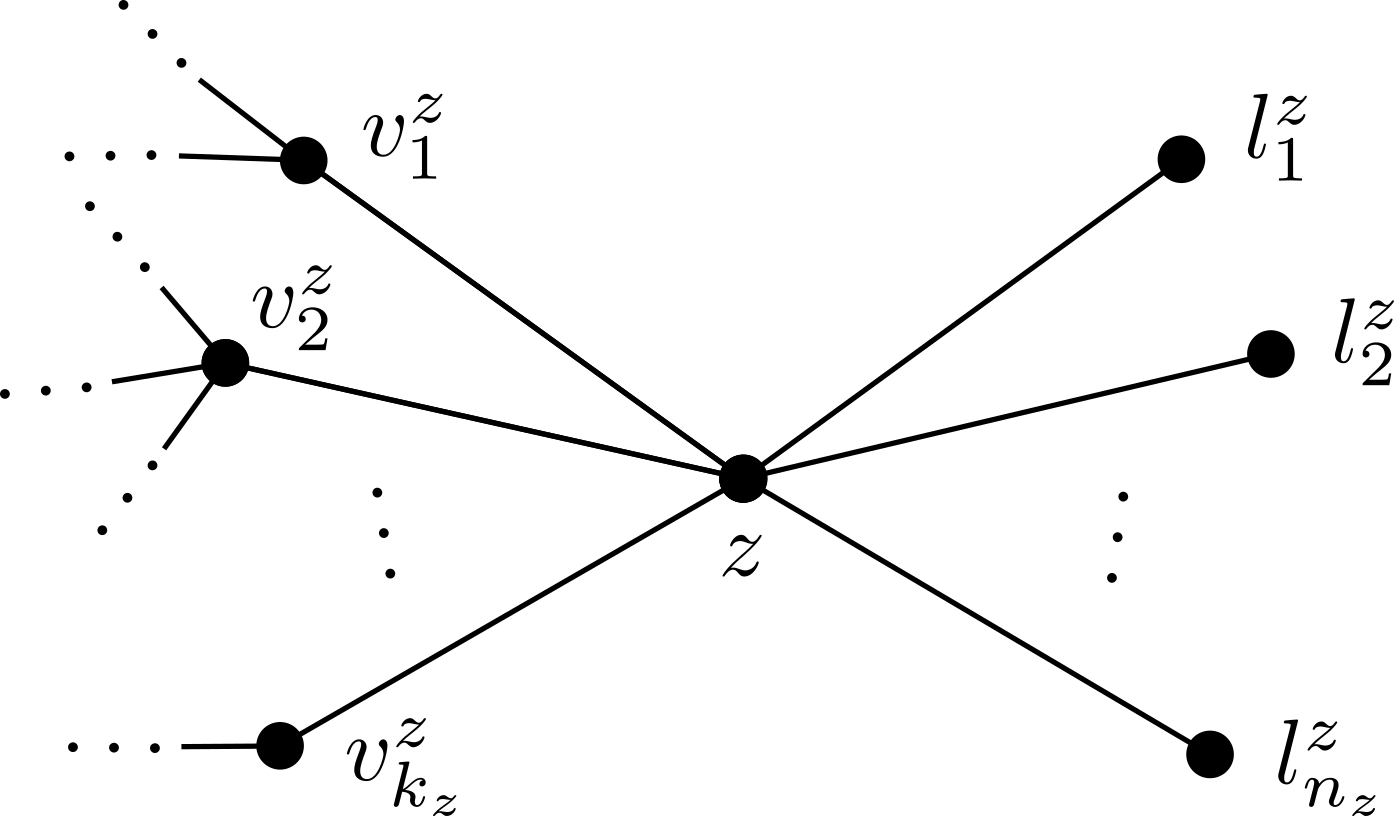}
\end{center}
\caption{$\st(z)=\Gamma_{\geq l_i^z}$.}
\label{figure star in tree}
\end{figure}

Step 1 of our induction leads to the following base cases:
\begin{enumerate}
\item The left-most kernel $\Roar{\AG}{\set{A_{\geq v} \mid v \in V(\Gamma)}^t}$ generated by all the partial conjugations of $O$ acting trivially on $\st(z)$ for all $z\in Z$;
\item for all $v \in V(\Gamma) \bsl L$: the group $\Out{\ll v \rr}\cong \mathbb{Z}/2\mathbb{Z}$;
\item for all $z \in Z$: the group $\Roar{A_{\st(z)}}{\set{\ll v^z_1\rr, \ldots, \ll v^z_{k_z} \rr, \ll z \rr}^t}$.
\end{enumerate}
In Step 2, we apply for each group of the last item the projection map $P_\Delta$ where $\Delta= \lk(z)$. Its kernel is a free abelian group of rank $n_z$, generated by the twists of the leafs adjacent to $z$. The image of this projection map is the Fouxe-Rabinovitch group
\begin{align*}
\Roar{A_{\lk(z)}}{\set{\ll v^z_1\rr, \ldots, \ll v^z_{k_z} \rr}^t}.
\end{align*}
It contains $n_z-1$ maximal parabolic subgroups, given by the stabilisers of $\ll l^z_1,\ldots, l^z_i, v^z_1,\ldots, v^z_{k_z}\rr$, $1\leq i \leq n_z-1$. Again, we do not have to apply Step 3.

\subsubsection{Constructions}
\label{section constructions}
This section does not really contain examples but rather shows how to obtain new examples from known ones.

\paragraph{Direct product}
Let $\Gamma_1$ and $\Gamma_2$ be graphs and let $\Gamma\coloneqq \Gamma_1 \ast \Gamma_2$ be their join. On the level of RAAGs, this means that $A_{\Gamma}=A_{\Gamma_1}\times A_{\Gamma_2}$. Let $\leq , \leq_1\nolinebreak, \leq_2$ be the standard orderings and $[\cdot], [\cdot]_1, [\cdot]_2$ the corresponding equivalence classes of $\Gamma, \Gamma_1, \Gamma_2$, respectively. Let $Z,Z_1, Z_2$ be the (possibly trivial) subgraphs of $\Gamma, \Gamma_1, \Gamma_2$ consisting of all vertices that are adjacent to every vertex of the corresponding graph; clearly, $Z=Z_1\ast Z_2$. It is easy to see that for $v_i\in \Gamma_i$, one has
\begin{align*}
[v_i]=	\begin{cases}
			Z &, v_i\in Z_i;\\
			[v_i]_i &, v_i\not\in Z_i;
		\end{cases}
&& \text{ and } &&
\Gamma_{\geq v_i} =	(\Gamma_i)_{\geq_i v_i}\cup  Z_j .
\end{align*}
We do not spell out the consequences for all of the induction, but would like to point out the following implication for the ranks of the corresponding automorphism groups and hence the dimensions of the associated coset complexes:
\begin{align*}
\rk(\Outo)=\begin{cases}
\rk(\Outo[A_{\Gamma_1}]) + \rk(\Outo[A_{\Gamma_2}]) &, Z_i=\emptyset \text{ for some } i; \\
\rk(\Outo[A_{\Gamma_1}]) + \rk(\Outo[A_{\Gamma_2}]) + 1 &\text{, otherwise}.
\end{cases}
\end{align*}
Note that $Z_i=\emptyset$ is equivalent to saying that the center $Z(A_{\Gamma_i})$ is trivial.

A particularly simple instance of this is the situation where $\AG = F_n \times F_m$ with $n,m > 1$. Then, $\Outo = \Out{F_n} \times \Out{F_m}$  and the coset complex $\CC{\Outo}{\mcP(\Outo)\,}$ is isomorphic to the the join of the free factor complexes associated to $\Out{F_n}$ and $\Out{F_m}$.

\paragraph{Free product}
Let $\Gamma\coloneqq \Gamma_1 \sqcup \Gamma_2$ be the disjoint union of the graphs $\Gamma_1$ and $\Gamma_2$, i.e. $A_{\Gamma}=A_{\Gamma_1}\ast A_{\Gamma_2}$, and keep the notation of the prior paragraph otherwise. Let $D,D_1, D_2$ be the (possibly trivial) subgraphs of $\Gamma, \Gamma_1, \Gamma_2$ consisting of all their isolated vertices; we have $D=D_1 \sqcup D_2$. For $v_i\in \Gamma_i$, one has:
\begin{align*}
[v_i]=	\begin{cases}
			D &, v_i\in D_i;\\
			[v_i]_i &, v_i\not\in D_i;
		\end{cases}
&& \text{ and } &&
\Gamma_{\geq v_i} =	\begin{cases}
			\Gamma &,v_i\in D_i;\\
			(\Gamma_i)_{\geq_i v_i} &, v_i\not\in D_i.
		\end{cases}
\end{align*}
Similar to the case of direct products, this implies:
\begin{align*}
\rk(\Outo)=\begin{cases}
\rk(\Outo[A_{\Gamma_1}]) + \rk(\Outo[A_{\Gamma_2}]) &, D_i=\emptyset \text{ for some } i; \\
\rk(\Outo[A_{\Gamma_1}]) + \rk(\Outo[A_{\Gamma_2}]) + 1 &\text{, otherwise}.
\end{cases}
\end{align*}
This allows for example to generalise the example of tree-RAAGs given above to the setting of forests.

\paragraph{Complement graph}
For a graph $\Gamma$, let $\Gamma^c$ denote its complement, i.e. the graph with vertex set $V(\Gamma)$ where $v$ and $w$ form an edge if and only if they do not form an edge in $\Gamma$. Let $\leq_c$ and $[\cdot]_c$ denote the standard ordering and its equivalence classes on $\Gamma^c$. Then it is easy to see that 
\begin{align*}
[v]=[v]_c && \text{ and } && v\leq_c w \Leftrightarrow w \leq v.
\end{align*}
In particular, one has $\rk(\Outo[A_{\Gamma^c}]) = \rk(\Outo)$. This also explains the analogy between the settings of direct and free products considered above.

\section{Cohen--Macaulayness, higher generation \& rank}
In this section, we generalise the results of \cref{section spherical complex for Out(RAAG)}: We show that the coset complex of parabolic subgroups of a relative automorphism group $O$ of a RAAG is not only spherical, but even Cohen--Macaulay. This is used to determine the degree of generation that families of (possibly non-maximal) parabolic subgroups provide. We also give an interpretation of the rank in terms of a ``Weyl group'' of $O$.

\paragraph{Notation and standing assumptions} As before, let $\Gamma$ be a graph, $\mcG$ and $\mcH$ families of special subgroups of $\AG$ such that $\mcG$ is saturated with respect to $(\mcG, \mcH)$, define $O\coloneqq \Outo[A_\Gamma; \mcG, \mcH^t]$ and set $\preceq\coloneqq \leq_\mcG$ to be the $\mcG$-ordering on $V(\Gamma)$.
Let $T_\mcG$ denote the set of $\sim_\mcG$-equivalence classes of vertices of $\Gamma$. 

\label{section CM and higher generation}
\subsection{Cohen--Macaulayness}
For coset complexes, the Cohen--Macaulay property can be characterised as follows.
\begin{theorem}[{\cite[Theorem 2.11]{Bru:Highergeneratingsubgroups}}]
\label{characterisation CC CM}
Let $G$ be a group and \mcH\, be a finite family of subgroups of $G$. Then $\CC{G}{\mcH}$ is homotopy Cohen--Macaulay if and only if every $\mcH'\subseteq \mcH$ is $(|\mcH'|-1)$-generating for $G$.
\end{theorem}

This allows us to generalise \cref{homotopy type complex of parabolics in ROARs} in the following way:

\begin{theorem}
\label{CM of coset complex parabolic subgroups}
Let $O\coloneqq \Outo[A_\Gamma; \mcG, \mcH^t]$. The coset complex $\CC{O}{\mcP(O)}$ is homotopy Cohen--Macaulay.
\end{theorem}
\begin{proof}
By \cref{characterisation CC CM}, it suffices to show that for all $\mcP'\subseteq \mcP$, the coset complex $\CC{O}{\mcP'}$ is $(|\mcP'|-1)$-spherical. This can be done following the induction of \cref{section induction}: We first iteratively apply restriction maps to conical subgroups as in \cref{section induction conical restrictions}. In each step, the parabolic subgroups in $\mcP'$ satisfy a dichotomy that allows us to apply \cref{short exact sequences}. We get an analogue of \cref{join of base cases}: The coset complex $\CC{O}{\mcP'}$ is homotopy equivalent to the join $\ast_{[v]_{\mcG} \in T_\mcG} \CC{O_v}{\mcP_v}$ where $O_v$ is exactly as in \cref{join of base cases} and $\mcP_v\subseteq \mcP(O_v)$. There is a one-to-one correspondence between the parabolic subgroups occurring in the join and the elements of $\mcP'$; in particular, $\sum_{[v]_{\mcG} \in T_\mcG} |\mcP_v|= |\mcP'|$. One now follows the arguments of \cref{section coset complexes of conical ROARs} to show that if $[v]_{\mcG}$ is an abelian equivalence class of size $n$, we have $\CC{O_v}{\mcP_v}\simeq \CC{\GL{n}{\mathbb{Z}}}{\mcQ}$ with $\mcQ\subseteq \mcP(\GL{n}{\mathbb{Z}})$ and that if $[v]_{\mcG}$ is a free equivalence class, we have 
\begin{equation*}
\CC{O_v}{\mcP_v}\simeq \CC{\Out{A;\mcA^t}}{\mcQ}
\end{equation*}
where $A=F_n\ast A_1\ast \ldots \ast A_k$, $\mcA=\set{A_1,\ldots, A_k}$ and $\mcQ\subseteq \mcP(\Out{A;\mcA^t})$. Both $\CC{\GL{n}{\mathbb{Z}}}{\mcP (\GL{n}{\mathbb{Z}})\,}$ and $\CC{\Out{A;\mcA^t}}{\mcP (\Out{A;\mcA^t})\,}$ are homotopy Cohen--Macaulay: In the first case, this holds because the coset complex is isomorphic to the building associated to $\GL{n}{\mathbb{Q}}$ (see \cref{building and coset complex}), in the second case this is \cref{CM free factor complex}.
Hence, \cref{characterisation CC CM} implies that $\CC{O_v}{\mcP_v}$ is spherical of dimension $|\mcP_v|-1$. It now follows from \cref{homotopy type of wedges and joins} that $\CC{O}{\mcP'}$ is spherical of dimension $(\sum_{[v]_{\mcG} \in T_\mcG} |\mcP_v|) -1 = |\mcP'|-1$.
\end{proof}

An immediate consequence of this is that $\CC{O}{\mcP(O)}$ is a \emph{chamber complex}, i.e. that each pair $\sigma,\tau$ of facets of $\CC{O}{\mcP(O)}$ can be connected by a sequence of facets $\sigma=\nolinebreak\tau_1,\ldots, \tau_k=\tau$ such that for all $1\leq i \leq k$, the intersection of $\tau_i$ and $\tau_{i+1}$ is a face of codimension $1$ (see \cite[Proposition 11.7]{Bjo:Topologicalmethods} and \cite[Remark 2.8]{Bru:Highergeneratingsubgroups}). 

\subsection{Parabolic subgroups of lower rank}
\label{section parabolics of lower rank}
\begin{definition} 
\label{definition lower rank parabolics}
Let $r\coloneqq \rk(O)$ and $1\leq m \leq r-1$. We define the family of \emph{rank-$m$ standard parabolic subgroups of $O$} as the set of all intersections of $(r-m)$ distinct maximal standard parabolic subgroups,
\begin{equation*}
\mcP_{m}(O) \coloneqq \set{P_1\cap \ldots \cap P_{r-m} \mid P_1, \ldots , P_{r-m} \text{ distinct elements of } \mcP(O)}.
\end{equation*}
In particular, we have $\mcP(O)=\mcP_{r-1}(O)$. 
\end{definition}

Every parabolic subgroup of $O$ is itself a relative automorphism group of $\AG$. The term ``rank-$m$'' parabolic subgroup is justified by the following:

\begin{proposition}
\label{rank of parabolic subgroups}
For all $P\in \mcP_{m}(O)$, we have $\rk(P)=m$.
\end{proposition}
\begin{proof}
For every $P\in  \mcP_{m}(O)$, there is a $V\subset V(\Gamma)$ and for every $v\in V$ a subset $J_v\subset \set{1,\ldots, |[v]_\mcG|}$ such that
\begin{align*}
P= \Roar{\AG}{\mcG', \mcH^t}, &\text{ where } \mcG'= \mcG \cup \set{A_{\Delta_v^j} \, \middle| \, v\in V, \, j\in J_v}, \\ 
& \Delta_v^j=\set{v_1, \ldots , v_j} \cup \Gamma_{\succ v} \text{ and } \sum_{v\in V} |J_v|=r-m.
\end{align*}
 As $\mcG$ contains $P(\mcH)$, so does $\mcG'$. It is easy to check that if $v\in V$, the $\mcG$-equivalence class $[v]_\mcG$ can be written as the disjoint union of $(|J_v|+1)$-many $\mcG'$-equivalence classes and that otherwise, one has $[v]_\mcG=[v]_{\mcG'}$. From this, the claim follows immediately.
\end{proof}

Theorem \ref{introduction higher generation by parabolic subgroups} is now an easy corollary of Cohen--Macaulayness of $\CC{O}{\mcP(O)}$ and the results of \cite{Bru:Highergeneratingsubgroups}:
\begin{corollary}
\label{higher generation by parabolic subgroups}
The family $\mcP_m(O)$ of rank-$m$ parabolic subgroups of $O$ is $m$-generating, the corresponding coset complex $\CC{O}{\mcP_m(O)}$ is $m$-spherical.
\end{corollary}
\begin{proof}
As the coset complex is homotopy Cohen--Macaulay by \cref{CM of coset complex parabolic subgroups}, this is an immediate consequence of \cite[Theorem 2.9]{Bru:Highergeneratingsubgroups}.
\end{proof}

In the case where $O=\GL{n}{\mathbb{Z}}$, this is \cite[Theorem 3.3]{AH:Highergenerationsubgroups}; for this case, 2-generation was also already shown by Tits in \cite[Section 13]{Tit:Buildingssphericaltype}.
Observe that although $\CC{O}{\mcP_m(O)}$ is $m$-spherical, it is a priori a complex of dimension 
\begin{equation*}
|\mcP_m(O)|-1=\binom{r}{m}-1.
\end{equation*}

\paragraph{Presentations for $O$}
A consequence of higher generation is that one can obtain presentations of $O$ from presentations of the parabolic subgroups as follows:
Write $\mcP_m(O)=\set{P_1,\ldots, P_{\binom{r}{m}}}$. For each $i$, let $L_i$ be the set of all inversions, transvections and partial conjugations of $O$ that are contained in $P_i$. By \cref{generators ROARs}, the set $L_i$ generates $P_i$. Let $P_i = \grep{L_i}{R_i}$ be a presentation for $P_i$. Then we have:
\begin{corollary}
\label{presentation by parabolic subgroups}
Let $1\leq m \leq r-1$ and $k\coloneqq \binom{r}{m}$. Then:
\begin{enumerate}
\item The union $\bigcup_{i=1}^{k} L_i$ is a generating set for $O$.
\item If $m\geq 2$, a presentation for $O$ is given by $O= \grep{\bigcup_{i=1}^{k} L_i}{\bigcup_{i=1}^{k} R_i}$.
\end{enumerate}
\end{corollary}
\begin{proof}
This follows from \cref{higher generation by parabolic subgroups} and \cref{higher generating subgroups by generation}.
\end{proof}

\subsection{Interpretation of rank in terms of Coxeter groups}
\label{section rank via Coxeter groups}
The rank of a group with a $BN$-pair is given by the rank of the associated Weyl group $W$, which is a Coxeter group. This is also true in the setting of relative automorphism groups of RAAGs as we will see in what follows.

\begin{definition}
Let $\Aut{\Gamma}$ denote the group of graph automorphisms of $\Gamma$. This group embeds in $\Out{\AG}$
and we define $\AutO(\Gamma)$ as the intersection $\Aut{\Gamma}\cap O$. 
\end{definition}

If $O$ is equal to $\Out{F_n}$ or $\GL{n}{\mathbb{Z}}$, we have $\AutO(\Gamma)=\Aut{\Gamma}=\Sym(n)$, the Weyl group associated to $\GL{n}{\mathbb{Q}}$, which has rank $n-1$. In general, $\AutO(\Gamma)$ can be seen as the group of ``algebraic'' graph automorphisms of $\Gamma$. It appears as ``$\Sym^0(\Gamma)$'' in \cite[Section 3.2]{CCV:Automorphisms2dimensional} where it is studied under the additional assumption that $\Gamma$ be connected and triangle-free.

\begin{lemma}
\label{algebraic graph automorphisms}
The group $\AutO(\Gamma)$ is naturally isomorphic to the direct product 
\begin{equation*}
\AutO(\Gamma) \cong \bigoplus_{[v]_\mcG \in T_{\mcG}}\Sym([v]_\mcG).
\end{equation*}
\end{lemma}
\begin{proof}
If $|[v]_\mcG|>1$, the group $O$ contains for all $x,y\in [v]_\mcG$ the transvection $\rho_x^y$ and the inversion $\iota_y$. It follows that the full group $\Sym([v]_\mcG)$ of permutations of $[v]_\mcG$ is contained in $\AutO(\Gamma)$, so the direct product $\bigoplus_{[v]_\mcG \in T_{\mcG}}\Sym([v]_\mcG)$ is a subgroup of $\AutO(\Gamma)$.

It remains to show that this group does not contain any other elements, i.e. that every element of $\AutO(\Gamma)$ preserves all the $\mcG$-equivalence classes of $V(\Gamma)$. To see this, assume that $\phi\in \Aut{\AG}$ represents an element of $O$ such that $\phi(v)=v'$ for some $v,v'\in V(\Gamma)$. We will show that $v\sim_{\mcG} v'$:

For $x\in V(\Gamma)$ and a word $w$ in the alphabet $V(\Gamma)^{\pm 1}$, let $\#_x(w)\in \mathbb{Z}$ denote the number of occurrences of $x$ in $w$, counted with sign. For $g\in \AG$, let $\#_x(g) \in \nolinebreak \mathbb{Z}/2\mathbb{Z}$ be the image of $\#_x(w)$ in $\mathbb{Z}/2\mathbb{Z}$, where $w$ is a word representing $g$ --- this number only depends on $g$ and not on the chosen representative $w$. 
Now assume that $v\not= v'$. Then clearly, $\#_{v'}(v)=0$ and $\#_{v'}(\phi(v))=\#_{v'}(v')=1$. Writing $\phi$ as a product of inversions, transvections and partial conjugations, it follows that there must be such a Laurence generator $[\psi]\in O$ with $\#_{v'}(\psi(v))=1$.
This is only possible if $\psi$ is given by the transvection $\rho_{v}^{v'}$. However, if this is contained in $O$, we know that $v\preceq v'$.
As $\phi^{-1}$ sends $v'$ to $v$, we also have $v'\preceq v$, hence $v\sim_{\mcG} v'$.
\end{proof}

Recall that the \emph{rank of a Coxeter system $(W,S)$} is given by $\rk(W,S)=|S|$.
\begin{corollary}
\label{rank via Coxeter system}
There is a subset $S \subset \AutO(\Gamma)\leq O$ such that $(\AutO(\Gamma), S)$ is a Coxeter system of rank equal to $\rk(O)$.
\end{corollary}
\begin{proof}
The symmetric group on a set of $n$ elements is the Coxeter group of type $A_{n-1}$, so the claim follows from \cref{algebraic graph automorphisms}.
\end{proof}

Additional comments on this can be found in the ``BN-pairs'' paragraph of \cref{section closing comments}.

\section{Closing comments and open questions}
\label{section closing comments}

We conclude with comments on the limitations of our constructions and on open questions related to the complex $\RAAGcomplex=\CC{O}{\mcP(O)}$.

\paragraph{Description as a subgroup poset in $\boldsymbol{\AG}$}
Both in the setting of $\GL{n}{\mathbb{Z}}$ and of Fouxe-Rabinovitch groups $\Out{A; \mcA^t}$, we studied the coset complex of parabolic subgroups by finding an isomorphic poset of subgroups of $A_\Gamma$ and then determined its homotopy type. These were the poset of direct summands of $\mathbb{Z}^n$ and the relative free factor complex $\FRfactor(A,\mcA)$, respectively. In general, however, the author is not aware of a natural description of $\RAAGcomplex$ which looks similar.

It is not hard to see that if $P=\Stab_O(A_\Delta)$ and $P'=\Stab_O(A_{\Delta'})$ are distinct parabolic subgroups, then the $O$-orbits of $[A_\Delta]$ and $[A_{\Delta'}]$ intersect trivially. Hence, the map $g\Stab_O(A_\Delta)\mapsto [g(A_\Delta)]$ defines a bijection between the vertices of $\RAAGcomplex$ and the union of the $O$-orbits of conjugacy classes of special subgroups $A_{\Delta_v^j}$ that we used for the definition of parabolic subgroups.
However, it is not true that the adjacency relation in $\RAAGcomplex$ is given by containment of corresponding subgroups of $\AG$.

These orbits can be described more explicitly: Let $\Delta_v^j\subset \Gamma$ be the full subgraph of $\Gamma$ with vertex set $\set{v_1, \ldots , v_j} \cup \Gamma_{\succ v}$. As $A_{\succeq v}$ is stabilised by $O$, it is clear that $[g(A_{\Delta_v^j})]\leq [A_{\succeq v}]$. If $[v]$ is abelian, $\Gamma_{\succeq v}\subseteq \st([v])$. This implies that the $O$-orbit of $[A_{\Delta_v^j}]$ is given by
\begin{equation*}
\set{[g(A_{\set{v_1, \ldots, v_j}})\times A_{\succ v}] \,\middle|\, g \in \Out{A_{[v]}}\cong \GL{|[v]|}{\mathbb{Z}}}.
\end{equation*}
If on the other hand $[v]$ is a free equivalence class, one has $\Gamma_{\succeq v} = \Delta \ast Z$, where $Z\coloneqq \lk(v)\cap \Gamma_{\succeq v}$ is a complete graph and $\Delta = [v] \sqcup \Gamma_1 \sqcup \ldots \sqcup \Gamma_k$ with $\Gamma_1, \ldots , \Gamma_k$ the connected components of $\Gamma_{\succ v}\bsl Z$. 
Hence, $\AG$ decomposes as a direct product $A_{\Delta} \times A_Z$ and every element of $O$ preserves this product structure. It follows that the $O$-orbit of $[A_{\Delta_v^j}]$ is equal to
\begin{equation*}
\set{[g(A_{\set{v_1, \ldots , v_j}} \ast A_{\Gamma_1} \ast \ldots \ast A_{\Gamma_k}) \times A_Z] \,\middle|\, g\in O\cap \operatorname{Out}(A_\Delta)}
\end{equation*}
Using \cite[Lemma 2.11]{HM:Relativefreesplitting} (see \cref{sec:rel_free_factor}), every element in this orbit is of the form $F \ast A_{\Gamma_1}^{a_1} \ast \ldots \ast A_{\Gamma_k}^{a_k}$, where $a_j\in A_\Delta$ and $F$ is a free group of rank $j$.

\paragraph{Limitations of our construction}
It seems that our definition of parabolic subgroups and the corresponding coset complex capture well the aspects of $\Out{\AG}$ that come from similarities of this group with $\GL{n}{\mathbb{Z}}$ and $\Out{F_n}$: Firstly, our definitions recover the Tits building as $\CC{\GL{n}{\mathbb{Z}}}{\mcP(\GL{n}{\mathbb{Z}})}$ and the free factor complex as $\CC{\Out{F_n}}{\mcP(\Out{F_n})}$. Secondly, the results we obtain show strong similarities in behaviour between the general situation of $\Out{\AG}$ and these special cases: The associated coset complex is spherical, even Cohen--Macaulay (\cref{CM of coset complex parabolic subgroups}) and families of parabolic subgroups are highly generating with the degree of generation depending on the rank of these subgroups (\cref{higher generation by parabolic subgroups}). Another strong indication which suggests a certain optimality of our definitions is the description of $\rk(\Outo)$ in terms of a Coxeter subgroup (\cref{rank via Coxeter system}).
Furthermore, our induction leads to well-suited families of parabolic subgroups in all those ``components'' of $\Outo$ that closely resemble  general linear groups and automorphism groups of free groups; i.e.~the base cases that are given by $\GL{n}{\mathbb{Z}}$, $n\geq 2$, and Fouxe-Rabinovitch groups containing transvections (\cref{item base case GLn} and \cref{item FR with transvections} in \cref{section summary inductive procedure}).

However, our construction is rather transvection-based in the sense that the standard ordering of $V(\Gamma)$ --- which is used to define the parabolic subgroups --- is entirely determined by the transvections that $\Out{\AG}$
contains. This makes our definition of parabolic subgroups quite \emph{local}: Whether or not $v\leq w$ can be read off from the one-balls around these vertices. 
This is also reflected by the fact that the conical subgraphs $\Gamma_{\geq v}$, which play a central role in our induction, are contained in the two-ball around $v$ if $v$ is not an isolated vertex.
In contrast, certain aspects of $\Out{\AG}$ seem not to be mere generalisations of phenomena in arithmetic groups and automorphism groups of free groups. For example, $\Out{\AG}$ contains partial conjugations which cannot be written as a product of transvections. The existence of these partial conjugations is a \emph{global} phenomenon in the sense that the shape of the connected components of $\Gamma\bsl \st(v)$ is not determined by local conditions on $\Gamma$.
These aspects are not very well represented in $\RAAGcomplex$: The base cases of our induction that correspond to them do not contain any parabolic subgroups. In the extremal case where there is no equivalence class of $V(\Gamma)$ that has size greater than one, $\mcP(\Outo)$ is even empty.
For specific applications, one might try to overcome this by introducing further parabolic subgroups that capture these global aspects. However, the author currently does not see a canonical way to do this.

\paragraph{BN-pairs} The existence of a ``Weyl group'' $\AutO(\Gamma)$ as described in \cref{section rank via Coxeter groups} suggests that one might be able to transfer additional notions from the theory of groups with BN-pair to automorphism groups of RAAGs. It does for instance seem reasonable to define a ``Borel-subgroup'' by taking the intersection of all standard parabolic subgroups or to use the Weyl group to define apartments in $\RAAGcomplex$. For this, it might be helpful to use the \emph{standard representation} $\Out{\AG}\to \GL{|V(\Gamma)|}{\mathbb{Z}}$ induced by the abelianisation. The question that has yet to be clarified is to what extent this point of view might be fruitful for studying automorphism groups of RAAGs; one should keep in mind that all this can also be done for $O=\Out{F_n}$ which is far away from having a BN-pair.

\paragraph{Boundary structures} 
Both buildings and free factor complexes can be seen as boundary structures of classifying spaces --- in the first case, this is due to Borel--Serre who constructed a bordification of symmetric spaces whose boundary can be described by rational Tits buildings \cite{BS:Cornersarithmeticgroups}; in the second case, it was shown in \cite{BG:Homotopytypecomplex} that the free factor complex can be seen as a subspace of the simplicial boundary of Culler--Vogtmann Outer space. In the RAAG-setting, one may ask whether a similar statement holds and $\RAAGcomplex$ can be seen as a boundary structure of the RAAG Outer space defined in \cite{BCV:OuterspaceRAAGs}, \cite{CSV:Outerspaceuntwisted} or a similar space. However, without further changes, this will not work for arbitrary $O$. In particular, if $O$ does not contain any transvection, the complex $\RAAGcomplex$ is trivial, while this need not be the case for the RAAG Outer space and its boundary. This for example occurs for RAAGs defined by \emph{focused graphs} that appear in the work of Bregman--Fullarton \cite{BF:Infinitegroupsacting}, if the standard ordering on the graph $\Gamma$ is trivial. In that case, $\Outo$ is a semi-direct product of a free abelian group generated by partial conjugations and the (finite) group of inversions.

\paragraph{Geometric aspects}
This text focuses on the topology of $\RAAGcomplex$. It also seems very reasonable, however, to ask what can be said about the geometry of this complex.
Motivated by the work of Masur and Minsky \cite{MM:Geometrycomplexcurves.a}, who showed that the curve complex $\mcC(S)$ is hyperbolic, Bestvina and Feighn \cite{BF:Hyperbolicitycomplexfree} proved that the free factor complex is hyperbolic as well. This is only one of many results in the study of $\Out{F_n}$ from a geometric point of view, which has become popular in recent years.
On the other hand, there is also a rich theory concerning metric aspects of buildings (for an overview, see \cite[Section 12]{AB:Buildings}). Combining these two theories should be an interesting topic for further investigations.

\bibliographystyle{halpha}
\bibliography{mybibliography}
\bigskip
  \footnotesize
  Benjamin Br\"uck\\ \textsc{Department of Mathematics \\
ETH Z\"urich\\
Rämistrasse 101\\
8092 Z\"urich, Switzerland}\\
\texttt{ https://people.math.ethz.ch/\textasciitilde bbrueck} \\ \nopagebreak
  \texttt{benjamin.brueck@math.ethz.ch}
\end{document}